\documentclass[10pt,reqno]{amsart}
\allowdisplaybreaks[4]
\RequirePackage[OT1]{fontenc}
\RequirePackage{amsthm,amsmath}
\RequirePackage[numbers]{natbib}
\RequirePackage[colorlinks,linkcolor=blue,citecolor=blue,urlcolor=blue]{hyperref}
\usepackage{amssymb}
\usepackage{anysize}
\usepackage{hyperref}
\usepackage{extarrows}
\usepackage{multirow}
\usepackage{natbib}
\usepackage{stmaryrd}
\usepackage{xcolor}
\usepackage{amsmath}
\usepackage{enumitem}

\numberwithin{equation}{section}
\theoremstyle{plain} 
\newtheorem{thm}{Theorem}[section]
\newtheorem{lem}[thm]{Lemma}
\newtheorem{cor}[thm]{Corollary}
\newtheorem{pro}[thm]{Proposition}
\newtheorem{assumption}[thm]{Assumption}
\newtheorem{defn}[thm]{Definition}
\theoremstyle{remark}
\newtheorem{rem}[thm]{Remark}

\renewcommand{\Re}{\mathrm{Re}\,}
\renewcommand{\Im}{\mathrm{Im}\,}
\newcommand{\re}{\mathrm{Re}\,}
\newcommand{\im}{\mathrm{Im}\,}

\newcommand{\E}{{\mathbb E }}
\newcommand{\R}{{\mathbb R }}
\newcommand{\N}{{\mathbb N}}
\newcommand{\Z}{{\mathbb Z}}
\renewcommand{\P}{{\mathbb P}}
\newcommand{\C}{{\mathbb C}}

\newcommand{\ii}{\mathrm{i}}
\newcommand{\deq}{\mathrel{\mathop:}=}
\newcommand{\e}[1]{\mathrm{e}^{#1}}
\newcommand{\ntr}{\mathrm{tr}\,}
\newcommand{\dd}{\mathrm{d}}
\newcommand{\ie}{\emph{i.e., }}
\newcommand{\eg}{\emph{e.g., }}
\newcommand{\cf}{\emph{c.f., }}

\newcommand{\wt}{\widetilde}
\newcommand{\bs}{\boldsymbol}
\newcommand{\la}{\langle}
\newcommand{\ra}{\rangle}
\renewcommand{\mathbf}[1]{\bs{#1}}

\marginsize{1in}{1in}{1in}{0.94in}

\begin{document}

 \begin{minipage}{0.85\textwidth}
 \vspace{3.5cm}
 \end{minipage}
\begin{center}
\large\bf
Spectral rigidity  for addition of random matrices  at the regular edge
\end{center}

\renewcommand{\thefootnote}{\fnsymbol{footnote}}	
\vspace{1cm}
\begin{center}
 \begin{minipage}{0.32\textwidth}
\begin{center}
Zhigang Bao\footnotemark[1]  \\
\footnotesize {HKUST}\\
{\it mazgbao@ust.hk}
\end{center}
\end{minipage}
\begin{minipage}{0.32\textwidth}
\begin{center}
L\'aszl\'o Erd{\H o}s\footnotemark[2]  \\
\footnotesize {IST Austria}\\
{\it lerdos@ist.ac.at}
\end{center}
\end{minipage}
\begin{minipage}{0.33\textwidth}
 \begin{center}
Kevin Schnelli\footnotemark[3]\\
\footnotesize 
{KTH Royal Institute of Technology}\\
{\it schnelli@kth.se}
\end{center}
\end{minipage}
\footnotetext[1]{Partially supported by  Hong Kong RGC grants ECS 26301517 and  GRF  16301519.}
\footnotetext[2]{Partially supported by ERC Advanced Grant RANMAT No.\ 338804.}
\footnotetext[3]{Partially supported by  the G\"oran Gustafsson Foundation and the Swedish Research Council grant VR-2017-05195.}

\renewcommand{\thefootnote}{\fnsymbol{footnote}}	

\end{center}

\vspace{1cm}

\begin{center}
 \begin{minipage}{0.83\textwidth}\footnotesize{
 {\bf Abstract.}  We consider the sum of two large  Hermitian matrices  $A$ and $B$ 
 with a Haar unitary conjugation bringing them into a general relative position. We prove that  
  the eigenvalue density on the  scale slightly above 
   the local eigenvalue spacing is asymptotically given by the free additive convolution  of the laws of $A$ and $B$
  as the dimension of the matrix increases.
  This implies  optimal rigidity of the eigenvalues and optimal rate of convergence
  in Voiculescu's theorem.  Our previous works \cite{BES15b, BES16} established these results
  in the bulk spectrum, the current paper completely settles  the problem at the spectral edges provided they have  the typical 
  square-root behavior. The key element of our proof  is to compensate the deterioration  
 of the  stability of the subordination 
 equations  by sharp  error estimates that properly   
 account for the local density near the edge.  Our results also hold if the Haar unitary matrix is replaced by the Haar orthogonal matrix.
 
}
\end{minipage}
\end{center}

 \vspace{2mm}
 
 {\small
\footnotesize{\noindent\textit{Date}: \today}
 
 \vspace{2mm}

 }

\thispagestyle{headings}

\section{Introduction}

The pioneering work of Voiculescu \cite{Voi91} identified the eigenvalue density of
the sum of two  Hermitian $N\times N$ matrices $A$ and $B$ in a general relative position as
the free additive convolution of the eigenvalue densities $\mu_A$ and $\mu_B$ of $A$ and $B$. The primary example
for general relative positions is asymptotic freeness that can be generated by conjugation via a Haar distributed 
unitary matrix.  In fact, under some mild regularity condition on $\mu_A$ and $\mu_B$, {\it local laws} also hold, asserting that
the empirical eigenvalue  density of the sum  converges on small scales as well. The optimal 
 precision in such  local law pins down the location of individual eigenvalues with 
an error bar that is just slightly above the local eigenvalue spacing. With an optimal error term, it 
identifies the  speed of convergence of order $N^{-1+\epsilon}$ in Voiculescu's limit theorem. 

After several gradual improvements on the  precision in \cite{Kargin2012, Kargin, BES15}, the local law  
on the optimal $N^{-1+\epsilon}$ scale was established in \cite{BES15b} and the optimal convergence speed was
obtained in \cite{BES16}. All these results were, however, restricted to the {\it regular bulk} spectrum, \ie
to the spectral  regime where the density of the free convolution is non-vanishing and bounded from above.
In particular, the regime of the spectral edges were not covered. Under mild conditions on the
limiting eigenvalue densities of $A$ and $B$, the free convolution density always vanishes as the square-root
function near the edges of its support. We call such type of edges {\it regular}. We remark that the regular edge is typical in many random matrix models, for instance, the semicircle law; \ie 
the limiting density for Wigner matrices.

Near the edges the eigenvalues are  sparser hence they fluctuate more;  naively,  the extreme
eigenvalues might   be  prone to very large fluctuations due to the room available to them
on the opposite side of the support. Nevertheless, for Wigner matrices and many related ensembles
with independent or weakly dependent entries  it has been shown that the eigenvalue fluctuation
does not exceed its natural threshold, the local spacing, even at the edge; see \eg~\cite{EKYY, LS13, AEK15} and references therein.
In general, it implies a very strong concentration of the empirical measure.
 For the smallest and largest eigenvalues
it means a fluctuation of order $N^{-2/3}$.
 In fact, the precise fluctuation is universal and it
 follows the Tracy--Widom distribution; see \eg~\cite{TW, BEY14, LS16} for  proofs in  various models.

In this paper we present a comprehensive edge local law    on optimal scale and with optimal precision for the ensemble $A+UBU^*$
where $U$ is Haar unitary. We assume that the laws of $A$ and $B$ are close to  continuous limiting profiles $\mu_\alpha$ and $\mu_\beta$
with a single interval support and power law behavior at the edge with exponent less than one. We prove that 
the free convolution $\mu_\alpha\boxplus\mu_\beta$ has a square root singularity at its edge and $\mu_A\boxplus \mu_B$ 
 closely trails this behavior. Furthermore, we establish that the eigenvalues of $A+UBU^*$ follow $\mu_A\boxplus \mu_B$ 
down to the scale of the local spacing,  uniformly throughout the spectrum. In particular, we show that the extreme eigenvalues 
are in the optimal $N^{-\frac23+\varepsilon}$ vicinity of the deterministic spectral edges. Previously, similar results were only known with $o(1)$ precision, see \cite{CM}  for instance. We expect that Tracy--Widom law holds at the regular edge of our additive model. Very recently, bulk universality has been demonstrated in \cite{CL}.

Our analysis also implies optimal rate of
convergence for Voiculescu's global law for  free convolution densities with the typical square root edges.

The result demonstrates that  the Haar randomness in the additive model has a similarly strong  concentration 
of the empirical density as already proved for the Wigner ensemble earlier. In fact, the additive model is only the simplest prototype
of a large family of models involving polynomials of Haar unitaries and deterministic matrices; other examples include
the ensemble in  the single ring theorem \cite{GKZ11, BES16b}.
The technique developed in the current paper can potentially handle square root edges in more complicated ensembles
where the main source of randomness is the Haar unitaries.

After the  statement of  the main result and the introduction of a few basic quantities,
we show in Section~\ref{s. subordination at the edge} that $\mu_\alpha\boxplus\mu_\beta$ has under suitable conditions a square root singularity at the lowest edge and we establish stability properties of the subordination equations around that edge. In Section~\ref{sec:general} an informal outline of the proof that explains the main difficulties stemming from the edge 
in contrast to the related analysis in the bulk. Here we highlight only the key point.
A typical proof of the local laws has two parts: $(i)$ stability analysis  of a deterministic (Dyson) equation for the limiting eigenvalue
distribution, and $(ii)$ proof that the empirical density approximately satisfies the Dyson equation 
and estimate the error.  Given these two inputs, the local law follows by  simply inverting the Dyson equation.
For our model the Dyson equation is actually  the pair of the  {\it subordination equations},
that define the free convolution.
Near the spectral edge, the subordination  equations become unstable. A similar phenomenon is well known
for the Dyson equation of Wigner type models, but  it has not  yet been analyzed for the 
subordination equations. This instability can only be compensated by a very  accurate estimate
on the approximation error; a formidable task given the complexity of the analogous  error estimates
in the bulk \cite{BES16}. Already the bulk analysis required carefully selected counter terms and weights 
in the fluctuation averaging mechanisms before recursive moment estimates could be started. 
All these ideas are used  at the edge, even up to higher order,  but they still  fall short of the necessary precision.
The key novelty is to identify a very specific linear combination of two basic fluctuating quantities with a fluctuation smaller
than those of its constituencies,  indicating a very special strong correlation between~them.

{\it Notation:}
The symbols $O(\,\cdot\,)$ and $o(\,\cdot\,)$ stand for the standard big-O and little-o notation. We use~$c$ and~$C$ to denote positive finite constants that do not depend on the matrix size~$N$. Their values may change from line to line.

We denote by $M_N(\C)$ the set of $N\times N$ matrices over $\C$. For a vector $\mathbf{v}\in \mathbb{C}^N$, we use $\|\mathbf{v}\|$ to denote its Euclidean norm. For $A\in M_N(\C)$, we denote by $\|A\|$ its operator norm and by $\|A\|_2$ its Hilbert-Schmidt norm.  We use $\ntr A=\frac{1}{N}\sum_{i} A_{ii}$ to denote the normalized trace of an $N\times N$ matrix $A=(A_{ij})_{N,N}$.

 Let $\mathbf{g}=(g_1,\ldots, g_N)$ be a real or complex Gaussian vector. We write $\mathbf{g}\sim \mathcal{N}_{\mathbb{R}}(0,\sigma^2I_N)$ if $g_1,\ldots, g_N$ are independent and identically distributed (i.i.d.) $N(0,\sigma^2)$ normal variables; and we write $\mathbf{g}\sim \mathcal{N}_{\mathbb{C}}(0,\sigma^2I_N)$ if $g_1,\ldots, g_N$ are i.i.d.\ $N_{\mathbb{C}}(0,\sigma^2)$ variables, where $g_i\sim N_{\mathbb{C}}(0,\sigma^2)$ means that $\Re g_i$ and $\Im g_i$ are independent $N(0,\frac{\sigma^2}{2})$ normal variables. 
 
 For two possibly $N$-dependent numbers $a,b\in \mathbb{C}$, we write $a\sim b$ if there is a (large) positive constant $C>1$ such that $C^{-1}|a|\leq |b|\leq C|a|$. 
We use double brackets to denote index sets, \ie for $n_1, n_2\in\R$, $\llbracket n_1,n_2\rrbracket\deq [n_1, n_2] \cap\Z$, and denote by $\C^+$ the upper complex half-plane, \ie $\C^+\deq\{z\in\C\,:\Im z>0\}$.

{\it Acknowledgement.} The authors are grateful to the anonymous referee for pointing out a gap
in the first version of the paper as well as for numerous useful remarks and comments.

\section{Definition of the Model and main results}

\subsection{Model and assumptions} Let $A\equiv A_N=\text{diag}(a_1, \ldots, a_N)$ and $B\equiv B_N=\text{diag}(b_1, \ldots, b_N)$ be two deterministic real diagonal matrices in $ M_N(\mathbb{C})$.  Let $U\equiv U_N$ be a random unitary matrix which is Haar distributed on $\mathcal{U}(N)$, where $\mathcal{U}(N)$ is the $N$-dimensional unitary group. We study the following random Hermitian matrix
\begin{align}
H\equiv H_N\deq A+UBU^*.  \label{17080150}
\end{align}
More specifically, we study the eigenvalues of $H$, denoted by $\lambda_1\leq\ldots\leq \lambda_N$.  
Throughout the paper, we are mainly working in the vicinity of the bottom of the spectrum. The discussion for the top of the spectrum is analogous.  Let $\mu_A$, $\mu_B$ and $\mu_H$ be the empirical eigenvalue distributions of $A$,  $B$, and $H$,~\ie 
\begin{align*}
\mu_A\deq\frac{1}{N}\sum_{i=1}^N \delta_{a_i}\,,\quad \qquad \mu_B\deq\frac{1}{N} \sum_{i=1}^N \delta_{b_i}\,,\quad \qquad \mu_H\deq\frac{1}{N}\sum_{i=1}^N \delta_{\lambda_i} \,.
\end{align*}

For any probability measure $\mu$ on the real line, its Stieltjes transform is defined as 
\begin{align*}
m_\mu(z)\deq\int_{\mathbb{R}} \frac{1}{x-z}\,\dd\mu(x)\,,\qquad \qquad z\in \mathbb{C}^+\,,
\end{align*}
where $z$ is called the {\it spectral parameter}. Throughout the paper, we write $z=E+\mathrm{i}\eta$,~\ie $E=\Re z$,  $\Im z=\eta$.

In this paper, we assume  that there are two  $N$-independent absolutely continuous probability measures~$\mu_\alpha$ and~$\mu_\beta$ with continuous density functions $\rho_\alpha$ and $\rho_\beta$, respectively,  such that the following assumptions, Assumptions~\ref{a.regularity of the measures} and~\ref{a. levy distance}, are satisfied. The first one discusses some qualitative properties of $\mu_\alpha$ and $\mu_\beta$, while the second one demands that $\mu_A$ and $\mu_B$ are close to $\mu_\alpha$ and $\mu_\beta$, respectively. 

\begin{assumption} \label{a.regularity of the measures} We assume the following:
\begin{itemize}
\item[$(i)$] Both density functions $\rho_\alpha$ and $\rho_\beta$ have single non-empty interval supports, $[E_-^\alpha,E_+^\alpha]$ and $[E_-^\beta,E_+^\beta]$, respectively,
and $\rho_\alpha$ and $\rho_\beta$ are strictly positive in the interior of their supports. 
\item[$(ii)$] In a small $\delta$-neighborhood of the lower edges of the supports, 
these measures have a power law behavior, namely, there is a (small) constant $\delta>0$ and exponents $ -1<t_-^\alpha, t_-^\beta<1$ such that 
\begin{align*}
C^{-1}\leq\frac{\rho_\alpha(x)}{(x-E_-^\alpha)^{t_-^\alpha}}\leq C\,,\qquad\qquad \forall x\in [E_-^\alpha, E_-^\alpha+\delta]\,,  \nonumber\\
C^{-1}\leq\frac{\rho_\beta(x)}{(x-E_-^\beta)^{t_-^\beta}}\leq C\,,\qquad\qquad \forall x\in [E_-^\beta, E_-^\beta+\delta]\,,
\end{align*}
hold for some  positive constant $C>1$. 
\item[$(iii)$]   We assume  that at least one of the  following two bounds holds
\begin{equation}\label{mbound}
  \sup_{z\in \C^+}|m_{\mu_\alpha}(z)|\le C\,, \qquad\qquad   \sup_{z\in \C^+}|m_{\mu_\beta}(z)|\le C\,,
\end{equation}
for some positive constant $C$. 
\end{itemize} 
\end{assumption}

\begin{assumption}\label{a. levy distance} We assume the following:
\begin{itemize}
\item[$(iv)$] For the L\'evy-distances $d_L$, we have that
\begin{align}\label{levy}
   \mathbf{d}\deq d_L(\mu_A, \mu_\alpha) + d_L(\mu_B, \mu_\beta)\le N^{-1+\epsilon}\,,
\end{align}
for any  constant $\epsilon>0$ when $N$ is sufficiently large. 
\item[$(v)$]  For the lower edges, we have 
\begin{align}\label{supab}
     \inf\, \mathrm{supp} \,\mu_A\ge E_-^\alpha -\delta\,, \qquad\qquad  \inf\, \mathrm{supp}\, \mu_B\ge E_-^\beta -\delta\,,
\end{align}
for any constant $\delta>0$ when $N$ is sufficiently large.
\item[$(vi)$] For the upper edges, we assume that there is a constant $C$ such that
\begin{align}
     \sup\, \mathrm{supp} \,\mu_A\le C\,, \qquad\qquad  \sup\, \mathrm{supp}\, \mu_B\le C\,.
\end{align}
\end{itemize}
\end{assumption}
A direct consequence of $(v)$ and $(vi)$ above is that there is a  constant $C'$ such that $\|A\|, \|B\|\leq C'\,.$

Since \cite{Voi91}, it is well known now that $\mu_H$ can be weakly approximated by  a deterministic probability measure, called the free additive convolution of $\mu_A$ and $\mu_B$. Here we briefly introduce some notations concerning the free additive convolution, which will be necessary to state our main results. 

For a probability measure $\mu$ on $\R$, we denote by $F_\mu$ its negative reciprocal Stieltjes transform,~\ie
\begin{align}\label{le reciprocal m}
F_\mu(z)\deq -\frac{1}{m_\mu(z)}\,,\qquad\qquad z\in\C^+\,.
\end{align}
Note that $F_{\mu}\,:\C^+\rightarrow \C^+$ is analytic and satisfies
\begin{align}
 \lim_{\eta\nearrow\infty}\frac{F_{\mu}(\ii\eta)}{\ii\eta}=1\,.
\end{align}
Conversely, if $F\,:\,\C^+\rightarrow\C^+$ is an analytic function with
$\lim_{\eta\nearrow\infty} F(\ii\eta)/\ii\eta=1$, then $F$ is the negative reciprocal Stieltjes transform of a probability
measure $\mu$, \ie $F(z) = F_{\mu}(z)$, for all $z\in\C^+$; see \eg~\cite{Aki}.

The {\it free additive convolution} is the symmetric binary operation on Borel probability measures on~$\R$ characterized by the following result.
\begin{pro}[Theorem 4.1 in~\cite{BB}, Theorem~2.1 in~\cite{CG}]\label{le prop 1}
Given two Borel probability measures, $\mu_1$ and $\mu_2$, on $\R$, there exist unique analytic functions, $\omega_1,\omega_2\,:\,\C^+\rightarrow \C^+$, such that,
 \begin{itemize}[noitemsep,topsep=0pt,partopsep=0pt,parsep=0pt]
  \item[$(i)$] for all $z\in \C^+$, $\im \omega_1(z),\,\im \omega_2(z)\ge \im z$, and
  \begin{align}\label{le limit of omega}
  \lim_{\eta\nearrow\infty}\frac{\omega_1(\ii\eta)}{\ii\eta}=\lim_{\eta\nearrow\infty}\frac{\omega_2(\ii\eta)}{\ii\eta}=1\,;
  \end{align}
  \item[$(ii)$] for all $z\in\C^+$, 
  \begin{align}\label{le definiting equations}
   F_{\mu_1}(\omega_{2}(z))=F_{\mu_2}(\omega_{1}(z))\,,\qquad\qquad \omega_1(z)+\omega_2(z)-z=F_{\mu_1}(\omega_{2}(z))\,.
  \end{align}
 \end{itemize}
\end{pro}

The analytic function $F\,:\,\C^+\rightarrow \C^+$ defined by
\begin{align}\label{le kkv}
 F(z)\deq F_{\mu_1}(\omega_{2}(z))=F_{\mu_2}(\omega_{1}(z))\,,
\end{align}
 is, in virtue of~\eqref{le limit of omega}, the negative reciprocal Stieltjes transform of a probability measure $\mu$, called the free additive convolution of $\mu_1$ and $\mu_2$, denoted by $\mu\equiv\mu_1\boxplus\mu_2$. The functions $\omega_1$ and $\omega_2$ are referred to as the {\it subordination functions}. The subordination phenomenon for the addition of freely independent non-commutative random variables was first noted by Voiculescu~\cite{Voi93} in a generic situation
and extended to full generality by Biane~\cite{Bia98}.

Choosing $(\mu_1, \mu_2)=(\mu_\alpha, \mu_\beta)$ in Proposition \ref{le prop 1}, we denote the associated subordination functions $\omega_1$ and $\omega_2$ by $\omega_\alpha$ and $\omega_\beta$, respectively. Analogously, for the choice $(\mu_1, \mu_2)=(\mu_A, \mu_B)$, we denote by $\omega_A$ and $\omega_B$ the associated subordination functions.   With the above notations, we obtain from (\ref{le definiting equations}) and (\ref{le kkv}) the following subordination equations
\begin{align}
m_{\mu_A}(\omega_B(z))&=m_{\mu_B}(\omega_A(z))=m_{\mu_A\boxplus\mu_B}(z),\nonumber\\
\omega_A(z)+\omega_B(z)-z&=-\frac{1}{m_{\mu_A\boxplus\mu_B}(z)}.  \label{170730100}
\end{align}
The same system of equations hold if we replace the subscripts $(A, B)$ by $(\alpha, \beta)$.

We denote the lower and upper edges of the support of $\mu_\alpha\boxplus\mu_\beta$ by 
\begin{align}
E_-\deq\inf\,\mathrm{supp}\, \mu_\alpha\boxplus\mu_\beta\,, \qquad\qquad E_+\deq \sup\,\textrm{supp} \,\mu_\alpha\boxplus\mu_\beta\,.  \label{17080330}
\end{align} 
In Section~\ref{s. subordination at the edge}, we establish 
various qualitative properties of $\mu_\alpha\boxplus\mu_\beta$ and of $\mu_A\boxplus\mu_B$.
In particular, under Assumption  \ref{a.regularity of the measures}, 
we show that $\mu_\alpha\boxplus\mu_\beta$  has a square-root decay at the lower edge, see (\ref{17080390}).

\subsection{Main results}
To state our results, we introduce some more terminology. 
We  denote the Green function or resolvent of $H$ and its normalized trace by 
\begin{align*}
G(z)\equiv G_H(z)\deq \frac{1}{H-z}\,, \qquad m_H(z)\deq\ntr G(z)=\frac{1}{N}\sum_{i=1}^N G_{ii}(z)\,, \qquad\qquad z\in \mathbb{C}^+.  
\end{align*}
Observe that $m_H(z)$ is also the Stieltjes transform of $\mu_H$, \ie
\begin{align*}
m_H(z)= \int_\R \frac{1}{x-z} \dd\mu_H(x)=\frac{1}{N}\sum_{i=1}^N \frac{1}{\lambda_i-z}, \qquad z\in \mathbb{C}^+.
\end{align*}
We further set 
\begin{align}
\mathcal{K}\deq \|A\|+\|B\|+1\,.  \label{17072840}
\end{align}
Moreover, for any  spectral parameter $z=E+\mathrm{i}\eta\in \mathbb{C}^+$, we let
\begin{align}
\kappa\equiv \kappa(z)\deq  \min \{|E-E_-|, |E-E_+|\}\,,  \label{17080102}
\end{align}
with $E_\pm$ given in~\eqref{17080330}. We then introduce the  following domain of the spectral parameter $z$: For any $0<a\le b$ and $0<\tau<\frac{E_+-E_-}{2}$,
\begin{align}
\mathcal{D}_\tau(a, b)\deq \{z=E+\ii \eta\in \mathbb{C}^+:  -\mathcal{K}\leq E\leq E_-+\tau, \quad a\leq \eta\leq b\}. \label{17020201}
\end{align}
For any (small) positive constant $\gamma>0$, we set 
\begin{align*}
\eta_{\rm m}\deq N^{-1+\gamma}.
\end{align*}
Let $\eta_\mathrm{M}>1$ be some sufficiently large constant. 
In the rest of the paper, we will mainly work in the regime $z\in \mathcal{D}_{\tau}(\eta_{\rm m}, \eta_\mathrm{M})$ with sufficiently small constant $\tau>0$. In particular, we usually have $\eta_{\mathrm{m}}\le \eta\le\eta_{\mathrm{M}}$.

We also need the following definition on high-probability estimates from~\cite{EKY}. In Appendix~\ref{appendix A} we collect some of its properties.

\begin{defn}\label{definition of stochastic domination}
Let $\mathcal{X}\equiv \mathcal{X}^{(N)}$ and $\mathcal{Y}\equiv \mathcal{Y}^{(N)}$ be two sequences of
 nonnegative random variables. We say that~$\mathcal{Y}$ stochastically dominates~$\mathcal{X}$ if, for all (small) $\epsilon>0$ and (large)~$D>0$,
\begin{align}
\P\big(\mathcal{X}^{(N)}>N^{\epsilon} \mathcal{Y}^{(N)}\big)\le N^{-D},
\end{align}
for sufficiently large $N\ge N_0(\epsilon,D)$, and we write $\mathcal{X} \prec \mathcal{Y}$ or $\mathcal{X}=O_\prec(\mathcal{Y})$.
 When
$\mathcal{X}^{(N)}$ and $\mathcal{Y}^{(N)}$ depend on a parameter $v\in \mathcal{V}$ (typically an index label or a spectral parameter), then $\mathcal{X}(v) \prec \mathcal{Y} (v)$, uniformly in $v\in \mathcal{V}$, means that the threshold $N_0(\epsilon,D)$ can be chosen independently of $v$. 
\end{defn}

With these definitions and notations, we now state our main result.
\begin{thm}[Local law at the regular edge] \label{thm. strong law at the edge} Suppose that Assumptions \ref{a.regularity of the measures} and \ref{a. levy distance} hold. Let $\tau>0$ be a sufficiently small constant and fix any (small) constants $\gamma>0$ and $\varepsilon>0$. Let $d_1, \ldots, d_N\in \mathbb{C}$ be any deterministic complex number satisfying 
\begin{align*}
\max_{i\in \llbracket 1, N\rrbracket} |d_i| \leq 1.
\end{align*}
Then
\begin{align}
\Big| \frac{1}{N} \sum_{i=1}^N d_i\Big(G_{ii}(z)-\frac{1}{a_i-\omega_B(z)}\Big)\Big|\prec \frac{1}{N\eta} \label{17072330}
\end{align}
holds uniformly on $\mathcal{D}_\tau(\eta_{\rm m},\eta_\mathrm{M})$ with $\eta_{\rm m}=N^{-1+\gamma}$ and any  constant $\eta_\mathrm{M}>0$. In particular, choosing $d_i=1$ for all $i\in \llbracket 1, N\rrbracket$, we have the estimate
\begin{align}
\Big| m_H(z)-m_{\mu_A\boxplus\mu_B}(z)\Big|\prec \frac{1}{N\eta}\,, \label{17011304}
\end{align}
uniformly on $\mathcal{D}_\tau(\eta_{\rm m},\eta_\mathrm{M})$. Moreover, we have the improved estimate
\begin{align}
\Big| m_H(z)-m_{\mu_A\boxplus\mu_B}(z)\Big|\prec \frac{1}{N(\kappa+\eta)}\,, \label{17072847}
\end{align}
uniformly for all $z=E+\ii\eta\in \mathcal{D}_\tau(0,\eta_\mathrm{M})$ with $E\leq E_--N^{-\frac{2}{3}+\varepsilon}$. Here, $\kappa=|E-E_-|$ is given in~\eqref{17080102}.

\end{thm}
Let $\gamma_j$ be the $j$-th $ N$-quantile of  $\mu_\alpha\boxplus \mu_\beta$, \ie $\gamma_j$ is the smallest real number such that
\begin{align}\label{quantile}
 \mu_{\alpha}\boxplus\mu_\beta\big((-\infty, \gamma_j]\big)= \frac{j}{N}.
\end{align}
Similarly, we define $\gamma_j^*$ to be the $j$-th $N$-quantile of $\mu_A\boxplus \mu_B$.

The following theorem is on the rigidity property of the eigenvalues of $H$. 
\begin{thm}[Rigidity at the lower edge] \label{thm. rigidity of eigenvalues} Suppose that Assumptions \ref{a.regularity of the measures} and \ref{a. levy distance} hold.  For any sufficiently small constant $c>0$, we have that for all $1\leq i\leq cN $,  
\begin{align}
   |\lambda_i-\gamma_i^*|\prec i^{-\frac{1}{3}}N^{-\frac23}. \label{17072845a}
\end{align}
In fact, the same estimate also holds if $\gamma_i^*$ is replaced with $\gamma_i$.
\end{thm}

With the following additional assumptions on the upper edges of $\mu_\alpha$, $\mu_\beta$ and $\mu_A$, $\mu_B$, we can combine the current edge analysis with our strong local law in the bulk regime in~\cite{BES16}. This yields the rigidity result for the whole spectrum.
\begin{assumption}\label{a. rigidity entire spectrum} We assume the following:
\item[$(ii')$] In a small $\delta$-neighborhood of the upper edges of their supports, the measures $\mu_\alpha$ and $\mu_\beta$ have a power law behavior, namely, there is a (large) constant $C\geq 1$ and exponents ${ -1< t_+^\alpha, t_+^\beta<1}$ such that 
\begin{align*}
C^{-1}\leq\frac{\rho_\alpha(x)}{(E_+^\alpha-x)^{t_+^\alpha}}\leq C\,,\qquad\qquad \forall x\in [E_+^\alpha-\delta, E_+^\alpha]\,,  \nonumber\\
C^{-1}\leq\frac{\rho_\beta(x)}{(E_+^\beta-x)^{t_+^\beta}}\leq C\,,\qquad\qquad \forall x\in [E_+^\beta-\delta, E_+^\beta]\,,
\end{align*}
hold for some sufficiently small constant $\delta>0$. 
\item[$(v')$]  For the upper edges of $\mu_A$ and $\mu_B$, we have 
\begin{align*}
     \sup\,\mathrm{supp} \,\mu_A\le E_+^\alpha +\delta\,, \qquad\qquad  \sup\,\mathrm{supp}\, \mu_B\le E_+^\beta +\delta\,,
\end{align*}
for any constant $\delta>0$ when $N$ is sufficiently large.
\item[$(vii)$] The density function  of $\mu_\alpha\boxplus\mu_\beta$  has a single interval support, \ie
\begin{align*}
\mathrm{supp}\, \mu_\alpha\boxplus\mu_\beta=[E_-, E_+]\,,
\end{align*} 
and has strictly positive density on $(E_-,E_+)$.
\end{assumption}

\begin{cor} [Rigidity for the whole spectrum] \label{c. rigidity for whole spectrum} Suppose that Assumptions \ref{a.regularity of the measures},  \ref{a. levy distance} and \ref{a. rigidity entire spectrum} hold. Then we have, for all $i\in \llbracket 1, N\rrbracket $, the estimate
\begin{align}
 |\lambda_i-\gamma_i^*|\prec \max \big\{i^{-\frac{1}{3}}, (N-i+1)^{-\frac{1}{3}}\big\}N^{-\frac23}. \label{17072845}
\end{align}
The same estimate also holds if $\gamma_i^*$ is replaced with $\gamma_i$.
Moreover, we have the following estimate on the convergence rate of $\mu_H$,
\begin{align}
\sup_{x\in \mathbb{R}} \big| \mu_H((-\infty, x])-\mu_A\boxplus\mu_B((-\infty, x])\big|\prec \frac{1}{N}\,.  \label{17080225}
\end{align}
The same estimate also holds if $\mu_A\boxplus\mu_B$ is replaced by $\mu_\alpha\boxplus\mu_\beta$.
\end{cor}
\begin{rem}
In this work we focus on the extremal edges. Under Assumption~\ref{a.regularity of the measures} one can indeed prove~\cite{BES18a} that $\mu_\alpha\boxplus\mu_\beta$ is supported on a single interval. In case that $\mu_\alpha$ or $\mu_\beta$ are supported on several intervals, the free convolution $\mu_\alpha\boxplus\mu_\beta$ may also be supported on several intervals. In that case the presented work will directly apply to the smallest and largest edge points.
\end{rem}

 \begin{rem} All of our results above hold also for the orthogonal setup, \ie when $U$ is a random orthogonal matrix Haar distributed on the orthogonal group $\mathcal{O}(N)$. The proof is nearly the same as the unitary setup. A discussion on the necessary modification for the block additive model 
in the bulk regime can be found in Appendix C of \cite{BES16b}. Here for our model, the modification can be done in the same way.  We omit the details.
\end{rem}

\section{Properties of the subordination functions at the regular edge} \label{s. subordination at the edge}
In this section, we collect some key properties of the subordination functions and related quantities, that will often be used 
 in Sections~\ref{s. Entrywise estimate}-\ref{s.rigidity}. We first introduce 
\begin{align}
\mathcal{S}_{AB}&\equiv \mathcal{S}_{AB}(z)\deq  (F'_{A}(\omega_B(z))-1)(F'_{B}(\omega_A(z))-1)-1\,, \nonumber\\
\mathcal{T}_A&\equiv \mathcal{T}_A(z)\deq \frac{1}{2}\Big(F''_{A}(\omega_B(z)) (F'_{B}(\omega_A(z))-1)^2+F''_{B}(\omega_A(z))(F'_{A}(\omega_B(z))-1)  \Big)\,,\nonumber\\
  \mathcal{T}_B&\equiv \mathcal{T}_B(z)\deq \frac{1}{2}\Big(F''_{B}(\omega_A(z)) (F'_{A}(\omega_B(z))-1)^2+F''_{A}(\omega_B(z))(F'_{B}(\omega_A(z))-1)  \Big)\,, \label{17080110}
\end{align}
where we use the shorthand notation $F_A\equiv F_{\mu_A}$ and $F_B\equiv F_{\mu_B}$ for the negative reciprocal Stieltjes transforms of~$\mu_A$ and~$\mu_B$, and where $\omega_A$ and $\omega_B$ are the subordination functions associated  through~\eqref{le definiting equations}. The main result in this section is the following proposition on the spectral domain $\mathcal{D}_\tau(\eta_{\rm m}, \eta_\mathrm{M})$; see~\eqref{17020201}.

\begin{pro}\label{le proposition 3.1} Suppose that  Assumptions \ref{a.regularity of the measures} and  \ref{a. levy distance} hold.  Then, for sufficiently small constant $\tau>0$, we have the following statements:
\begin{itemize}
\item[$(i)$] There exist strictly positive constants $k$ and $K$, such that
\begin{align}
&\min_{i}| a_i-\omega_B(z)|\geq k\,, \qquad& &\min_i | b_i-\omega_A(z)|\geq k\,, \label{17020502}\\
&\big|\omega_A(z)\big|\leq K, \qquad&  &\big|\omega_B(z)\big|\leq K\,, \label{17020503}
\end{align}
hold uniformly on $\mathcal{D}_\tau(\eta_{\rm m},\eta_\mathrm{M})$, for $N$ sufficiently large.
\item[$(ii)$]  For the Stieltjes transform $m_{\mu_A\boxplus\mu_B}$ of $\mu_A\boxplus\mu_B$, we have that
\begin{align}
\Im m_{\mu_A\boxplus\mu_B} (z)\sim
\begin{cases}
\sqrt{\kappa+\eta}\,,\quad& \text{if }\qquad E\in \mathrm{supp}\,\mu_A\boxplus\mu_B\,,\\
\frac{\eta}{\sqrt{\kappa+\eta}}\,, & \text{if }\qquad E\not\in \mathrm{supp}\,\mu_A\boxplus\mu_B\,,
\end{cases}
 \label{17080120}
\end{align}
uniformly on $z=E+\ii\eta\in\mathcal{D}_\tau(\eta_{\rm m},\eta_\mathrm{M})$, for $N$ sufficiently large, with $\kappa$ given in~\eqref{17080102}.
\item[$(iii)$] For $\mathcal{S}_{AB}$, $\mathcal{T}_A$ and $\mathcal{T}_B$ defined in (\ref{17080110}), we have
\begin{align}
\mathcal{S}_{AB}(z)\sim  \sqrt{\kappa+\eta}\,, \quad\qquad |\mathcal{T}_A(z)|\leq C\,,\qquad\quad |\mathcal{T}_B(z)|\leq C\,, \label{17080121}
\end{align}
uniformly on $z\in \mathcal{D}_\tau(\eta_{\rm m}, \eta_\mathrm{M})$, for $N$ sufficiently large, with some constant $C$. In addition, for $z=E+\mathrm{i}\eta\in \mathcal{D}_\tau(\eta_{\rm m}, \eta_\mathrm{M})$ with $|E-E_-|\leq \delta$ and $\eta\leq \delta$ for some sufficiently small constant $\delta>0$, we also have 
\begin{align}
|\mathcal{T}_A(z)|\geq c\,,\quad \qquad  |\mathcal{T}_B(z)|\geq c\,, \label{17080122}
\end{align}
for $N$ sufficiently large, with some strictly positive constant $c=c(\delta)$. 
\item[$(iv)$] For $\omega_A$, $\omega_B$ and $\mathcal{S}_{AB}$ we have
\begin{align}\label{le lipschitz stuff}
|\omega_A'(z)| \leq C\frac{1}{\sqrt{\kappa+\eta}}\,,\quad \qquad |\omega_B'(z)| \leq C\frac{1}{\sqrt{\kappa+\eta}}\,,\quad \qquad |\mathcal{S}_{AB}' (z)|\leq C\frac{1}{\sqrt{\kappa+\eta}}\,,
\end{align}
any $z\in \mathcal{D}_\tau (\eta_{\rm m}, \eta_\mathrm{M})$, for $N$ sufficiently large, with some constant $C$. 
\end{itemize}
\end{pro}
The proof of Proposition~\ref{le proposition 3.1} is split into two steps. In the first step, carried out in Subsection~\ref{appendix. properties of free convolution}, we derive the analogous statements for the $N$-independent measures $\mu_\alpha$ and $\mu_\beta$. This step requires only Assumption~\ref{a.regularity of the measures}. In the second step, carried out in Subsection~\ref{app:stab}, we show that the statements carry over to the $N$-dependent measures $\mu_A$ and $\mu_B$ under Assumption~\ref{a. levy distance}, for $N$ sufficiently large.

\subsection{Free convolution measure $\mu_\alpha\boxplus\mu_\beta$} \label{appendix. properties of free convolution}
In this subsection, we derive some properties of the free additive convolution of $\mu_\alpha$ and $\mu_\beta$. We will always assume that $\mu_\alpha$ and $\mu_\beta$ satisfy Assumption~\ref{a.regularity of the measures}. From Assumption \ref{a.regularity of the measures} $(iii)$ we know that
\begin{align}
\sup_{z\in \mathbb{C}^+} |m_{\mu_\alpha\boxplus\mu_\beta}(z)|\leq C.  \label{17080326}
\end{align}
In addition, under Assumption \ref{a.regularity of the measures}, we see from Theorem 2.3 and Remark 2.4 in \cite{Bel1} that $\omega_\alpha(z)$, $\omega_\beta(z)$ and $m_{\mu_\alpha\boxplus\mu_\beta}(z)$ can be extended continuously to $\mathbb{C}^+\cup \mathbb{R}$. This together with (\ref{17080326}) implies that $\mu_\alpha\boxplus\mu_\beta$ is absolutely continuous with a continuous and bounded density function. 

Recall from Assumption~\ref{a.regularity of the measures} that $\mathrm{supp}\,\mu_\alpha= [E_-^\alpha,E_+^\alpha]$ and $\mathrm{supp}\,\mu_\beta=[E_-^\beta,E_+^\beta]$. We introduce the spectral domain $\mathcal{E}\subset\C$ by setting
\begin{align}
 \mathcal{E}\deq\{z\in\C^+\cup\R\,:\, E_-^\alpha+E_-^\beta-1\le \re z\le E_+^\alpha+E_+^\beta+1\,, 0\le\im z\le \eta_\mathrm{M}\}\,, 
\end{align}
where $\eta_\mathrm{M}>0$ is any constant. By Lemma~3.1 in~\cite{Voi86}, we have that $\mathrm{supp}\,\mu_\alpha\boxplus\mu_\beta\subset \mathcal{E}\cap\R$.

\begin{lem}\label{lemma 1}
There exists a constant $C$ such that
 \begin{align}
 \sup_{z\in\mathcal{E}}( |\omega_\alpha(z)|+|\omega_\beta(z)|)\le C\,.
 \end{align}
\end{lem}

\begin{proof}
 Let $L>\max\{|E_+^\alpha+E_+^\beta+1|,|E_-^\alpha+E_-^\beta-1|\}$ and $M>10$ be large numbers to be chosen later. We will argue by contradiction. Assume first that there is $z\in\mathcal{E}$ such that
 \begin{align}\label{ram1}
  |\omega_\alpha(z)|> LM \,,\qquad\qquad |\omega_\beta(z)|> L\,.
 \end{align}
Then we have from~\eqref{le definiting equations} that
\begin{align}
 \frac{1}{\omega_\alpha(z)+\omega_\beta(z)-z}=-\int_\R\frac{\dd\mu_\alpha(x)}{x-\omega_\beta(z)}&=\frac{1}{\omega_\beta(z)}+O((\omega_\beta(z))^{-2})\label{le first s}\,,\\
 \frac{1}{\omega_\alpha(z)+\omega_\beta(z)-z}=-\int_\R\frac{\dd\mu_\beta(x)}{x-\omega_\alpha(z)}&=\frac{1}{\omega_\alpha(z)}+O((\omega_\alpha(z))^{-2})\,,\label{le second s}
\end{align}
as $L\rightarrow\infty$. Thus we get from~\eqref{le second s}, as $z\in\mathcal{E}$, that in the same limit
\begin{align}\label{ram11}
 \frac{\omega_\beta(z)}{\omega_\alpha(z)}=O\left((\omega_\alpha(z)^{-1}\right)\,.
\end{align}
But then we have from~\eqref{ram1} and~\eqref{ram11} that
\begin{align}
  \frac{L}{|\omega_\alpha(z)|}\le\frac{|\omega_\beta(z)|}{|\omega_\alpha(z)|}\le C\frac{1}{|\omega_\alpha(z)|}\,,
\end{align}
hence for $L$ sufficiently large, we get a contradiction.

Next, assume that there is $z\in\mathcal{E}$ such that
\begin{align}\label{ram2}
 |\omega_\alpha(z)|> LM \,,\qquad\qquad |\omega_\beta(z)|\le L\,.
\end{align}
Then we conclude from~\eqref{le definiting equations} that
\begin{align}\label{ram4}
 \frac{1}{|m_{\mu_\alpha}(\omega_\beta(z))|}=|\omega_\alpha(z)+\omega_\beta(z)-z|\ge\frac{LM}{2}\,,
\end{align}
for $M$ sufficiently large, where we used that $z\in\mathcal{E}$. On the other hand, the Stieltjes transform $m_{\mu_\alpha}(z)$ does not have any zeros in $\mathcal{E}$ as the support of $\mu_\alpha$ is connected. Thus there is a constant $c>0$, depending on $L$, such that $|m_{\mu_\alpha}(z')|\ge c$, for all $z'\in\C^+$ with $|z'|\le L$. Hence, for $M$ sufficiently large, we get a contradiction from~\eqref{ram4}.

Finally, as both, ~\eqref{ram1} and~\eqref{ram2}, have been ruled out, we can conclude that
\begin{align}
 |\omega_\alpha(z) |\le LM\,,\qquad\qquad |\omega_\beta(z)|\le L\,,
\end{align}
for all $z\in \mathcal{E}$. This completes the proof of Lemma~\ref{lemma 1}. 
\end{proof}

Recall from (\ref{17080330}) that $E_-=\inf\,\mathrm{supp}\,\mu_\alpha\boxplus\mu_\beta$. Recall further that, for any spectral parameter $z$, $\kappa=\kappa(z)$ defined in (\ref{17080102}) is the distance of $\Re z$ to the endpoints of $\mathrm{supp}(\mu_\alpha\boxplus\mu_\beta)$. 

\begin{lem}\label{lemma 2}
Let $u\in\R$ with $u\le E_-$, then we have
 \begin{align}\label{le C15}
  \re\omega_\alpha(u)\le E_-^\beta\,,\qquad\qquad \re\omega_\beta(u)\le E_-^\alpha\,.
 \end{align}
 Moreover, $\re\omega_\alpha$ and $\re\omega_\beta$ are monotone increasing on $(-\infty,E_-)$.
\end{lem}

\begin{proof}
We argue by contradiction. Assume that there exists $y'$ with $y'\le E_-$ such that $\re \omega_\alpha(y')>E_-^\beta$. Then either $\Re\omega_\alpha(y')\in(E_-^\beta,E_+^\beta)$ or $\Re\omega_\alpha(y')\ge E_+^\beta$. In the first case, using that the imaginary part of the identity $m_{\mu_\alpha\boxplus\mu_\beta}(z)= m_\alpha(\omega_\beta(z))$, we conclude that $\im m_{\mu_\alpha\boxplus\mu_\beta}(y')>0$, \ie the density of $\mu_\alpha\boxplus\mu_\beta$ at $y'$ is strictly positive. This contradicts the definition of $E_-$ (as the lowest endpoint $\mathrm{supp}\,\mu_\alpha\boxplus\mu_\beta$).

In the second case, $\re\omega_\alpha(y')\ge E_+^\beta$, we have
\begin{align}\label{llk}
 \re m_{\mu_\beta}(\omega_\alpha(y'))=\int_{E_-^\beta}^{E_+^\beta}\frac{ (x-\re\omega_\alpha(y'))\dd\mu_\beta(x)}{|x-\omega_\alpha(y')|^2}<0\,.
\end{align}
However, since $\re m_{\mu_\beta}(\omega_\alpha(y'))=\re m_{\mu_\alpha\boxplus\mu_\beta}(y')$, we get a contradiction as
\begin{align}
 \re m_{\mu_\alpha\boxplus\mu_\beta}(y')=\int_y^\infty\frac{\dd\mu_\alpha\boxplus\mu_\beta(x)}{x-y'}>0\,,
\end{align}
by the definition of $E_-$. 

From the above, we get $\re \omega_\alpha(y')\le E_-^\beta$. Repeating the argument for $\omega_\beta$, we~obtain~\eqref{le C15}.

Finally, that $\re\omega_\alpha$ and $\re\omega_\beta$ are increasing on $(-\infty,E_-)$ follows from the observation that $\re m_{\mu_\alpha\boxplus\mu_\beta}$ is increasing on $(-\infty,E_-)$, the subordination property $m_{\mu_\alpha\boxplus\mu_\beta}(z)=m_{\mu_\beta}(\omega_\alpha(z))$ and~\eqref{llk}. The same argument shows that $\re \omega_\alpha$ is increasing on $(-\infty,E_-)$. This finishes the proof of Lemma~\ref{lemma 2}.
\end{proof}

We next show that we actually have $\re \omega_\alpha(E_-)\leq E_-^\beta-k_0$ and $\re \omega_\beta(E_-)\le E_-^\alpha-k_0$, for some constant $k_0>0$. Our argument relies on the following computational lemma.

\begin{lem}\label{lemma computation}
Let $\omega=\lambda+\ii \nu$, with $\nu\ge0$ and $|\omega|\le \vartheta$, for some small $\vartheta>0$. Let $-1<t<1$. 
Then,
\begin{align}
 \int_0^\vartheta\frac{x^t\,\dd x}{(x-\lambda)^2+\nu^2}\sim\begin{cases} \frac{\lambda^t}{\nu}\,,\qquad &\textrm{if}\qquad \lambda>\nu\,,\\ 
 |\omega|^{t-1}\sim \lambda^{t-1}\,,\qquad &\textrm{if}\qquad \lambda<-\nu\,,\\
                                              \nu^{t-1}\,,\qquad &\textrm{if}\qquad \nu>|\lambda|\,.
                                             \end{cases}
\end{align}
\end{lem}
\begin{proof}
 Follows from elementary estimations.
\end{proof}

Recall from~\eqref{le reciprocal m} that $F_{\mu}(w)=-1/m_{\mu}(w)$, $w\in\C^+$, denotes the negative reciprocal Stieltjes transform of any probability measure $\mu$. As $F_{\mu}\,:\C^+\rightarrow \C^+$ is analytic, and since $\mu$ is a probability measure, it admits the Nevanlinna  representation
\begin{align}\label{representation}
 F_{\mu}(z)-z=\Re F_{\mu}(\ii)+\int_\R\left(\frac{1}{x-z}-\frac{x}{1+x^2}\right)\,\dd \widehat\mu(x)\,,
\end{align}
 where $\widehat\mu$ is a Borel measure on $\R$. Assuming in addition that $\mu$ is compactly supported, a large $z$-expansion of both sides of~\eqref{representation} reveals that

\begin{align}
 \int_\R x\,\dd\mu(x)=\Re F_{\mu}(\ii)-\int_\R\frac{x}{1+x^2}\,\dd \widehat\mu(x)\,,
 \end{align}
 and
 \begin{align}\label{representation mass}
 \widehat\mu(\R)=\int_\R x^2\,\dd\mu(x)-\Big(\int_\R x\,\dd\mu(x)\Big)^2\,.
\end{align}

\begin{lem}\label{lemma 6}
Let $\mu$ be a probability measure on $\R$ which is absolutely continuous with respect to Lebesgue measure, is of bounded support and satisfies $m_{\mu}(x)\not=0$, for all $x\in\R\backslash \mathrm{supp}\, {\mu}$. Let $\widehat\mu$ be a Borel measure on $\R$ such that~\eqref{representation} holds, then we have that
\begin{align}\label{lefrasu}
 \mathrm{supp}\,\mu=\mathrm{supp}\,\widehat\mu\,.
\end{align}
\end{lem}
\begin{proof}
The proof is almost identical to the proof of Lemma~3.2 in~\cite{BES18a}, we repeat it here for convenience of the reader. Outside the support of $\mu$, the Stieltjes transform $m_{\mu}$ extends continuously to the real line and is real valued there. Taking the imaginary parts in~\eqref{representation} and using that $F_{\mu}(z)=-1/m_{\mu}(z)$, we get
\begin{align}\label{le smanl1}
 \frac{\im m_{\mu}(z)}{|m_{\mu}(z)|^2}-\im z=\im F_{\mu}(z)-z=\int_\R\frac{\im z}{|y-z|^2}\dd\widehat\mu(y)\,.
\end{align}
Since $m_{\mu}(x)\not=0$, for all $x\in\R\backslash \mathrm{supp}\, {\mu}$, we can take the limit $\im z\searrow 0$ in~\eqref{le smanl1}, and hence conclude by the Stieltjes inversion formula that $\widehat\mu$ is absolutely continuous with respect to Lebesgue measure on $\R\backslash\mathrm{supp}\,\mu_\alpha$ with vanishing density function. We conclude that $\mathrm{supp}\,\widehat\mu \subseteq \mathrm{supp}\,\mu$.

To conclude that $\mathrm{supp}\,\mu  \subseteq \mathrm{supp}\,\widehat\mu$ we argue by contradiction: Suppose that $\mathrm{supp}\,\widehat\mu$
is a proper subset of  $\mathrm{supp}\,\mu$. Then there is a non-empty open interval $I\subset\mathrm{supp}\,\mu\backslash\mathrm{supp}\,\widehat{\mu}$ such that $f(\omega)\deq F_\mu(\omega)-\omega\,:\,\C^+\rightarrow\C^+$ extends continuously to $I$ with $\im f(\omega)=0$, for all $\omega\in I$. Hence by the Schwarz reflection principle, $f$ extends analytically through $I$ and  $m_\mu$ is meromorphic on~$I$. However, since $I\subset\mathrm{supp}\,\mu$, we have $\lim_{\eta\searrow 0}\im m_{\mu}(\omega+\ii\eta)>0$ by Assumption~\ref{a.regularity of the measures}, for almost all $\omega\in I$. Since $m_\mu$ is meromorphic on $I$ and $\im f(\omega)=\im m_\mu(\omega)/|m_\mu(\omega)|^2$, $\omega\in I$, we hence also have $\lim_{\eta\searrow 0}\im f(\omega+\ii\eta)>0$ for almost all $\omega\in I$, a contradiction to $\im f(\omega)=0$, for all $\omega\in  I$. We conclude that $I$ is empty and we have $\mathrm{supp}\,\widehat\mu = \mathrm{supp}\,\mu$. This proves~\eqref{lefrasu}.\end{proof}

\begin{rem}
 The assumptions of Lemma~\ref{lemma 6} are satisfied for $\mu_\alpha$ and $\mu_\beta$ as follows easily from Assumption~\ref{a.regularity of the measures}. Note that $m_{\mu_\alpha}(x)\not=0$ for $x\in\R\backslash \mathrm{supp}\,\mu_{\alpha}$ is guaranteed by the single interval support condition. However, the condition that $m_{\mu_A}(x)\not=0$, respectively $m_{\mu_B}(x)\not=0$ cannot be guaranteed. But in this case we have the following inclusions for the supports of the $N$-dependent measures $\widehat\mu_A$ and $\widehat\mu_B$:
 \begin{align}\label{na da ischer der support}
  \mathrm{supp}\,\widehat{\mu}_A\subset I_{\mu_A}\,,\qquad \mathrm{supp}\,\widehat{\mu}_B\subset I_{\mu_B}\,,
 \end{align}
where $I_{\mu_A}$, $I_{\mu_B}$ is the smallest interval containing $\mathrm{supp}\,\mu_A$, $\mathrm{supp}\,\mu_B$. This easily follows from the proof of Lemma~\ref{lemma 6} by noticing that $m_{\mu_A}(x)\not=0$, for $x\in\R\backslash I_{\mu_A}$, since $E\mapsto \re m_{\mu_A}(E)$ is monotone on that domain, and similar for $m_{\mu_B}$.
\end{rem}

\begin{lem}\label{lemma 3}
There is a constant $k_0>0$, such that
 \begin{align}\label{key to everything}
 \re \omega_\alpha(E_-)\le E_-^\beta-k_0\,,\qquad\re\omega_\beta(E_-)\le E_-^\alpha-k_0\,.
 \end{align}
 Moreover, there exists a constant $C$, such that
 \begin{align}\label{a lot of rs}
  \im \omega_\alpha(z)+\im\omega_\beta(z)\le \eta+C\im m_{\mu_\alpha\boxplus\mu_\beta}(z)\,,
 \end{align}
for all $z\in\mathcal{E}$. The constants $k_0$ and $C$ only depend on $\mu_\alpha$ and $\mu_\beta$.
\end{lem}
\begin{proof}Let $z\in\mathcal{E}$. Taking the imaginary part in the subordination equations~\eqref{le definiting equations} we get
 \begin{align*}
  \frac{\im\omega_\alpha(z)+\im\omega_\beta(z)-\im z}{|\omega_\alpha(z)+\omega_\beta(z)-z|^2}=\im m_{\mu_\alpha\boxplus\mu_\beta}(z)\,.
 \end{align*}
Thus we obtain
\begin{align*}
\im \omega_\alpha(z)+\im\omega_\beta(z)=\im z+|\omega_\alpha(z)+\omega_\beta(z)-z|^2\im m_{\mu_\alpha\boxplus\mu_\beta}(z)\le \eta+C\im m_{\mu_\alpha\boxplus\mu_\beta}(z)\,,
\end{align*}
where we used Lemma~\ref{lemma 1} to get the inequality. This proves~\eqref{a lot of rs}.

We move on to prove the estimates in~\eqref{key to everything}. Using
\begin{align}\label{needed later on}
\im m_{\mu_\alpha\boxplus\mu_\beta}(z)=\im \omega_\alpha(z)\int_\R\frac{\dd\mu_\beta(x)}{|x-\omega_\alpha(z)|^2}=\im\omega_\beta(z)\int_\R\frac{\dd\mu_\alpha(x)}{|x-\omega_\beta(z)|^2}\,,
\end{align}
and~\eqref{le definiting equations}, we can write
\begin{align*}
 \frac{\im m_{\mu_\alpha\boxplus\mu_\beta}(z)}{\im z}\bigg(\Big({\int_\R\frac{\dd\mu_\alpha(x)}{|x-\omega_\beta(z)|^2}}\Big)^{-1}+\Big({\int_\R\frac{\dd\mu_\beta(x)}{|x-\omega_\alpha(z)|^2}} \Big)^{-1}\bigg)-1&=\frac{\im m_{\mu_\alpha\boxplus\mu_\beta}(z)}{\im z}\frac{1}{|m_{\mu_\alpha\boxplus\mu_\beta}(z)|^2}\,,
\end{align*}
for all $z\in\mathcal{E}\cap\C^+$. Since $\im m_{\mu_\alpha\boxplus\mu_\beta}(z)/\im z>0$, for all $z\in\mathcal{E}\cap\C^+$, we obtain
\begin{align}\label{good identity}
\frac{\left|\int_\R\frac{\dd\mu_\alpha(x)}{x-\omega_\beta(z)} \right|^2}{\int_\R\frac{\dd\mu_\alpha(x)}{|x-\omega_\beta(z)|^2}}+\frac{\left|\int_\R\frac{\dd\mu_\beta(x)}{x-\omega_\alpha(z)}\right|^2}{\int_\R\frac{\dd\mu_\beta(x)}{|x-\omega_\alpha(z)|^2}}\ge 1\,,
\end{align}
for all $z\in\mathcal{E}\cap\C^+$, where we used the subordination equations to express $m_{\mu_\alpha\boxplus\mu_\beta}(z)$. To condense the notation we introduce the quantities
\begin{align}
 R_\alpha(\omega)\deq\frac{\left|\int_\R\frac{\dd\mu_\alpha(x)}{x-\omega} \right|^2}{\int_\R\frac{\dd\mu_\alpha(x)}{|x-\omega|^2}}\,,\qquad R_\beta(\omega)\deq\frac{\left|\int_\R\frac{\dd\mu_\beta(x)}{x-\omega}\right|^2}{\int_\R\frac{\dd\mu_\beta(x)}{|x-\omega|^2}}\,,\qquad \omega\in\C^+\,.
\end{align}

Fix some small $\vartheta>0$. Recalling Lemma~\ref{lemma computation}, we observe that there is a constant $c>0$ (depending on~$\vartheta$) such that
\begin{align}\label{from the computational lemma}
\int_{E_-^\beta}^{E_-^\beta+\vartheta}\frac{\dd\mu_\beta(x)}{|x-\omega|^2}\ge c\begin{cases}
                                                                                    \frac{(\re \omega-E_-^\beta)^{t_-^\beta}}{\im \omega}\,,\qquad &\textrm{ if }\qquad \re\omega-E_-^\beta\ge\im\omega\,,\\
                                                                                   |\re\omega-E_-^\beta|^{t_-^\beta-1}\,, &\textrm{ if }\qquad \re \omega-E_-^\beta\le-\im\omega\,,\\
                                                                                    (\im \omega)^{t_-^\beta-1}\,, &\textrm{ if }\qquad \im\omega>|\re\omega-E_-^\beta|\,,
                                                                                   \end{cases}
\end{align}
for all $\omega$ with $|\omega-E_-^\beta|\le\vartheta$. (Since $-1<t_-^\beta<1$,  the integral may be divergent in the limit $\im \omega\rightarrow 0$, but this does not affect the following argument.)

Similarly, we have for $\omega\in \C$ satisfying $|\omega-E_-^\beta|\le \vartheta$,
\begin{align}\label{le diverging2}
    \left|\int_\R \frac{\dd \mu_\beta(x)}{x-\omega}  \right|\le C+ C'\int_0^\vartheta \frac{x^{t_-^\beta} }{|x + E_-^\beta - \omega|} \,\dd x \le C +C' \,\big|E_-^\beta-\omega\big|^{t_-^\beta}\,,
\end{align}
for some strictly positive constants $C$ and $C'$ depending on $\vartheta$. In particular, for $t_-^\beta\in[0,1)$, the right side of~\eqref{le diverging2} is bounded. The inequalities~\eqref{from the computational lemma} and~\eqref{le diverging2} also hold true, upon possibly adjusting the constants, with the roles of $\alpha$ and $\beta$ interchanged. We remark that we used similar estimates in the proof of Lemma~3.12 in~\cite{BES18a}.

Next, we introduce the quantities
 \begin{align}
   d_\alpha(z)\deq \mathrm{dist}( \omega_\alpha(z),\mathrm{supp}\,\mu_\beta)\,,\qquad d_\beta(z)\deq \mathrm{dist}( \omega_\beta(z),\mathrm{supp}\,\mu_\alpha)\,.
 \end{align}
We now claim that there are constants $k_0>0$ and $\varrho>0$ such that $d_\alpha(z)\ge  k_0$ and $d_\beta(z)\ge k_0$ for all $z\in\C^+\cup\R$ with $|z-E_-|\le \varrho$. We proceed by distinguishing two cases: First assume that there is a $z$ with $|z-E_-|\le\varrho$ such that
\begin{align}\label{le case a}
 d_\alpha(z)\le \epsilon k\,,\qquad\qquad d_\beta(z)>k\,, 
\end{align}
for some small constants $k>0$ and $\epsilon>0$ to be chosen below.

For $t_-^\beta\ge 0$, we obtain from~\eqref{from the computational lemma} and~\eqref{le diverging2} that for such $z$, we have
\begin{align}\label{new liebling}
 R_\beta(\omega_\alpha(z))\le C\begin{cases}
                               (\re\omega_\alpha(z)-E_-^\beta)^{1-t_-^\beta}\,, \quad &\textrm{ if }\quad |\re\omega_\alpha(z)-E_-^\beta|\ge\im\omega_\alpha(z)\,,\\
                               (\im \omega_\alpha(z))^{1-t_-^\beta}\,,\quad &\textrm{ if }\quad |\re\omega_\alpha(z)-E_-^\beta|<\im\omega_\alpha(z)\,.
                              \end{cases}
\end{align}
Either way, we have $R_\beta(\omega_\alpha(z))\le C(d_\alpha(z))^{1-t_-^\beta}\le C(\epsilon k)^{1-t_-^\beta}$, where we used that $t_-^\beta<1$.

For $-1<t_-^\beta<0$, we obtain from~\eqref{from the computational lemma} and~\eqref{le diverging2} that for $z$ with $|z-E_-|\le\varrho$ and~\eqref{le case a} satisfied, 
\begin{align}\label{new liebling 2}
 R_\beta(\omega_\alpha(z))\le C\begin{cases}
                               \im \omega_\alpha(z) |\re\omega_\alpha(z)-E_-^\beta|^{t_-^\beta}\,, \quad &\textrm{ if }\quad \re\omega_\alpha(z)-E_-^\beta\ge\im\omega_\alpha(z)\,,\\
                               |\re \omega_\alpha(z)-E_-|^{1+t_-^\beta}\,, &\textrm{ if }\quad \re \omega_\alpha(z)-E_-^\beta\le-\im\omega_\alpha(z)\,,\\
                               (\im \omega_\alpha(z))^{1+t_-^\beta}\,,\quad &\textrm{ if }\quad |\re\omega_\alpha(z)-E_-^\beta|<\im\omega_\alpha(z)\,.
                               \end{cases}
\end{align}
In all three cases we find that $R_\beta(\omega_\alpha(z))\le C(d_\alpha(z))^{1+t_\beta}\le C(\epsilon k)^{1+t_-^\beta}$, where we used $1+t_-^\beta>0$. 

Since $d_\beta(z)>k$ and since we assumed that $\mu_\alpha$ is not a single point mass, we have by the Cauchy-Schwarz inequality that 
\begin{align}
 R_\alpha(\omega_\beta(z))=\frac{\left|\int_\R\frac{\dd\mu_\alpha(x)}{x-\omega_\beta(z)} \right|^2}{\int_\R\frac{\dd\mu_\alpha(x)}{|x-\omega_\beta(z)|^2}}\le 1-C_S(k,\varrho)\,,
\end{align}
for some strictly positive constant $C_S(k,\varrho)>0$ depending on $k$ and $\varrho$ (and $\mu_\alpha$). Hence, 
\begin{align}
R_\alpha(\omega_\beta(z))+R_\beta(\omega_\alpha(z))\le 1-C_S(k,\varrho)+C'(\epsilon k)^{1-|t_-^\beta|}\,.
\end{align}
with $C'$ depending on $\varrho$. Thus for $\epsilon<(C_S(k,\varrho)/C')^{1/(1-|t_-^\beta|)}/k$ we get a contradiction with~\eqref{good identity}, for any $k>0$. Thus there is no $z$ with $|z-E_-|\le\varrho$ such that~\eqref{le case a} can hold.

Assume thus that there is a $z$ with $|z-E_-|\le \varrho$ such that
\begin{align}\label{le case b}
 d_\alpha(z)\le \epsilon k\,,\qquad\qquad d_\beta(z)\le k\,, 
\end{align}
for some sufficiently small $k>0$ chosen below and with $\epsilon$ depending on $k$ as above. Following the argumentation in~\eqref{new liebling} and~\eqref{new liebling 2} with the roles of $\alpha$ and $\beta$ interchanged, we find
\begin{align}
 R_\alpha(\omega_\beta(z))=\frac{\left|\int_\R\frac{\dd\mu_\alpha(x)}{x-\omega_\beta(z)} \right|^2}{\int_\R\frac{\dd\mu_\alpha(x)}{|x-\omega_\beta(z)|^2}}\le Ck^{1-|t_-^\alpha|}\,,
\end{align}
while at the same time we have $R_\beta(\omega_\alpha(z))\le C(\epsilon k)^{1-|t_-^\beta|}$ as we had above, with the constants depending on $\varrho$. Hence choosing $k>0$ sufficiently small, we get a contradiction with~\eqref{good identity}, and we exclude~\eqref{le case b}.

 We can therefore conclude that, for $\epsilon>0$ and $k>0$ sufficiently small, we have for all $z$ with $|z-E_-|\le \varrho$, that
\begin{align}\label{le distance new}
 d_\alpha(z)> \epsilon k\,,\qquad\qquad d_\beta(z)> k\,.
\end{align}
Choosing $z=E_-$ this proves together with~\eqref{a lot of rs} and (\ref{le C15}) the estimates in~\eqref{key to everything} with $k_0\deq\epsilon k$. This concludes the proof of Lemma~\ref{lemma 3}.
\end{proof}

\begin{lem}\label{lemma 7}
The lowest endpoint $E_-$ of  $\mathrm{supp}\,\mu_\alpha\boxplus\mu_\beta$ is the smallest real solution to the equation
\begin{align}\label{ja wo ist das edge}
 (F'_{\mu_\alpha}(\omega_\beta(z))-1)(F'_{\mu_\beta}(\omega_\alpha(z))-1)=1\,,\qquad z\in\R\,.
\end{align}
Moreover, there are constants $\kappa_0>0$ and $\eta_0>0$ such that
\begin{align}\label{omega behavior}
 \Im m_{\mu_\alpha\boxplus\mu_\beta}(z)\sim\Im \omega_\alpha(z)\sim \im\omega_\beta(z)\sim\begin{cases}\sqrt{\kappa+\eta}\,,\quad & \textrm{if}\; E\ge E_-\,,\\ 
  \frac{\eta}{\sqrt{\kappa+\eta}}&\textrm{if }\; E<E_-\,,\end{cases} 
\end{align}
uniformly for all $z=E+\mathrm{i}\eta\in\mathcal{E}_0$ where
\begin{align}\label{def:eps0}
\mathcal{E}_0\deq  \big\{ z\in \C\,:\, -\kappa_0\le \Re z-E_- \le\kappa_0, 0\le \Im z\leq \eta_0\big\}\,.
\end{align}
\end{lem}

\begin{proof}[Proof of Lemma~\ref{lemma 7}]
 From Lemma~\ref{lemma 3} we know that $\Re\omega_\alpha(E_-)\le E_-^\beta -k_0$ and $\Re\omega_\beta(E_-)\le E_-^\alpha-k_0$, $k_0>0$.  From the subordination equations~\eqref{le definiting equations} and~\eqref{representation}, we have that
 \begin{align}\label{le franz}
  F_{\mu_\alpha\boxplus\mu_\beta}(z)=F_{\mu_\alpha}(\omega_\beta(z))=\Re F_{\mu_\alpha}(\ii)+\omega_\beta(z)+\int_\R\left(\frac{1}{x-\omega_\beta(z)}-\frac{x}{1+x^2}\right)\dd\widehat{\mu}_{\alpha}(x)\,,
 \end{align}
for a Borel measure $\widehat\mu_\alpha$ on $\R$ with, according to Lemma~\ref{lemma 6}, $\mathrm{supp}\,\widehat\mu_\alpha=\mathrm{supp}\,\mu_\alpha$. Arguing as in the proof of Lemma~\ref{lemma 6}, we notice that $u\in\R$ is an edge of the measure $\mu_\alpha\boxplus\mu_\beta$, if  $\im m_{\mu_\alpha\boxplus\mu_\beta}(u)=0$ and $m_{\mu_\alpha\boxplus\mu_\beta}$ fails to be analytic at $u\in\R$. Analyticity breaks down if either $F_{\mu_\alpha\boxplus\mu_\beta}(u)=0$ or, according to~\eqref{le franz}, if  $\omega_\beta(u)\in \mathrm{supp}\,\widehat{\mu}_\alpha=\mathrm{supp}\,\mu_\alpha$,  or if $\omega_\beta$ fails to be analytic at $u$. For the lowest edge at $u=E_-$, we can exclude $F_{\mu_\alpha\boxplus\mu_\beta}(u)=0$ by (\ref{17080326}) and also $\omega(u)\in\mathrm{supp}\,\mu_\alpha$ as $\re\omega_\alpha(E_-)\le E_-^\beta-k_0$, $k_0>0$. Thus $E_-\in\R$ is the smallest point where $\omega_\beta$ is not analytic.
 
We next claim that $\omega_\beta$ is not analytic at $u\in\R$ if $(F_{\mu_\alpha}'(\omega_\beta(u))-1)(F_{\mu_\beta}'(\omega_\alpha(u))-1)=1$. We argue as follows. From~\eqref{representation} we know that there is a Borel measure $\widehat\mu_\beta$ such that
 \begin{align}\label{le neva}
 F_{\mu_\beta}(\omega)=\Re F_{\mu_\beta}(\ii)+\omega+\int_{\R}\left(\frac{1}{x-\omega}-\frac{x}{1+x^2}\right)\,\dd\widehat\mu_\beta(x)\,,
 \end{align}
and $F_{\mu_\beta}$ is analytic in a disk of radius $k_0$ centered at $\omega=\omega_\beta(E_-)$ by (\ref{key to everything}). Here we also used that $\mathrm{supp}\,\widehat\mu_\beta=\mathrm{supp}\,\mu_\beta$ by Lemma~\ref{lemma 6}. It follows that
\begin{align}\label{hurra}
 F'_{\mu_\beta}(\omega)=1+\int_{\R}\frac{\dd\widehat\mu_\beta(x)}{(x-\omega)^2}\,,
\end{align}
and in particular that $F'_{\mu_\beta}(\omega_\alpha(E_-))>1$, since $\omega_\alpha(E_-)$ is real valued as $E_-$ is the lower endpoint of the support of $\mu_\alpha\boxplus\mu_\beta$ (recall~\eqref{a lot of rs}). By the analytic inverse function theorem, the functional inverse  $F^{(-1)}_{\mu_\beta}$ of $F_{\mu_\beta}$ is analytic in a neighborhood of $F_{\mu_\beta}(\omega_\alpha(E_-))$. Thus the function
\begin{align}\label{le z}
 \widetilde z(\omega)\deq -F_{\mu_\alpha}(\omega)+\omega+F^{(-1)}_{\mu_\beta}\circ F_{\mu_\alpha}(\omega)
\end{align}
is well-defined and analytic in a complex neighborhood of $\omega_\alpha(E_-)\in\R$. It follows from~\eqref{le definiting equations} that $\omega_\beta(z)$ is a solution $\omega=\omega_\beta(z)$ to the equation $z=\widetilde z(\omega)$ (with $\im \omega_\beta(z)\ge \im z$). Moreover, we have $\omega_\alpha(z)=F^{(-1)}_{\mu_\beta}\circ F_{\mu_\alpha}(\omega_\beta(z))$.

The function $\widetilde{z}(\omega)$ admits the following Taylor expansion in a complex neighborhood of $\omega_\beta(E_-)$,
\begin{align}\label{taylor expansion}
\widetilde{z}(\omega)=E_-+\widetilde{z}'(\omega_\beta(E_-))(\omega-\omega_\beta(E_-))+\frac{1}{2}\widetilde{z}''(\omega_\beta(E_-))(\omega-\omega_\beta(E_-))^2+O\left((\omega-\omega_\beta(E_-))^3\right)\,.
\end{align}
In particular, $\widetilde{z}(\omega)$ admits an inverse around $z=E_-$ that is locally analytic  if and only if $\widetilde z'(\omega_\beta(E_-))\not=0$. Thus the smallest edge $E_-$ of the support of $\mu_\alpha\boxplus\mu_\beta$, is the smallest $u\in\R$ such that $\widetilde z'(\omega_\beta(u))=0$. To find the location of the edge, we compute
\begin{align}
 \widetilde z'(\omega)=-F'_{\mu_\alpha}(\omega)+1+\frac{1}{F'_{\mu_\beta}\circ F_{\mu_\beta}^{(-1)}\circ F_{\mu_\alpha}(\omega)}F'_{\mu_\alpha}(\omega)\,.
\end{align}
Hence, choosing $\omega=\omega_\beta(z)$, we get
\begin{align}\label{stinker}
 \widetilde z'(\omega_\beta(z))=-F'_{\mu_\alpha}(\omega_\beta(z))+1+\frac{1}{F'_{\mu_\beta}(\omega_\alpha(z))}F'_{\mu_\alpha}(\omega_\beta(z))\,,
\end{align}
thence, from $\widetilde z'(\omega_\beta(E_-))=0$ we have
\begin{align}\label{stinker bis}
 (F'_{\mu_\alpha}(\omega_\beta(E_-))-1)(F'_{\mu_\beta}(\omega_\alpha(E_-))-1)=1\,.
\end{align}
This proves~\eqref{ja wo ist das edge}.
 
 We move on to proving~\eqref{omega behavior}. From~\eqref{le z} we compute,
\begin{align}
 \widetilde z''(\omega)&=-F''_{\mu_\alpha}(\omega)+\frac{1}{F'_{\mu_\beta}\circ F_{\mu_\beta}^{(-1)}\circ F_{\mu_\alpha}(\omega)}F''_{\mu_\alpha}(\omega)\nonumber\\&\qquad-\frac{1}{(F'_{\mu_\beta}\circ F_{\mu_\beta}^{(-1)}\circ F_{\mu_\alpha}(\omega))^3}\left(F''_{\mu_\beta}\circ F_{\mu_\beta}^{(-1)}\circ F_{\mu_\alpha}(\omega)\right)\cdot( F'_{\mu_\alpha}(\omega))^2\nonumber\,,
\end{align}
and thus by choosing $\omega=\omega_\beta(z)$, we get
\begin{align*}
 \widetilde z''(\omega_\beta(z))=-F''_{\mu_\alpha}(\omega_\beta(z))+\frac{1}{F'_{\mu_\beta}(\omega_\alpha(z))}F''_{\mu_\alpha}(\omega_\beta(z))-\frac{1}{(F'_{\mu_\beta}(\omega_\alpha(z)))^3}F''_{\mu_\beta}(\omega_\alpha(z))\cdot( F'_{\mu_\alpha}(\omega_\beta(z)))^2\,.
\end{align*}
This we can rewrite as
\begin{align}
 \widetilde z''(\omega_\beta(z))=\frac{F''_{\mu_\alpha}(\omega_\beta(z))}{F'_{\mu_\beta}(\omega_\alpha(z))}\big(1-F'_{\mu_\beta}(\omega_\alpha(z))\big)-\frac{1}{(F'_{\mu_\beta}(\omega_\alpha(z)))^3}F''_{\mu_\beta}(\omega_\alpha(z))\cdot\big( F'_{\mu_\alpha}(\omega_\beta(z))\big)^2\,.
\end{align}
Thus choosing~$z=E_-$ and recalling~\eqref{stinker} and~\eqref{stinker bis}, we get
\begin{align}\label{le ses}
 \widetilde z''(\omega_\beta(E_-))=\frac{F''_{\mu_\alpha}(\omega_\beta(E_-))}{F'_{\mu_\beta}(\omega_\alpha(E_-))}\big(1-F'_{\mu_\beta}(\omega_\alpha(E_-))\big)-\frac{F''_{\mu_\beta}(\omega_\alpha(E_-))}{F'_{\mu_\beta}(\omega_\alpha(E_-))}\big(F'_{\mu_\alpha}(\omega_\beta(E_-))-1\big)^2\,.
\end{align}
From~\eqref{hurra}, we directly get
\begin{align}\label{sas 1}
 F'_{\mu_\beta}(\omega_\alpha(E_-))=1+\int_\R\frac{\dd\widehat\mu_\beta(x)}{(x-\omega_\alpha(E_-))^2}>1\,,\qquad F'_{\mu_\alpha}(\omega_\beta(E_-))=1+\int_\R\frac{\dd\widehat\mu_\alpha(x)}{(x-\omega_\beta(E_-))^2}>1\,,
\end{align}
where we used that  $\widehat\mu_\alpha(\R)>0$ and $\widehat\mu_\beta(\R)>0$ as follows from~\eqref{representation mass} and the assumption that $\mu_\alpha$ and $\mu_\beta$ are not single point masses. Moreover, recalling from~\eqref{key to everything} that $\omega_\alpha(E_-)\le E_-^\beta-k_0$, $\omega_\beta(E_-)\le E_-^\alpha-k_0$, we obtain
\begin{align}\label{sas 2}
 F''_{\mu_\beta}(\omega_\alpha(E_-))=\int_\R\frac{\dd\widehat\mu_\beta(x)}{(x-\omega_\alpha(E_-))^3}>0\,,\qquad F''_{\mu_\alpha}(\omega_\beta(E_-))=\int_\R\frac{\dd\widehat\mu_\alpha(x)}{(x-\omega_\beta(E_-))^3}>0\,.
\end{align}
Thus we infer from~\eqref{le ses}, ~\eqref{sas 1} and~\eqref{sas 2} that there are constants $c>0$ and $C<\infty$ such~that
\begin{align}\label{flugi}
-C\le\widetilde z''(\omega_\beta(E_-))\le -c\,. 
\end{align}
Choosing $\omega=\omega_\beta(z)$ (thus $\widetilde z(\omega_\beta(z))=z)$
and using  $\widetilde z'(\omega_\beta(E_-))=0$, $\widetilde z''(\omega_\beta(E_-))<0$ in~\eqref{taylor expansion}, we get
\begin{align}\label{kind nervt}
\omega_\beta(z)-\omega_\beta(E_-)=\sqrt{\frac{-2}{\widetilde z''(\omega_\beta(E_-))}}\sqrt{E_--z}+O(|z-E_-|)\,,
\end{align}
 for $z$ in a neighborhood of $E_-$. The branch of the square root is chosen such that $\im \omega_\beta(z)>0$, $z\in\C^+$.

Next, setting $z=E+\ii\eta$, we observe that~\eqref{flugi} and~\eqref{kind nervt} imply, for $z$ near $E_-$, that
\begin{align}
 \im \omega_\beta(z)\sim\begin{cases} \sqrt{\kappa+\eta}\,,\qquad &\textrm{if}\; E\ge E_-\,, \\ \frac{\eta}{\sqrt{\kappa+\eta}}\,, &\textrm{if}\; E<E_-\,.\end{cases}
\end{align}
This proves the third estimate in~\eqref{omega behavior}. The second estimate is obtained in the same way by interchanging the roles of the indices $\alpha$ and $\beta$. Finally the first estimate follows from~\eqref{needed later on} and the fact that $\omega_\alpha(z)$ and $\omega_\beta(z)$, $z\in\mathcal{E}_0$, are away from the supports of the measure $\mu_\beta$ respectively $\mu_\alpha$ by~\eqref{key to everything} and~\eqref{kind nervt}. This shows~\eqref{omega behavior} and concludes the proof of Lemma~\ref{lemma 7}.
\end{proof}

\begin{rem} From (\ref{kind nervt}) and $m_{\mu_\alpha\boxplus\mu_\beta}(z)=m_{\mu_\alpha}(\omega_\beta(z))$ we get the precise behavior of $m_{\mu_\alpha\boxplus\mu_\beta}(z)$~on~$\mathcal{E}_0$, 
\begin{align*}
 m_{\mu_\alpha\boxplus\mu_\beta}(z)- m_{\mu_\alpha\boxplus\mu_\beta}(E_-)= m_{\mu_\alpha}'(\omega_\beta(E_-))\sqrt{\frac{-2 (\omega_\beta(E_-))}{\widetilde z''(\omega_\beta(E_-))}}\sqrt{E_--z}+O(|z-E_-|)\,,
\end{align*}
and thus by the Stieltjes inversion formula we have the square root behavior for the density of $\mu_\alpha\boxplus\mu_\beta$,
\begin{align}
{\rm d} \mu_{\alpha}\boxplus\mu_\beta (x)\sim \sqrt{x-E_-} \, {\rm d} x\,, \qquad \forall x\in [E_-, E_-+\kappa_0]\,. \label{17080390}
\end{align}
\end{rem}

\begin{cor}\label{le second corollary} Let $\mathcal{E}_0$ be as in~\eqref{def:eps0}. Then the following behaviors hold uniformly for $z\in\mathcal{E}_0$,
\begin{align}\label{mprime}
m'_{\mu_\alpha\boxplus\mu_\beta}(z)&\sim \frac{1}{\sqrt{|z-E_-|}}\,,\qquad & m''_{\mu_\alpha\boxplus\mu_\beta}(z)\sim \frac{1}{|z-E_-|^{3/2}}\,,\\
 \omega'_\alpha(z)&\sim \frac{1}{\sqrt{|z-E_-|}}\,,\qquad& \omega''_\alpha(z)\sim \frac{1}{|z-E_-|^{3/2}}\,,\label{bring it on 2}
 \end{align}
 and
 \begin{align}
 & F'_{\mu_\alpha}(\omega_\beta(z))\sim 1\,,& F''_{\mu_\alpha}(\omega_\beta(z))\sim 1\,,\qquad \qquad& F'''_{\mu_\alpha}(\omega_\beta(z))\sim 1\,.\label{bring it on 3}
 \end{align}
The same estimates hold true when the roles of the subscripts $\alpha$ and $\beta$ are interchanged.
\end{cor}
\begin{proof}
 Having established~\eqref{omega behavior} for the behavior of $\omega_\alpha$ and $\omega_\beta$ around the smallest edge $E_-$, the behaviors in~\eqref{mprime} follow directly. Using the subordination equations~\eqref{le definiting equations}, we note that $F'_{\mu_\alpha}(\omega_\beta(z))\omega'_\beta(z)=F'_{\mu_\beta}(\omega_\alpha(z))\omega'_\alpha(z)=-m_{\mu_\alpha\boxplus\mu_\beta}'(z)/(m_{\mu_\alpha\boxplus\mu_\beta}(z))^{2}$, which together with~\eqref{mprime} imply~\eqref{bring it on 2}. Finally,~\eqref{bring it on 3} follows directly from the analyticity of $F_{\mu_\beta}$ and $F_{\mu_\alpha}$ in neighborhood of $\omega_\alpha(E_-)$, respectively $\omega_\beta(E_-)$. 
\end{proof}

 Let us define a second subdomain $\mathcal{E}_{\kappa_0}$ of $\mathcal{E}$ by setting
\begin{align} \label{def:E kappa0}
\mathcal{E}_{\kappa_0}\deq  
\{z\in\mathcal{E}\,:\, E_-^\alpha+E_-^\beta-1\le \re z-E_-\le \kappa_0\,, 0\le\im z\le \eta_\mathrm{M}\}
\end{align}
with $\kappa_0$ and $\eta_\mathrm{M}$ as in~\eqref{def:eps0}. Note that $\mathcal{E}_0\subset\mathcal{E}_{\kappa_0}\subset\mathcal{E}$. 
We further introduce the functions 
\begin{align}\label{T}
 \mathcal{S}_{\alpha\beta}&\equiv \mathcal{S}_{\alpha\beta}(z)\deq  (F'_{\mu_\alpha}(\omega_\beta(z))-1)(F'_{\mu_\beta}(\omega_\alpha(z))-1)-1\,, \nonumber\\
\mathcal{T}_\alpha&\equiv \mathcal{T}_\alpha(z)\deq \frac{1}{2}\big(F''_{\mu_\alpha}(\omega_\beta(z)) \big(F'_{\mu_\beta}(\omega_\alpha(z))-1\big)^2+
 F''_{\mu_\beta}(\omega_\alpha(z))\big(F'_{\mu_\alpha}(\omega_\beta(z))-1)  \big)\,, \nonumber\\
\mathcal{T}_\beta&\equiv \mathcal{T}_\beta(z)\deq \frac{1}{2}\big(F''_{\mu_\beta}(\omega_\alpha(z)) \big(F'_{\mu_\alpha}(\omega_\beta(z))-1\big)^2+F''_{\mu_\alpha}(\omega_\beta(z))\big(F'_{\mu_\beta}(\omega_\alpha(z))-1)  \big)\,,\qquad z\in\C^+\,.
\end{align}
These functions are essentially the first and second order derivatives of the subordination equations~(\ref{le definiting equations}).  We have the following corollary on the estimates of $m_{\mu_\alpha\boxplus\mu_\beta}$, $\omega_\alpha$, $\omega_\beta$ and also the above functions.

\begin{cor} \label{cor.17080370} Let $\mathcal{E}_{\kappa_0}$ be as in~\eqref{def:E kappa0} and let $\mathcal{E}_0$ be as in~\eqref{def:eps0}. Then
\begin{align}\label{kev1}
\Im m_{\mu_\alpha\boxplus\mu_\beta}(z)\sim \Im \omega_\alpha(z)\sim \Im \omega_\beta(z)\sim   
\begin{cases}\sqrt{\kappa+\eta}\,, \qquad&\textrm{if } E\ge E_-\,,\\
\frac{\eta}{\sqrt{\kappa+\eta}}\,, & \textrm{if } E<E_-\,,
\end{cases}
\end{align}
 and
 \begin{align}\label{kev2}
  \mathcal{S}_{\alpha\beta}(z)\sim\sqrt{\kappa+\eta}
 \end{align}
hold uniformly for $z\in\mathcal{E}_{\kappa_0}$, with $\kappa$ given in~\eqref{17080102}. Moreover, we have
\begin{align}\label{kev3}
\mathcal{T}_\alpha(z)\sim 1\,,\qquad \qquad
 \mathcal{T}_\beta(z)\sim 1\,,
\end{align}
uniformly for $z\in\mathcal{E}_0$, respectively
\begin{align}\label{kev3334}
&|\mathcal{T}_\alpha(z)|\leq C\,,\qquad\qquad |\mathcal{T}_\beta(z)|\leq C\,,
\end{align}
uniformly for $z\in\mathcal{E}_{\kappa_0}$, for some constant $C$.
\end{cor}

\begin{proof}[Proof of Corollary \ref{cor.17080370}]
 Having established~\eqref{omega behavior} for the behavior of $\omega_\alpha$ and $\omega_\beta$ on $\mathcal{E}_0$, the behaviors in~\eqref{kev1},~\eqref{kev2} and~\eqref{kev3} can be checked by elementary computations using Taylor expansions as in the proof of Lemma~\ref{lemma 7}, and the estimates in~\eqref{sas 1} and~\eqref{sas 2} . 
 
 Consider now the complementary domain $\mathcal{E}_{\kappa_0}\setminus \mathcal{E}_0$. Observe that  $\kappa+\eta\sim 1$ in $\mathcal{E}_{\kappa_0}\setminus \mathcal{E}_0$. Hence, we have
 \begin{align}
 \Im m_{\mu_\alpha\boxplus\mu_\beta}(z)=\int_\R\frac{\eta}{(x-E)^2+\eta^2} \,{\rm d} \mu_{\alpha}\boxplus\mu_\beta(x)\sim \eta \label{17080350}
 \end{align}
 uniformly on $\mathcal{E}_{\kappa_0}\setminus \mathcal{E}_0$. Then, from (\ref{a lot of rs}), (\ref{17080350}) and $\Im \omega_\alpha(z)\geq \eta$, $\Im \omega_\beta(z)\geq \eta$, we get 
 \begin{align}
 \Im \omega_\alpha(z)\sim \eta\,, \qquad\qquad   \Im \omega_\alpha(z)\sim \eta\,. \label{17080380}
 \end{align}
Observe that  both estimates in (\ref{kev1}) are of the same order as $\eta$ if $z\in \mathcal{E}_{\kappa_0}\setminus\mathcal{E}_0$.  Hence, we have
(\ref{kev1}).
 
 Next, we show that (\ref{kev2}) can be extended to the whole $\mathcal{E}_{\kappa_0}\setminus \mathcal{E}_0$. Since $\kappa+\eta\sim 1$, it suffices to show that  the left side of (\ref{kev2}) is comparable to $1$ on $\mathcal{E}_{\kappa_0}\setminus \mathcal{E}_0$. We first consider real $z\in [E_-^\alpha+E_-^\beta-1, E_-]$. Using  (\ref{hurra}) and its analogue for $F'_{\mu_\alpha}$, (\ref{stinker bis}), (\ref{kev2}), the monotonicity of $\omega_\alpha(z)$ and $\omega_\beta(z)$ on $(-\infty, E_--\kappa_0]$ (\cf Lemma \ref{lemma 2}), and (\ref{key to everything}),  we  see that 
 \begin{align*}
0\leq (F'_{\mu_\alpha}(\omega_\beta(z))-1)(F'_{\mu_\beta}(\omega_\alpha(z))-1)\leq 1- c\,, \qquad \quad\forall  z\in [E_-^\alpha+E_-^\beta-1, E_--\kappa_0] \,,
 \end{align*}
 for some small constant $c>0$.  Hence, we have 
 \begin{align}
 \big|(F'_{\mu_\alpha}(\omega_\beta(z))-1)(F'_{\mu_\beta}(\omega_\alpha(z))-1)-1\big|\sim 1\,, \qquad\quad\forall  z\in [E_-^\alpha+E_-^\beta-1, E_--\kappa_0]\,.  \label{17080360}
 \end{align}
 Then, (\ref{17080360}) can be extended to all $z=E+\ii\eta$, with $E\in [E_-^\alpha+E_-^\beta-1, E_--\kappa_0]$ and $0\leq \eta\leq \wt{\eta}_0$ for sufficiently small constant $\wt{\eta}_0>0$ by continuity. This together with (\ref{kev2}) gives the estimate in the regime $E\in [E_-^\alpha+E_-^\beta-1, E_-+\kappa_0]$ and $0\leq \eta\leq \eta_0$ after possibly reducing $\eta_0$ to $\wt{\eta}_0$ if $\eta_0>\wt{\eta}_0$. 
 
 It remains to show that the left side of (\ref{kev2}) is proportional to $1$ when $E\in [E_-^\alpha+E_-^\beta-1, E_-+\kappa_0]$ and $\eta_0\leq \eta\leq \eta_\mathrm{M}$.  To this end, we first recall (\ref{hurra}), and observe from   (\ref{le franz}) that 
 \begin{align}
\frac{\Im F_{\mu_\alpha}(\omega_\beta(z))-\Im \omega_\beta(z)}{\Im \omega_\beta(z)}=\int_\R\frac{1}{|x-\omega_\beta(z)|^2}\,\dd\widehat{\mu}_{\alpha}(x)\,. \label{17080378}
 \end{align}
 Hence, using (\ref{hurra}), (\ref{17080378}) and their $F_{\mu_\beta}$ analogues,  we have  
 \begin{align}
 |(F'_{\mu_\alpha}(\omega_\beta(z))-1)(F'_{\mu_\beta}(\omega_\alpha(z))-1)| &\leq \frac{\Im F_{\mu_\alpha}(\omega_\beta(z))-\Im \omega_\beta(z)}{\Im \omega_\beta(z)}\frac{\Im F_{\mu_\beta}(\omega_\alpha(z))-\Im \omega_\alpha(z)}{\Im \omega_\alpha(z)}\nonumber\\
 & = \frac{\Im \omega_\beta(z)-\eta}{\Im \omega_\beta(z)}\frac{\Im \omega_\alpha(z)-\eta}{\Im \omega_\alpha(z)}\leq 1-c\,, \label{17080379}
 \end{align}
 for a strictly positive constant $c$, where in the second step we used the  second equation  in (\ref{le definiting equations}) and in the last step we used that $\eta\geq  \eta_0$ and (\ref{17080380}).   Then, from (\ref{17080379}) we get (\ref{kev2}) in the whole of $\mathcal{E}_{\kappa_0}$.
 
Similarly, the upper bound in  (\ref{kev3334}) follows from (\ref{17080380}), (\ref{key to everything}), the monotonicity in Lemma \ref{lemma 2}, and the continuity of $\omega_\alpha$ and $\omega_\beta$.   Omitting the details, we conclude the proof of Corollary \ref{cor.17080370}.
\end{proof}
At this stage we have completed the first step in the proof of Proposition~\ref{le proposition 3.1}. In the next subsection, we carry out the second step where we translate results obtained so far for $\mu_\alpha$ and $\mu_\beta$ to the measures $\mu_A$ and $\mu_B$ by giving the actual proof of Proposition~\ref{le proposition 3.1}.

\subsection{Proof of Proposition~\ref{le proposition 3.1}}\label{app:stab}

\newcommand{\eps}{\varepsilon}

  Consider the $N$-dependent measures $\mu_A$ and $\mu_B$ while always assuming that they satisfy Assumption~\ref{a. levy distance}. 
Let $\omega_A(z)$ and $\omega_B(z)$ denote the subordination functions associated by~\eqref{170730100} to the measures $\mu_A$ and $\mu_B$. Recall further the definition of the $z$-dependent quantities $\mathcal{S}_{AB}$, $\mathcal{T}_A$ and $\mathcal{T}_B$ in~\eqref{17080110}. 

Recall that $E_-=\inf\,\mathrm{supp}\,\mu_\alpha\boxplus\mu_\beta$. Fix sufficiently small $\eps,\delta>0$ and let the domain $\mathcal{D}$ be defined by
\begin{align}\label{le domain DD}
   \mathcal{D}\deq \mathcal{D}_{\mathrm{in}}\cup \mathcal{D}_{\mathrm{out}}\,,
 \end{align}
 with
 \begin{align*}
  \mathcal{D}_{\mathrm{in}}& \deq\{ z\in \C^+ : |z-E_-|\le \delta\} \cap \{ \im z\ge N^{-1+10\eps}, \re z>E_- - N^{-1+10\eps}  \} \,,\\
  \mathcal{D}_{\mathrm{out}}&\deq  \{ z\in \C^+: |z-E_-|\le \delta\}\cap \{ \re z <E_- - N^{-1+10\eps} \} \,.
\end{align*}

Notice that the bounds on  $A, B$-quantities will be for spectral parameters $z$ that are
separated away from the limiting spectrum (\eg by assuming that $\im z\ge N^{-1+10\eps}$) 
unlike in case of the $\alpha, \beta$-quantities.

\begin{lem}\label{thm.17080401}
Let $\mu_A$, $\mu_B$, $\mu_\alpha$ and $\mu_\beta$ satisfy Assumptions~\ref{a.regularity of the measures} and~\ref{a. levy distance}.
Then there is a constant $C^*$ such that  for any $z\in\mathcal{D}$ we have
\begin{align}\label{omegadiff}
  |\omega_A (z)- \omega_\alpha(z)| +  |\omega_B (z)- \omega_\beta(z)|  &  \le C^* 
       \frac{N^{-1+\eps}} {\sqrt{|z-E_-|}} \le N^{-1/2+\eps}\,,\\
|\mathcal{S}_{AB}(z)|& \sim \sqrt{|z-E_-|}\,,\label{Sdiff}
\end{align}
and
\begin{equation}\label{Tdiff}
  |\mathcal{T}_A(z)| \sim 1\,, \qquad |\mathcal{T}_B(z)| \sim 1\,,
\end{equation}
for $N$ sufficiently large. Moreover, we have for any $z\in\mathcal{D}$ that
\begin{align}
      \im m_{\mu_A\boxplus\mu_B}(z) &\sim  \sqrt{|z-E_-|}\,, \qquad\qquad & &z \in\mathcal{D}_{\mathrm{in}}  \,,\label{mdiff1}\\
       \im m_{\mu_A\boxplus\mu_B}(z) &\lesssim
    \frac{\im z + O(N^{-1+\eps})}{ \sqrt{|z-E_-|}} \,,\qquad & &z\in \mathcal{D}_{\mathrm{out}}\,,\label{mdiff2}
\end{align}
for $N$ sufficiently large. Furthermore, for the imaginary parts the bound \eqref{omegadiff} is, for $N$ sufficiently large, sharpened to
\begin{equation}\label{imbound}
   |\im \omega_A -\im \omega_\alpha| + |\im\omega_B-\im\omega_\beta| \lesssim
    \frac{ (\im \omega_\alpha + \im \omega_\beta) N^{-1+\eps} + \im z}{\sqrt{|z-E_-|}}\,,
\end{equation}
for $z\in \mathcal{D}_{\mathrm{out}}$, $\eta\le N^{-1}$,
which implies that
\begin{equation}\label{bottom}
\inf \mathrm{supp}\, \mu_{A}\boxplus \mu_B \ge  E_-- N^{-1+10\eps}\,.
\end{equation}

Away from the spectral edge we have the following weaker versions of \eqref{Sdiff}, \eqref{Tdiff}: 
\begin{equation}\label{Sdiff1}
  |\mathcal{S}_{AB}(z)| \sim 1\,,
  \end{equation}
\begin{equation}\label{Tdiff1}
  |\mathcal{T}_A(z)| +|\mathcal{T}_B(z)| \le C\,,
\end{equation}
hold uniformly  for any $z$ with $\delta \le |z-E_-|\le C$,  for $N$ sufficiently large.
\end{lem}

\begin{proof} We start by rewriting the subordination equation for $\mu_A$ and $\mu_B$ (\cf~\eqref{le definiting equations} with $\mu_1=\mu_A$, $\mu_2=\mu_B$) as
 \begin{align}\label{le perturbed subo}
  F_{\mu_\alpha}(\omega_B(z))  -\omega_A(z)-\omega_B(z)+z &= r_1(z)\,,\nonumber\\
  F_{\mu_\beta}(\omega_A(z))  -\omega_A(z) - \omega_B(z) +z & = r_2(z)\,,
 \end{align}
 where we introduced
 \begin{align}\label{the r's}
   r_1(z)\deq F_{\mu_\alpha}(\omega_B(z)) - F_{\mu_A}(\omega_B(z))\,,\qquad \qquad r_2(z)\deq F_{\mu_\beta}(\omega_A(z)) - F_{\mu_B}(\omega_A(z))\,.
   \end{align}
   
   We rely on the following local stability result for the system~\eqref{le perturbed subo}.
 \begin{lem}\label{lemma local stability}
  Let $\omega_\alpha(z)$ and $\omega_\beta(z)$ be the unique solutions to~\eqref{le definiting equations} with $\mu_1=\mu_\alpha$ and $\mu_2=\mu_\beta$. Fix $z_0\in\mathcal{D}$. Assume that the functions $\omega_A$, $\omega_B\,:\,\C^+\rightarrow \C^+$ and $r_1$, $r_2\,:\,\C^+\rightarrow \C$ satisfy~\eqref{le perturbed subo} with $z=z_0$. Assume moreover that there is a function $q\equiv q(z)$ such that
 \begin{align}\label{le apriori closeness}
 |\omega_A(z_0)-\omega_\alpha(z_0)|\le q(z_0)\,,\qquad |\omega_B(z_0)-\omega_\beta(z_0)|\le q(z_0)\,,  
 \end{align}
 with $q(z)=o(1)$ and $q(z)/ \mathcal{S}_{\alpha\beta}(z)=o(1)$ as $N\to \infty$, uniformly in $z\in\mathcal{D}$, where $\mathcal{S}_{\alpha\beta}$ is given in~\eqref{T}. Then we have
 \begin{equation}\label{omom}
     |\omega_A(z_0)-\omega_\alpha(z_0)|+ |\omega_B(z_0) - \omega_\beta(z_0)| \le C \frac{|r_1(z_0)|+|r_2(z_0)|}{|\mathcal S_{\alpha\beta}(z_0)|}\,,
 \end{equation}
for $N$ sufficiently large, with some constant $C$ independent of $N$ and $z_0$.

 \end{lem}

\begin{proof}
The proof is similar to the proof of Proposition~4.1 in~\cite{BES15}. We start by Taylor expanding $F_{\mu_\alpha}(\omega_B(z_0))$ to second order around $\omega_\beta(z_0)$,  so that the first equation in~\eqref{le perturbed subo} reads
\begin{align*}
 F_{\mu_\alpha}(\omega_\beta(z_0))+F'_{\mu_\alpha}(\omega_\beta(z_0))\big(\omega_B(z_0)-\omega_\beta(z_0)\big)  -\omega_A(z_0)-\omega_B(z_0)+z_0 &= r_1(z_0)+O(|\omega_B(z_0)-\omega_\beta(z_0)|^2)\,,
\end{align*}
where we used that $F''_{\mu_\alpha}(\omega_\beta(z))$ is uniformly bounded for all $z\in\mathcal{D}$ (see~\eqref{bring it on 3}) and hence so is $F''_{\mu_\alpha}(\widetilde\omega)$ for all $\widetilde\omega\in\C$ in a $q(z_0)$-neighborhood of $\omega_\beta(z_0)$. Subtracting from this last equation the subordination equation $F_{\mu_\alpha}(\omega_\beta(z_0))-\omega_\alpha(z_0)-\omega_\beta(z_0)+z_0=0$, we obtain
\begin{align}\label{le A}
\big(F'_{\mu_\alpha}(\omega_\beta(z_0))-1\big)\Omega_B(z_0)  -\Omega_A(z_0) &= r_1(z_0)+O\left(|\Omega_B(z_0)|^2\right)\,,
 \end{align}
where we introduced $\Omega_B(z_0)\deq \omega_B(z_0)-\omega_\beta(z_0)$ and $\Omega_A(z_0)\deq\omega_A(z_0)-\omega_\alpha(z_0)$. Repeating the above with the roles of $(\alpha,A)$ and $(\beta,B)$ interchanged, we also obtain
\begin{align}\label{le B}
 \big(F'_{\mu_\beta}(\omega_\alpha(z_0))-1\big)\Omega_A(z_0)  -\Omega_B(z_0)  &= r_2(z_0)+O\left(|\Omega_A(z_0)|^2\right)\,.
\end{align}
Combining~\eqref{le A} and~\eqref{le B}, we conclude that
\begin{align}\label{smitty}
|\Omega_A(z_0)|&\le C\frac{|r_1(z_0)|+|r_2(z_0)|}{|\mathcal S_{\alpha\beta}(z_0)|}+C'\frac{|\Omega_A(z_0)|^2}{|\mathcal S_{\alpha\beta}(z_0)|}\,,\nonumber\\
|\Omega_B(z_0)|&\le C\frac{|r_1(z_0)|+|r_2(z_0)|}{|\mathcal S_{\alpha\beta}(z_0)|}+C'\frac{|\Omega_B(z_0)|^2}{|\mathcal S_{\alpha\beta}(z_0)|}\,,
\end{align}
with $\mathcal S_{\alpha\beta}$ given in~\eqref{T}, for some numerical constant $C$ and $C'$ independent of $N$ and $z_0$.

Next, since we are assuming that $|\Omega_A(z_0)|/|\mathcal S_{\alpha\beta}|\le q(z_0)/|\mathcal S_{\alpha\beta}|=o(1)$, and similarly for $\Omega_B(z_0)$, we conclude from~\eqref{smitty} that
\begin{align}
|\Omega_A(z_0)|+|\Omega_B(z_0)|&\le 4C\frac{|r_1(z_0)|+|r_2(z_0)|}{|\mathcal S_{\alpha\beta}(z_0)|}\,,
\end{align}
which, upon renaming the constants was to be proved.
\end{proof}

Returning to the proof of Lemma~\ref{thm.17080401}, we use a continuity argument to prove~\eqref{omegadiff}  for spectral 
parameters $z \in \mathcal{D}$ close to the imaginary axis. First, 
for any $z\in\mathcal{D}$ with $\im z=\eta_{\mathrm{M}}$, for some small
 fixed $\eta_{\mathrm{M}}\sim  1$, the local linear stability result of Lemma 4.2 of~\cite{BES15} shows that
  $ |\omega_A(z)-\omega_\alpha(z)|+ |\omega_B(z) - \omega_\beta(z)| \le 2|r_1(z)|+2|r_2(z)|\le N^{-1+2\epsilon}$, provided that $\im \omega_A(z) - \im z \ge c>0$ and $\im\omega_B(z)-\im z\ge c>0$. These bounds  follow from the subordination equation and
 the representation
\begin{align}
    \im \omega_A(z) - \im z =  \im F_{\mu_A}(\omega_B(z)) - \im \omega_B(z) = (\im z)\int_\R \frac{\dd\widehat\mu_A(x)}{|x-z|^2} \ge c'>0\,,
 \end{align}
 if $\im z\ge \eta_{\mathrm{M}}$, and similarly for $\omega_B$.  Here we also used that $\widehat\mu_A(\R)>0$; see~\eqref{representation mass}.

Having established~\eqref{omegadiff} for any $z\in \mathcal{D}$ with $\Im z= \eta_M$, we fix $E=\Re z$ and reduce the imaginary part of $z$.
Let $\eta_E$ be the smallest number such that ~\eqref{omegadiff} holds for any $z= E+\ii\eta$ with $\eta\in [\eta_E, \eta_M]$; 
our goal is to show that $\eta_E \le N^{-1+10\varepsilon}$. Suppose this is not the case. Use Lemma~\ref{lemma local stability}
with $q(z)\deq N^{-1+5\epsilon}/\sqrt{|z-E_-|}$ and 
with the choice $z_0:= E+\ii(\eta_E- \zeta)\in \mathcal{D}$ for some tiny $\zeta>0$. Since $\omega_A$ and $\omega_\alpha$ are
continuous (even  analytic) at $z_0$,
for a sufficiently small $\zeta>0$, the estimate~\eqref{omegadiff} for $z=E+\ii\eta_E$ guarantees~\eqref{le apriori closeness}
for $z_0$. The other conditions of Lemma~\ref{lemma local stability}, namely 
$q(z)=o(1)$ and $q(z)/ \mathcal{S}_{\alpha\beta}(z)=o(1)$, clearly hold 
since $\mathcal{S}_{\alpha\beta}(z)\sim\sqrt{|z-E_-|}$ by~\eqref{kev2} and $|z-E_-|>N^{-1+10\epsilon}$ for $z\in\mathcal{D}$.

Hence~\eqref{omom} follows by Lemma~\ref{lemma local stability}. We will verify below that
the right hand side of~\eqref{omom}, with this choice of $z_0$, is smaller than the estimate $C^*N^{-1+\eps} /\sqrt{|z_0-E_-|}$ in~\eqref{omegadiff}.
This shows that~\eqref{omegadiff} holds for $z_0$ with imaginary part below $\eta_E$, contradicting to the
definition of $\eta_E$. This  proves~\eqref{omegadiff} for all $z\in \mathcal{D}$.

It remains to bound the  right side of~\eqref{omom} under~\eqref{le apriori closeness}.
For $\delta>0$ in~\eqref{le domain DD} sufficiently small, it follows from~\eqref{key to everything} and~\eqref{kind nervt}, that $\re \omega_\beta(z_0)<E_-^\alpha-k_0/2$, and hence under~\eqref{le apriori closeness} that $\re \omega_B(z_0)<E_-^\alpha-k_0/4$, 
 for $N$ sufficiently large. Let now $I$ be a finite open interval such that $\mathrm{supp}\,\mu_A$, $\mathrm{supp}\,\mu_\alpha\subset I$, $\mathrm{dist}(\re \omega_B(z_0),I)\ge k_0/8$. Let $h\in C^2(\R)$, then integration by parts shows that 
\begin{align}\label{le new guy}
 \Big|\int_\R h(x)\dd\mu_A(x)-\int_\R h(x)\dd\mu_\alpha(x)\Big|= \Big|\int_\R h'(x)(\mathcal{F}_{\mu_A}(x)-\mathcal{F}_{\mu_\alpha}(x))\dd x\Big|\,,
\end{align}
where $\mathcal{F}_{\mu_A}$ and $\mathcal{F}_{\mu_\alpha}$ are the cumulative distribution functions of $\mu_A$ and $\mu_\alpha$. Thus, letting $s:=d_L(\mu_A,\mu_\alpha)$, with $d_L$ the L\'evy distance, we get
\begin{align}\label{le star}
  \Big|\int_\R h(x)(\dd\mu_A(x)-\dd\mu_\alpha(x))\Big|&=\Big| \int_\R \big(h'(x)-h'(x+s)\big)\mathcal{F}_{\mu_A}(x)\dd x+\int_\R h'(x+s)\big(\mathcal{F}_{\mu_A}(x)-\mathcal{F}_{\mu_\alpha}(x+s)\big)\dd x\Big|\nonumber\\
  &\le Cs \Big(\sup_{x\in \R} |h''(x)|+\int_\R|h'(x)|\dd x\Big)\,,  
\end{align}
where we used that $|\mathcal{F}_{\mu_A}(x)-\mathcal{F}_{\mu_\alpha}(x+s)|\le s$ from the definition of the L\'evy distance. Let now $\chi$ be a smooth cut-off function such that $\chi(x)=1$, if $x\in I$, and $\chi(x)=0$, if $\mathrm{dist}(x,I)\ge \frac{k_0}{100}$. Then using~\eqref{le star} with $h(x)=\chi(x)(x-\omega_B(z_0))^{-1}$ we conclude from~\eqref{le star} that, under~\eqref{le apriori closeness},
\begin{align}\label{label estimate for the ST}
 |m_{\mu_A}(\omega_B(z_0))-m_{\mu_\alpha}(\omega_B(z_0))|\le Cs\le C\mathbf{d}\,, 
\end{align}
where $C$ depends only on $k_0$ and with $\mathbf{d}$ given in~\eqref{levy}. Finally, by~\eqref{key to everything} we have
\begin{align*}
 m_{\mu_\alpha}(\omega_\beta(E_-))=\int_{\R}\frac{\dd\mu_\alpha(x)}{x-\omega_\beta(E_-)}\ge c_1>0\,.
\end{align*}
Furthermore, by the analyticity of the Stieltjes transform outside the support of $\mu_\alpha$ 
and under~\eqref{le apriori closeness} we conclude that
 $\Re m_{\mu_\alpha}(\omega_B(z_0))\ge c_1/2$, for sufficiently small $\delta$ in~\eqref{le domain DD}. 
 Similarly, $\Re m_{\mu_A}(\omega_B(z_0))\ge c_1/3$ from~\eqref{label estimate for the ST} as $\mathbf{d}= o(1)$.
 Hence, recalling~\eqref{the r's}, we get under~\eqref{le apriori closeness}, that  for sufficiently large $N$,
\begin{align}
 |r_1(z_0)|=|F_{\mu_A}(\omega_B(z))-F_{\mu_\alpha}(\omega_B(z_0))|&\le\frac{|m_{\mu_A}(\omega_B(z_0))-m_{\mu_\alpha}(\omega_B(z_0))|}{|m_{\mu_A}(\omega_B(z_0))m_{\mu_\alpha}(\omega_B(z_0))|}\nonumber\\&\le C\mathbf{d}\le CN^{-1+\eps}\,. 
\end{align}
Interchanging the roles of $(\mu_A,\mu_\alpha,\omega_B(z_0))$ and $(\mu_B,\mu_\beta,\omega_A(z_0))$, 
we obtain in the same way that $|r_2(z_0)|\le C\mathbf{d} \le CN^{-1+\eps}$, assuming~\eqref{le apriori closeness}.

Finally,  the denominator of~\eqref{omom} can be bounded
as $\mathcal{S}_{\alpha\beta}(z_0) \sim\sqrt{|z_0-E_-|}$ by~\eqref{kev2}. Thus  we have
 $$
     |\omega_A(z_0)-\omega_\alpha(z_0)|+ |\omega_B(z_0) - \omega_\beta(z_0)| \lesssim \frac{\mathbf{d}}{|\mathcal S_{\alpha\beta}(z_0)|} \le
     C^*     \frac{N^{-1+\eps}} {\sqrt{|z_0-E_-|}}\lesssim  N^{-1/2+\eps}\,
  $$
  if $C^*$ is chosen sufficiently large.
  The last step uses that $|z_0-E_-|>N^{-1+10\epsilon}$ since $z_0\in\mathcal{D}$.
This completes the proof of  \eqref{omegadiff}.

From this bound we can compare $\mathcal{S}_{\alpha\beta}$ and $\mathcal{S}_{AB}$,  $\mathcal{T}_\alpha$ and $\mathcal{T}_A$, and $\mathcal{T}_\beta$ and $\mathcal{T}_B$, \eg
\begin{align*}
   |\mathcal{S}_{AB}(z) - \mathcal{S}_{\alpha\beta}(z)|& \le  | (F'_{\mu_A}(\omega_B(z))-1)(F'_{\mu_B}(\omega_A(z))-1) - 
    (F'_{\mu_A}(\omega_\beta(z))-1)(F'_{\mu_B}(\omega_\alpha(z))-1)|\\
&\quad+ | (F'_{\mu_A}(\omega_\beta(z))-1)(F'_{\mu_B}(\omega_\alpha(z))-1) - 
    (F'_{\mu_\alpha}(\omega_\beta(z))-1)(F'_{\mu_\beta}(\omega_\alpha(z))-1)|\\
& \lesssim  |\omega_A(z)-\omega_\alpha(z)|+ |\omega_B(z) - \omega_\beta(z)|  +d \lesssim N^{-1/2+\eps}\,, \qquad z\in\mathcal{D}\,,
\end{align*}
(in the first estimate we used that $F$'s are all regular and in the second we used the same in addition to \eqref{key to everything} and
\eqref{supab}).
Since $|\mathcal{S}_{\alpha\beta}|\ge N^{-1/2+5\eps}$ in this regime, we immediately get \eqref{Sdiff}.
The bounds~\eqref{Tdiff},~\eqref{mdiff1},~\eqref{mdiff2},~\eqref{Sdiff1} are proven exactly in the same way
by showing that the difference between the finite-$N$ quantity and the limiting quantity is
smaller than the size of the limiting quantity given in \eqref{T} and~\eqref{mprime}.

The proof of \eqref{imbound} requires one more argument. 
Outside of the support, \eqref{omegadiff} is not optimal for the imaginary parts. Recall $r_1$ and $r_2$ from~\eqref{the r's}, $z\in\C^+$. Clearly
$$
  |\im r_1(z)|   \le C (\im \omega_B(z))  N^{-1+\eps}, \qquad  |\im r_2(z)|   \le C (\im \omega_A(z))  N^{-1+\eps}\,,\qquad\qquad z\in\mathcal{D}\,,
$$
since 
$$
   \im F_{\mu_\alpha}(\omega_B(z)) = \frac{\im m_{\mu_\alpha}(\omega_B(z))}{| m_{\mu_\alpha}(\omega_B(z))|^2} = \frac { \im\omega_B(z)}{| m_{\mu_\alpha}(\omega_B(z))|^2}
    \int_\R \frac{\dd\mu_\alpha(x)}{|x-\omega_B(z)|^2}\,,
$$
so changing $\alpha$ to $A$ yields a factor $N^{-1+\eps}$ by \eqref{levy}, since $\omega_B(z)$ is away from the support 
of~$\mu_\alpha$ as $\omega_\beta(z)$ is away from the support of $\mu_\alpha$ and we have $|\omega_B(z)-\omega_\beta(z)|\le N^{-1/2+\epsilon}$ by~\eqref{omegadiff}. Taking imaginary parts in~\eqref{le perturbed subo} and using the representations from~\eqref{representation} yields the estimates
\begin{align}\label{one}
     \im \omega_B(z) \int_\R \frac{\dd\widehat \mu_\alpha(x)}{|x-\omega_B(z)|^2}  - \im\omega_A(z) + \im z &= \im r_1(z) = O\big(
      \im \omega_B(z)  N^{-1+\eps}\big)\,,\nonumber\\
     \im \omega_A(z) \int_\R \frac{\dd\widehat \mu_\beta(x)}{|x-\omega_A(z)|^2}  - \im\omega_B(z) + \im z &= \im r_2(z)= O\big(
     \im \omega_A(z)  N^{-1+\eps}\big)\,,
\end{align}
$z\in\mathcal{D}$. Similarly, starting from the subordination equations for $\mu_\alpha$ and $\mu_\beta$, we have
\begin{align}\label{one1}
     \im \omega_\beta(z)  \int_\R \frac{\dd\widehat \mu_\alpha(x)}{|x-\omega_\beta(z)|^2}  - \im\omega_\alpha(z) + \im z &= 0\,,\nonumber\\
     \im \omega_\alpha(z) \int_\R \frac{\dd\widehat \mu_\beta(x)}{|x-\omega_\alpha(z)|^2} - \im\omega_\beta(z) + \im z &= 0\,.
\end{align}
In fact, using~\eqref{omegadiff} we can change $\omega_B$ to $\omega_\beta$ and $\omega_A$ to $\omega_\alpha$ in the integrands and error terms in~\eqref{one},  to get
\begin{align}\label{one2}
     \im \omega_B(z) \int_\R \frac{\dd\widehat \mu_\alpha(x)}{|x-\omega_\beta(z)|^2}  - \im\omega_A(z) + \im z &= O\big(
      \im \omega_\beta(z)  N^{-1+\eps}\big)+O\Big(\im \omega_\beta(z) \frac{N^{-1+\epsilon}}{\sqrt{|z-E_-|}}\Big)\,,\nonumber\\
     \im \omega_A(z)  \int_\R \frac{\dd\widehat \mu_\beta(x)}{|x-\omega_\alpha(z)|^2} - \im\omega_B(z) + \im z &= O\big(
     \im \omega_\alpha(z)  N^{-1+\eps}\big)+O\Big(\im \omega_\alpha(z) \frac{N^{-1+\epsilon}}{\sqrt{|z-E_-|}}\Big)\,,
\end{align}
$z\in\mathcal{D}$. Subtracting \eqref{one1} from \eqref{one2} and using that for very small 
$\eta$ the determinant of the resulting linear system is very close to 
  $\mathcal S_{\alpha\beta}(z)\sim \sqrt{|z-E_-|}$, $z\in\mathcal{D}$,
from \eqref{Sdiff}, we get
\begin{align*}
 |\im\omega_A(z)-\im\omega_\alpha(z)|+|\im\omega_B(z)-\im\omega_\beta(z)|\lesssim\frac{\im \omega_\alpha(z)+\im\omega_\beta(z)}{\sqrt{|z-E_-|}}N^{-1+\epsilon}+ \frac{\im\omega_\alpha(z)+\im\omega_\beta(z)}{|z-E_-|}N^{-1+\epsilon}	\,.
\end{align*}
Hence, recalling the bound in~\eqref{kev1} for $\im \omega_\alpha(z)$ and $\im\omega_\beta(z)$, we get
for $z\in \mathcal{D}_{\mathrm{out}}$ with $\eta\le N^{-1}$,
\begin{align*}
|\im\omega_A(z)-\im\omega_\alpha(z)|+|\im\omega_B(z)-\im\omega_\beta(z)|&\lesssim \frac{\im \omega_\alpha(z)+\im\omega_\beta(z)}{\sqrt{|z-E_-|}}N^{-1+\epsilon}+ \frac{\eta}{\sqrt{|z-E_-|}}\frac{N^{-1+\epsilon}}{{|z-E_-|}}	\nonumber\\
&\lesssim \frac{\im \omega_\alpha(z)+\im\omega_\beta(z)}{\sqrt{|z-E_-|}}N^{-1+\epsilon}+ \frac{\eta}{\sqrt{|z-E_-|}}\,.
\end{align*}
This proves~\eqref{imbound}.

Finally, to prove \eqref{bottom}, let $z= x+ \ii\eta$ with $x\le E_-- N^{-1+10\eps}$. At a distance of at least $N^{-1}$ below $E_-$,~we~get
\begin{align*}
   \im m_{\mu_\alpha\boxplus\mu_\beta} (z)= \im z \int_\R \frac{\dd\mu_\alpha\boxplus\mu_\beta(x)}{|x-z|^2} \le  C N\im z\,.
   \end{align*}
Moreover from $m_{\mu_\alpha\boxplus\mu_\beta}(z)= m_\alpha(\omega_\beta(z))$ we have $\im m_\alpha (\omega_\beta(z)) \sim \im \omega_\beta(z)$
since $\omega_\beta(z)$ is away from the support of $\mu_\alpha$. The same holds for $\omega_\alpha(z)$, so we get $\im  \omega_\alpha(z)+ \im \omega_\beta(z) \le C N\im z$.
Taking $\eta\searrow 0$, we note that the right hand side of \eqref{imbound} goes to zero. Thus
we get $\im \omega_A(x) = \im \omega_B (x) =0$. Since $\im m_{\mu_A\boxplus \mu_B}(z) \sim \im \omega_A(z)$
in this regime,~$x$ cannot lie in the support of $\mu_A\boxplus \mu_B$. This proves~\eqref{bottom}.
\end{proof}

Recall that $\gamma_j$ denoted the $j$-th  $N$-quantile of $\mu_\alpha\boxplus\mu_\beta$
from \eqref{quantile} and similarly let $\gamma^*_j$ denote the $j$-th  $N$-quantile of 
 $\mu_A\boxplus\mu_B$, \ie these are the smallest  numbers $\gamma_j$ and $\gamma_j^*$ such that
$$
   \mu_\alpha\boxplus\mu_\beta\big( (-\infty, \gamma_j]\big)
    = \mu_A\boxplus\mu_B\big( (-\infty, \gamma_j^*]\big)   = \frac{j}{N}.
    $$

\begin{lem}[Rigidity]  Suppose Assumptions~\ref{a.regularity of the measures} and \ref{a. levy distance} hold. Fix some sufficiently small $0<c<1$.
 Then we have, for any small $\epsilon>0$, the rigidity bound
\begin{equation}\label{rigi2}
    |\gamma_j -\gamma_j^*|\le   j^{-1/3} N^{-\frac{2}{3}+\eps}, \qquad j\in\llbracket 1, cN\rrbracket\,,
\end{equation}
for $N$ sufficiently large depending on $\epsilon$ and $c$.

Under the additional Assumption~\ref{a. rigidity entire spectrum} we have the rigidity estimate for all quantiles, \ie
\begin{equation}\label{rigi}
   |\gamma_j -\gamma_j^*|\le  \min\{ j^{-1/3}, (N+1-j)^{-1/3}\} N^{-\frac{2}{3}+\eps}, \qquad j\in\llbracket 1,N\rrbracket\,.
\end{equation}
\end{lem}

\begin{proof}
The proof of these rigidity results are fairly straightforward  from the information collected so far,
by using standard arguments to translate the closeness of Stieltjes transform of
two measures into closeness of their quantiles. We will just outline the argument.
Recall the domain $\mathcal{E}_{\kappa_0}$ from  \eqref{def:eps0}.

First, we establish that there are at most $N^\eps$ $\gamma_j$-quantiles as well as $N^\eps$
 $\gamma_j^*$-quantiles  in an $N^{-2/3+\eps}$ vicinity of $E_-=\inf\,\mathrm{supp}\,\mu_\alpha\boxplus\mu_\beta$.
This fact is immediate for the $\gamma_j$ quantiles since their distribution is given by the regular square root law, see
\eqref{17080390}. For the $\gamma_j^*$-quantiles, we know 
from \eqref{bottom}  that 
 $\gamma_1^* \ge E_- - N^{-1+10\eps}$. 
We compute from \eqref{mdiff1}
\begin{align*}
  \frac{j}{N} &= \int_{-\infty}^{\gamma_j^*} {\rm d} \mu_{A}\boxplus \mu_B(x) = \int_{E_--N^{-1+10\eps}}^{\gamma_j^*} \dd\mu_{A}\boxplus \mu_B(x) 
  \le C\int_{E_--N^{-1+10\eps}}^{\gamma_j^*} \im m_{\mu_A\boxplus\mu_B}(x+\ii N^{-1+10\eps}) {\rm d}x\\
 &\le C\int_{E_--N^{-1+10\eps}}^{\gamma_j^*} \big[ |x-E_-| + N^{-1+10\eps}\big]^{1/2} {\rm d}x \le C |\gamma_j^*-E_-|^{3/2} + CN^{-1+10\eps}|\gamma_j^*-E_-|\,,
\end{align*}
 which means that 
$$
|\gamma_j^*-E_-|\ge c\Big(\frac{j}{N}\Big)^{2/3}\,,
$$
with some positive constant $c>0$. So we have
\begin{equation}\label{first}
   \gamma_j^* \ge E_- + c N^{-2/3+\eps}, \qquad \mbox{if} \quad  j\ge c N^{3\eps/2},
\end{equation}
and note that the condition 
$ j\ge  cN^{3\eps/2} $ is equivalent to $\gamma_j \ge  E_-+ cN^{-2/3+\eps}$. In the other direction~we~use
$$
  \int_{E_--N^{-1+10\eps}}^{\gamma_j^*} \dd\mu_{A}\boxplus \mu_B(x)
  \ge c\int_{E_--N^{-1+10\eps}}^{\gamma_j^*} \im m_{\mu_A\boxplus \mu_B}(x+\ii N^{-1+10\eps})\, {\rm d}x
$$
if $|\gamma_j^* - E_-|\gg N^{-1+10\eps}$. Using again \eqref{mdiff1} we get 
$$
 \frac{j}{N} \ge c |\gamma_j^*-E_-|^{3/2}, \qquad \mbox{\ie} \qquad \gamma_j^* \le E_-+ C \Big(\frac{j}{N}\Big)^{2/3} \qquad \forall j,
$$
since this latter bound also holds in the case, when $|\gamma_j^* - E_-|\gg N^{-1+10\eps}$ is not satisfied.

Thus we have established 
\begin{equation}\label{ed}
     |\gamma_j - \gamma_j^*|  \le |\gamma_j -E_-| + |\gamma_j^*-E_-| \le C N^{-2/3+\eps}, \qquad \mbox{whenever}\quad 
      \gamma_j \le E_- + N^{-2/3+\eps}.
\end{equation}

From the continuity of the free convolution (Proposition 4.13 of \cite{BeV93}) and
the condition \eqref{levy} we get
$$
  {\rm d}_L (\mu_{A}\boxplus \mu_B, \mu_\alpha \boxplus \mu_\beta)\le
   {\rm d}_L(\mu_A, \mu_\alpha) + {\rm d}_L(\mu_B, \mu_\beta)\le N^{-1+\epsilon}\,.
$$
On the other hand, the definition of the L\'evy distance and the boundedness of  the density of
$\mu_{\alpha} \boxplus \mu_\beta$  below  $E_-+\kappa_0$ (see \eqref{17080390}) 
directly imply that 
\begin{align}\label{cumulative}
  \big| \mu_A \boxplus \mu_B \big( (-\infty, x) \big) -  \mu_\alpha \boxplus \mu_\beta \big( (-\infty, x) \big) \big|\le CN^{-1+\eps}
\end{align}
holds for any $x\le E_-+\kappa_0$.  Together with \eqref{ed}, this estimate immediately implies
the bound \eqref{rigi2}.

For the proof of \eqref{rigi}, we note that $(ii')$  and $(v')$ of Assumption~\ref{a. rigidity entire spectrum}
guarantee that near the upper edge of the support of $\mu_\alpha\boxplus\mu_\beta$
a similar rigidity statement holds as \eqref{rigi2}. Finally, $(ii')$ of Assumption~\ref{a. rigidity entire spectrum}
together with the continuity and boundedness of the density  of $\mu_\alpha\boxplus\mu_\beta$
(see \eqref{17080326}) imply that the density has a positive lower and upper bound away the two extreme edges
of its support. These information together with \eqref{levy} 
are sufficient to conclude that \eqref{cumulative} hold uniformly for any $x\in\R$. The corresponding
result \eqref{rigi} for the quantiles follows immediately. 
\end{proof}

\begin{proof}[Proof of Proposition~\ref{le proposition 3.1}] First, on the domain $\mathcal{D}$, $(i)$ of Proposition~\ref{le proposition 3.1}  follows from (\ref{omegadiff}), (\ref{key to everything}), the assumption (\ref{supab}) and also the continuity of $\omega_\alpha$ and $\omega_\beta$. In the complementary domain $\mathcal{D}_\tau(\eta_{\rm m}, \eta_\mathrm{M})\setminus \mathcal{D}$,  we first prove  (\ref{17020503}). Using the equations $m_{\mu_A\boxplus\mu_B}=m_{\mu_A}(\omega_B)= m_{\mu_B}(\omega_A)$, we see that the upper bounds on $\omega_A$ and $\omega_B$ follow from the fact that $|m_{\mu_A\boxplus\mu_B}(z)|\geq c$,  which can easily be derived from the rigidity~(\ref{rigi2}). For (\ref{17020502}), we further split into two regimes. In the regime $\eta\geq \eta_0$ for some small $\eta>0$, we  use  the fact $\Im \omega_A(z), \Im \omega_B(z)\geq \eta$ directly. In the regime $\eta\leq \eta_0$, we use   the continuity of $\omega_A$ and $\omega_B$, and also the monotonicity of the $ \omega_A (u)$ and $ \omega_B (u)$ for $ u\in (-\infty, E_--\delta]$  which can be  proved similarly to the monotonicity of $ \omega_\alpha (u)$ and $ \omega_\beta (u)$ (\cf (\ref{le C15})).

Similarly, on the domain $\mathcal{D}$, Proposition~\ref{le proposition 3.1} $(ii)$ follows from (\ref{mdiff1}) and (\ref{mdiff2}) directly. In the complementary domain $\mathcal{D}_\tau(\eta_{\rm m}, \eta_\mathrm{M})\setminus \mathcal{D}$, we apply again the rigidity result (\ref{rigi2}) to conclude the proof.

Statement $(iii)$ follows from~(\ref{Sdiff}), (\ref{Tdiff}), (\ref{Sdiff1}) and~(\ref{Tdiff1}). 

Finally, to prove item $(iv)$, we differentiate the subordination equations~\eqref{le definiting equations} with respect to $z$~to~get
\begin{align*}
\left(
\begin{array}{ccc}
1 & 1-F_A'(\omega_B(z))\\
1-F_B'(\omega_A(z)) & 1
\end{array}
\right) \left(\begin{array}{cc}
\omega_A'(z)\\
\omega_B'(z)
\end{array}
\right)
=\left(\begin{array}{cc}
1\\
1
\end{array}
\right)\,,
\end{align*}
with the shorthand $F_A\equiv F_{\mu_A}$, $F_B\equiv F_{\mu_B}$.
Hence, 
\begin{align*}
\left(\begin{array}{cc}
\omega_A'(z)\\
\omega_B'(z)
\end{array}
\right)=-\mathcal{S}^{-1}(z)  \left(\begin{array}{cc}
F'_A(\omega_B(z))\\
F'_B(\omega_A(z))
\end{array}
\right),
\end{align*}
where $\mathcal{S}\equiv \mathcal {S}_{AB}$. Using (\ref{17080110}) and (\ref{17020502}) and~\eqref{17080121}, we directly get the first two estimates in~\eqref{le lipschitz stuff}, since $F'_A(\omega_B(z))$ and $F_B'(\omega_A(z))$ are uniformly bounded on $\mathcal{D}_\tau(\eta_{\rm m}, \eta_\mathrm{M})$ by~\eqref{17020502}.

  Next, from the definition of $\mathcal{S}(z)$ in (\ref{17080110}), we observe that 
\begin{align}
|\mathcal{S}'(z)|=\Big|F''_B(\omega_A)(F'_A(\omega_B)-1)\omega_A'(z)+F''_A(\omega_B)(F'_B(\omega_A)-1) \omega_B'(z)\Big|\leq C |\mathcal{S}^{-1}(z)  |\,, \label{17072610}
\end{align}
where in the second step we used (\ref{17020502}) and the first two estimates in (\ref{le lipschitz stuff}). Hence, by~\eqref{17080121} we get the third estimate in~\eqref{le lipschitz stuff} and statement $(iv)$ is proved. This finishes the proof of Proposition~\ref{le proposition 3.1}. 
\end{proof}

\section{General structure  of the proof} \label{sec:general}
\subsection{Partial randomness decomposition} In this subsection, we recall  the partial randomness decomposition of the Haar unitary matrix used in~\cite{BES15b}, which will often be used below.

 Let $\mathbf{u}_i=(u_{i1}, \ldots, u_{iN})^T$ be the $i$-th column of the Haar distributed matrix $U$.  Let $\theta_i$ be the argument of $u_{ii}$.  The following partial randomness decomposition of $U$ is taken from~\cite{DS87} (see also \cite{Mezzadri}): For any $i\in \llbracket 1, N\rrbracket$, we can write
 \begin{align}
 U=-\e{\mathrm{i}\theta_i}R_i U^{\la i\ra}\,, \label{17072510}
 \end{align}
 where $U^{\la i\ra}$ is a unitary block-diagonal matrix whose $(i,i)$-th entry equals $1$, and its $(i,i)$-minor is Haar distributed on $\mathcal{U}(N-1)$. Hence, $U^{\la i\ra}\mathbf{e}_i=\mathbf{e}_i$ and $\mathbf{e}_i^* U^{\la i\ra}=\mathbf{e}_i^*$, where $\mathbf{e}_i$ is the $i$-th coordinator vector.  Here $R_i$ is a reflection matrix, defined as 
 \begin{align}
 R_i\deq  I-\mathbf{r}_i\mathbf{r}_i^*\,, \label{17072593}
 \end{align}
 where 
 \begin{align}
 \mathbf{r}_i\deq\sqrt{2} \frac{\mathbf{e}_i+\e{-\mathrm{i}\theta_i}\mathbf{u}_i}{\|\mathbf{e}_i+\e{-\mathrm{i}\theta_i} \mathbf{u}_i\|}\,.  \label{17072517}
 \end{align}
 Using $U^{\la i\ra}\mathbf{e}_i=\mathbf{e}_i$ and (\ref{17072510}), we see that 
 \begin{align}
 \mathbf{u}_i=U\mathbf{e}_i= -\e{\mathrm{i}\theta_i}R_i\mathbf{e}_i\,.  \label{17072530}
 \end{align}
 Hence, $R_i=R_i^*$ is actually the Householder reflection (up to a sign) sending $\mathbf{e}_i$ to $-\e{-\mathrm{i}\theta_i}\mathbf{u}_i$.  With the decomposition in (\ref{17072510}), we can write 
 \begin{align*}
 H= A+\wt{B}=A+R_i \wt{B}^{\la i\ra} R_i\,, 
 \end{align*}
 where we introduced the notations 
 \begin{align*}
 \wt{B}\deq UBU^*, \qquad \wt{B}^{\la i\ra}\deq U^{\la i\ra} B (U^{\la i\ra})^*\,. 
 \end{align*}
 Observe that $\wt{B}^{\la i\ra}\mathbf{e}_i=b_i\mathbf{e}_i$ and $\mathbf{e}_i^* \wt{B}^{\la i\ra}=b_i\mathbf{e}_i^*$. Clearly, $\wt{B}^{\la i\ra}$ is independent of $\mathbf{u}_i$. 
 
 It is known that $\mathbf{u}_i\in S_{\mathbb{C}}^{N-1}\deq\{\mathbf{x}\in \mathbb{C}^N\,:\, \mathbf{x}^*\mathbf{x}=1\}$ is a uniformly distributed complex vector, and there exists a Gaussian vector $\wt{\mathbf{g}}_i\sim \mathcal{N}_{\mathbb{C}}(0, N^{-1}I_N)$ such that 
 \begin{align*}
 \mathbf{u}_i=\frac{\wt{\mathbf{g}}_i}{\|\wt{\mathbf{g}}_i\|}\,.
 \end{align*}
 We then further introduce the notations
 \begin{align}
 \mathbf{g}_i\deq\e{-\mathrm{i}\theta_i}\wt{\mathbf{g}}_i\,, \qquad \quad\mathbf{h}_i\deq\frac{\mathbf{g}_i}{\|\mathbf{g}_i\| }=\e{-\mathrm{i}\theta_i} \mathbf{u}_i\,, \qquad\quad\ell_i\deq\frac{\sqrt{2}}{\| \mathbf{e}_i+\mathbf{h}_i\|}\,.  \label{17072590}
 \end{align}
 Observe that the components $g_{ik}$ of $\mathbf{g}_i$ are independent. Moreover, for $k\neq i$, $g_{ik}\sim N_{\mathbb{C}}(0, \frac{1}{N})$ while $g_{ii}$ is a $\chi$-distributed random variable with $\mathbb{E}g_{ii}^2=\frac{1}{N}$.  With the above notations, we can write $\mathbf{r}_i$ in (\ref{17072517})~as
 \begin{align}
 \mathbf{r}_i=\ell_i (\mathbf{e}_i+\mathbf{h}_i)\,.  \label{170725100}
 \end{align}
In addition, using (\ref{17072530}) and the fact $R_i^2=I$, we have 
\begin{align}
R_i\mathbf{e}_i=-\mathbf{h}_i\,,\qquad \qquad R_i\mathbf{h}_i=-\mathbf{e}_i\,, \label{17072573}
\end{align}
which also imply 
\begin{align}
\mathbf{h}_i^* \wt{B}^{\la i\ra}R_i= -\mathbf{e}_i^* \wt{B}\,,\qquad \qquad \mathbf{e}_i^* \wt{B}
^{\la i\ra}R_i= -b_i\mathbf{h}_i^*=-\mathbf{h}_i^*\wt{B}\,.  \label{170726105}
\end{align}
Here, in the first equality of the second equation we used that $\mathbf{e}_i^* \wt{B}^{\la i\ra}=b_i\mathbf{e}_i$. We introduce the vectors 
\begin{align*}
\mathring{\mathbf{g}}_i\deq \mathbf{g}_i-g_{ii}\mathbf{e}_i\,,\qquad\qquad \mathring{\mathbf{h}}_i\deq \frac{\mathring{\mathbf{g}}_i}{\|\mathbf{g}_i\|}\,,
\end{align*}
where the $\chi$-distributed  variable $g_{ii}$ is kicked out.

\subsection{Summary of the proof route} 

In this subsection, we summarize the main route of the proof. 
While the final goal of the local law is to understand $G_{ii}$, $i\in\llbracket1,N\rrbracket$, and its averaged version, we work with several auxiliary quantities first.  To understand their origin, it is useful to review the structure of our previous proofs
of the local laws in the bulk  \cite{BES15b,BES16}.  We first  introduce the following control parameters
\begin{align}
\Psi\equiv \Psi(z)\deq \sqrt{\frac{1}{N\eta}}\,,\qquad\qquad\Pi\equiv \Pi(z)\deq \sqrt{\frac{\Im m_H}{N\eta}}\,, \label{17012101}
\end{align}
for $z\in\C^+$, where $\eta=\im z$. In~\cite{BES15b}, we investigated two main quantities:
 \begin{align}
S_i\equiv S_i(z)\deq  \mathbf{h}_i^* \wt{B}^{\la i\ra} G\mathbf{e}_i\,,\qquad \qquad T_i\equiv T_i(z)\deq \mathbf{h}_i^* G\mathbf{e}_i\,.  \label{17072580}
\end{align}
In particular we showed that
$$
      S_i = \frac{z-\omega_B(z)}{a_i-\omega_B(z)} + O_\prec (\Psi)\,, \qquad\qquad T_i = O_\prec (\Psi)\,, 
      $$
by performing integration by parts in the $\mathbf{h}_i^*$ variable.
Using the identity  
\begin{align*}
     G_{ii} = \frac{1- (\wt{B} G)_{ii}}{a_i-z}
\end{align*}
 and that  
\begin{align*}
(\wt{B} G)_{ii} &= \mathbf{e}_i^* R_i \wt{B}^{\la i\ra}  R_i G\mathbf{e}_i = 
- \mathbf{h}_i^*\wt{B}^{\la i\ra}  R_i G\mathbf{e}_i = - S_i +    \mathbf{h}_i^*\wt{B}^{\la i\ra}  \mathbf{r}_i   \mathbf{r}_i^*G\mathbf{e}_i\\
&=  - S_i +    \ell_i^2 ( \mathbf{h}_i^*\wt{B}^{\la i\ra}  \mathbf{h}_i +  b_ih_{ii})  (G_{ii}+T_i)\,,
\end{align*}
we obtained the entry-wise local law for $G_{ii}$ from a precise control on $S_i$ and $T_i$.

Technically $S_i$ is a better quantity than $G_{ii}$ to handle since integration by parts can be directly applied to it.
However, along the calculation the quantity $T_i$ appeared and a second integration by parts was needed to control it.
We obtained a closed system of equations on the expectations of $S_i$ and $T_i$  (see (6.23)--(6.24) of \cite{BES15b})
from which the entry-wise local law in the bulk followed.  

To obtain the law for the normalized trace of $G$ in \cite{BES16}, we performed fluctuation averaging, but again not for $G_{ii}$ directly. We considered
averages  (with arbitrary bounded weights $d_i$) of the quantity
\begin{align*}
      Z_i \deq Q_i + G_{ii} \Upsilon\,,
\end{align*}
where we defined
\begin{align}
 Q_i&\equiv  Q_i(z)\deq (\wt{B}G)_{ii}\ntr G-G_{ii} \ntr \wt{B}G\,, \label{17021701}\\
\Upsilon&\equiv \Upsilon(z)\deq  \ntr \wt{B}G-(\ntr \wt{B}G)^2+\ntr G\ntr \wt{B} G\wt{B}\,. \label{17020511}
\end{align}
From the entry-wise laws it is clear that $|Q_i|, |\Upsilon| \prec \Psi$, and now we improve these bounds, at least 
in averaged sense in case of  $Q_i$.
Notice that $Q_i$ is the most ``symmetric" quantity, in particular $\sum_i Q_i =0$,  but
technically it is not  convenient  to perform a high moment estimate for its weighted average,  $\frac{1}{N}\sum_i d_i Q_{i}$.
The reason is that one step of integration by parts generates an additional term, $G_{ii} \Upsilon$,  which is hard to control
directly. So instead of averaging $Q_i$, in \cite{BES16} we included a counter term, \ie we averaged $Z_i$ instead.
We first proved that the average is one order better, \ie 
\begin{align}\label{Qav}
   \Big| \frac{1}{N}\sum_{i=1}^N d_i Z_i\Big| \prec \Psi^2
\end{align}
 with arbitrary deterministic weights $|d_i|\le 1$.  
Then, using~\eqref{Qav} with $d_i \equiv 1$, we obtained $|\Upsilon| \prec \Psi^2$. Thus {\it a posteriori} we showed that the 
counter term $G_{ii} \Upsilon$ is irrelevant for estimates of order $\Psi^2$ and we obtained the same bound \eqref{Qav}
for $Q_i$ as well. Finally, the bounds on the average of $Q_i$ with careful  choices of the weights $d_i$ 
and using the algebraic identities between $G$ and $\wt{B} G$ yielded the averaged law for $G_{ii}$
with the optimal $O_\prec(\Psi^2)$ error.

All results in  \cite{BES15b,BES16} concerned the bulk. It is well known from the analogous 
results for Wigner matrices that the edge analysis is more difficult. The main reason is that the 
corresponding Dyson equation, the subordination equation in the current model,  is unstable at the spectral edge, hence more precise
estimates are necessary for the error terms. Theoretically, all error terms involving $\Psi = \frac{1}{\sqrt{N\eta}}$ should be
improved by a factor of $\sqrt{\im m}$, where we set $m\deq m_{\mu_A\boxplus\mu_B}$. This factor reflects that the density of states is small at the edge (at a square root edge
we have $\im m(z) \sim \sqrt{\kappa+ \eta}$, where $\eta =\im z$ and $\kappa$ is the distance of $\re z$ to the edge).
This improvement exactly compensates for the bound of order $(\kappa+\eta)^{-1/2}$ on the inverse of the
linearization of the subordination equation near the edge.
However, this improvement is quite complicated to obtain and the method in \cite{BES16} is not sufficient.

In this paper we present a new strategy to obtain the stronger bound. To prepare for the higher accuracy, already in 
the entry-wise law we work with two new quantities $P_i$ and $K_i$ instead of $S_i$ and $T_i$. They are defined as
\begin{align}
P_{i}&\equiv  P_i(z)\deq (\wt{B}G)_{ii} \ntr G-G_{ii} \ntr (\wt{B} G)+(G_{ii}+T_{i})\Upsilon\,,  \label{17011301}\\
 K_i&\equiv  K_i(z)\deq T_{i}+ (b_i T_{i}+(\wt{B}G)_{ii}) \ntr G-(G_{ii}+T_i) \ntr (\wt{B}G)\,. \label{17072420}
\end{align}
We recognize that $P_i = Q_i + (G_{ii}+T_i) \Upsilon = Z_i + T_i\Upsilon$, \ie we included an additional counter term $T_i\Upsilon$
to the previous $Z_i$. While {\it a posteriori} this counter term turns out to be irrelevant, it is necessary in order to perform the
integration by parts more precisely.
Similarly, 
\begin{align}\label{ki}
   K_i= \big(1 + b_i \ntr G- \ntr (\wt{B}G)\big) T_i + Q_i\,,
\end{align}
\ie $K_i$ is a linear combination of  $T_i$ and $Q_i$, it is nevertheless easier to work with $K_i$.

 The proof of the estimates of the aforementioned quantities is divided into three parts. All these parts
are performed for a fixed $z$ and under the condition that $G_{ii}$'s and $T_i$'s satisfy a weak a priori bound (\cf (\ref{17020501}) ).
This condition will be verified later in Section \ref{s.weak law}.

In the first part (Section~\ref{s. Entrywise estimate}) we obtain entry-wise bounds of the form
\begin{align}\label{entry}
     |K_i|, |Q_i|, |T_i|, |P_i| \prec \Psi, \qquad \mbox{as well as}\qquad |\Upsilon| \prec \Psi\,;
\end{align}
see Proposition \ref{pro.17020310}.  Notice that the estimates are still in terms of $\Psi =\frac{1}{\sqrt{N\eta}}$ without the improving factor 
$\sqrt{\im m}$. These results would be possible to derive directly  from the estimates in \cite{BES15b} 
by operating with $S_i$ and $T_i$, we nevertheless use the new quantities, since the formulas derived
along the entry-wise bounds will be used in the improved bounds later. 

There is yet another reason for introducing the new quantities $P_i$ and $K_i$, namely that in the current paper
we have also changed the strategy concerning the entry-wise laws. In \cite{BES15b}, a precursor to \cite{BES16},
we first proved entry-wise laws by deriving a system of equations
for the expectation values (of $S_i$ and $T_i$), complemented with concentration inequalities
to enhance them to high probability bounds. For the improved bound on averaged quantities
high moment estimates were performed only in  \cite{BES16}, using the entry-wise law as an input.
In the current paper we organize the proof in a more straightforward way, similarly to
\cite{BES16b}.
We bypass the fairly complicated concentration
argument leading to the entry-wise law in \cite{BES15b} and we rely on high moment estimates directly
even for the entry-wise law. This strategy is not only conceptually cleaner but also allows us to use  essentially
the same calculations for the entry-wise and the averaged law. The estimates of many error terms are
shared in the two parts of the proofs; in case of some other estimates it will be sufficient to point out the
necessary improvements.
However, high moment estimates require to consider more carefully chosen quantities. For example, no direct high moment
estimates are possible for $S_i$ since it is even not a small quantity. But high moment estimates for the smaller quantities $T_i$ and $Q_i$ produce additional terms that are difficult to handle. It turns out that 
the carefully chosen counter terms in $P_i$ and $K_i$ make them suitable for performing  high moment bounds.

More precisely, in the first step we compute the high moments of $K_i$ and conclude that $|K_i|\prec \Psi$.
In the second step we prove a high moment bound for  $P_i =  Q_i+(G_{ii} +T_i)\Upsilon$, \ie prove $|P_i|\prec \Psi$.
In the third step we average this bound and conclude $|\Upsilon| \prec \Psi$, which in turn yields that $|Q_i|\prec \Psi$.
Finally, from \eqref{ki} we conclude that $|T_i|\prec \Psi$. This proves \eqref{entry} and completes the entry-wise bounds.

In the second part of the proof (Section \ref{s. rough bound})  we derive a rough bound on the averaged quantities. 
We will focus on 
$\frac{1}{N}\sum_i d_i Q_i$,
since $Q_i$ is the most fundamental quantity.
 Averaged quantities typically are one order better than the trivial entry-wise
bounds indicate, \ie 
 $|\frac{1}{N}\sum_i d_i Q_i| \prec \Psi^2= (N\eta)^{-1}$, and indeed this was proven in \cite{BES16}
in the bulk and could be extended to the edge. In fact, 
due to the improvement at the edge, now we expect  a bound of order $\Pi^2\approx\im m/N\eta$, but we cannot obtain this in general.
In this second part of the proof,
 we  therefore  prove a bound  of the form 
 $$
  \Big|\frac{1}{N}\sum_i d_i Q_i\Big| \prec \Pi\Psi\approx\frac{\sqrt{\im m}}{N\eta},
  $$  which is ``half-way" between the standard fluctuation averaging
bound and the  expected optimal bound.
We  compute the high moments of $\frac{1}{N}\sum_i d_i Q_i$ to achieve this bound. Interestingly, 
the apparently leading term in the high moment calculation already gives the optimal bound $\Pi^2$
(first term on the right of \eqref{17071833}),
but a ``cross-term" (when the derivative hits another factor of $\frac{1}{N}\sum_i d_i Q_i$)
is responsible for the weaker $\Pi\Psi$ bound.

Another point to make is that it is not necessary to compute the high moments of another quantity 
for the rough averaged bound,
unlike in  \cite{BES15b,BES16} and in the first part of the current proof, where we always operated with two 
different quantities in parallel. Various error terms 
along the calculation of $\frac{1}{N}\sum_i d_i Q_i$ do contain $T_i$, but 
these terms  can all be estimated using
the entry-wise bound $|T_i|\prec \Psi$ only.  Choosing a special weight sequence $d_i$ we also improve the bound on $\Upsilon$ to
$|\Upsilon| \prec \Pi \Psi$. In particular we could obtain an improved averaged bound on $P_i = Q_i + (G_{ii}+T_i)\Upsilon$ immediately,
and with a little effort on~$K_i$ and~$T_i$ as well, but we do not need them.

Finally, in the third part of the proof (Section~\ref{s.optimal FL})  we obtain the optimal $\Pi^2$  bound for the average of $Q_i$, but
 only for two very specially chosen weights, see \eqref{17021710}--\eqref{17022001}. 
 In fact, only the estimates on the ``cross-term" need to be 
 improved and the weights are carefully chosen to achieve an additional cancellation.
 Nevertheless, linear combinations of $Q_i$'s with  these two special sequences of weights are sufficient to imply 
 an optimal  self-consistent inequality for $\Lambda\deq|\Lambda_A|+|\Lambda_B|$ (see \eqref{17030301}).

 The above three steps are done for a fixed $z$, under an a priori input on the bounds of $G_{ii}$'s and $T_i$'s, (\cf (\ref{17020501}) ). 
 In order to get these inputs uniformly in $\mathcal{D}_\tau(\eta_{\rm m},\eta_\mathrm{M})$, we need to perform a continuity argument 
 in the imaginary part of the spectral parameter $\eta=\im z$. In Section \ref{s.weak law} 
 we will prove in Theorem \ref{thm. weak law at the edge} that
 \begin{equation}\label{weaklawinformal}
   |P_{ii}|+ |K_{ii}|+ \Lambda_{\rm d}^c+\Lambda_{T} \prec \Psi\,, \qquad \mbox{and} \quad 
   \Lambda_{\rm d}+\Lambda_A+ \Lambda_B\prec\frac{1}{(N\eta)^{\frac13}}\,,
 \end{equation}
  uniformly in $\mathcal{D}_\tau(\eta_{\rm m},\eta_\mathrm{M})$. Note that the latter bound is weaker than our final goal
   of order $(N\eta)^{-1}$. Hence we  call the second inequality in \eqref{weaklawinformal} {\it weak local law}, and the final bound \eqref{17072330}, the optimal average law for $G_{ii}$,
  is called {\it strong local law} since it relies on the optimal $(N\eta)^{-1}$ bound in \eqref{weaklawinformal}.
 The reason for this two-level approach, common in
 most local law proofs, is that the  uniformity of the estimates in $z$ is obtained by 
 a continuity  argument that cannot be optimally performed along the high moment estimates behind the 
 fluctuation averaging. 
 In fact, in the bulk regime, Theorem 2.6 of \cite{BES15b},
 (as well as
for  the analogous proofs for Wigner-type  matrices)  fluctuation averaging and high moment estimates were
 not even needed at  this stage; a weak local law was obtained by a straightforward 
averaging of the entrywise law. The edge case is more subtle; $\Lambda$ satisfies a
quadratic inequality (see \eqref{19030801} later) which is linearly unstable. To counter this effect, 
we need stronger bounds on the error term.
 In the entrywise estimate for $G_{ii}$, our error bounds are given in terms of  $\frac{1}{N\eta} \Im (G_{ii}+\mathcal{G}_{ii})$ 
 (\cf (\ref{the tilda guys}), (\ref{17020550}), (\ref{17021202})) and we need to exploit the smallness of  $ \Im (G_{ii}+\mathcal{G}_{ii})$
 via replacing it by its averaged (in $i$) version,
 $\Im m $.  
 Since now our entrywise estimate itself is done via high moment bounds, pulling out the random and $i$-dependent error bound from the expectation and averaging it to get the improved
 bound for $\frac{1}{N}\sum_i d_i Q_i$ is not feasible. 
 Hence, we need to perform a high moment estimate for $\frac{1}{N}\sum_i d_i Q_i$ independently to get the weak law (Lemma~\ref{lem.19030901}).
 The point is that the weaker version of this high moment bound is still compatible with 
 the continuity argument (Section \ref{s.weak law}), 
 leading to the second estimate in \eqref{weaklawinformal}.  On the other hand
  \eqref{weaklawinformal} is enough to guarantee that the input  (\ref{17020501}) holds uniformly in $\mathcal{D}_\tau(\eta_{\rm m},\eta_\mathrm{M})$. 
 With this uniform input, one can show that the discussion in the previous three steps hold also uniformly, leading  eventually to the strong law
 uniformly in $\mathcal{D}_\tau(\eta_{\rm m},\eta_\mathrm{M})$.
 
 We present the three parts explained above (Sections~\ref{s. Entrywise estimate}--\ref{s.optimal FL}) first
 since they represent the essential and strongest ingredients of our proof.
 The proof of the weak law  in Section \ref{s.weak law} relies on similar steps, except that 
 instead of {\it assuming}  the controls on $G_{ii}$
 and $T_i$,  they are {\it enforced} by  inserting  smooth cutoff functions. Thus, along the continuity argument we  can
 bootstrap a single unconditional estimate (on the quantity with cutoffs), a procedure 
 compatible with the high moment method.  The cutoffs involve additional 
 error terms that are still affordable as we are not aiming at the optimal bound for the moment.

 At the end, in Section \ref{s. strong local law}, by inverting the self-consistent inequality  \eqref{17030301}, we conclude that  $\Lambda_\iota\deq \omega_\iota^c -\omega_\iota$, $\iota=A,B$, see (\ref{17072550}) for the definition of $\omega_\iota^c$, are both stochastically dominated by $\Psi^2$. We finally notice~that
 $$ 
    \frac{1}{N}\sum_{i=1}^N d_i \Big( G_{ii} - \frac{1}{a_i - \omega_B^c}\Big)
 $$
 may be expressed as a linear combination of the $Q_i$, see \eqref{17072501}, this quantity is already stochastically bounded
 by  $ \Pi\Psi \le  \Psi^2$ from  the second part of the proof. Since
  replacing $\omega_B^c$ with $\omega_B$ yields an error
 of at most $\Psi^2$, we obtain \eqref{17072330}, the optimal average law for $G_{ii}$.

 The actual proofs are considerably more complicated than this informal summary. On one hand, many error terms need
 to be estimated that have not been mentioned here, in particular we need fluctuation averaging with random 
 weights, a novel complication that has not been considered before.
 On the other hand, in this summary we used the deterministic $\Psi=(N\eta)^{-1/2}$ and $\Pi \approx (\im m/N\eta)^{1/2}$
 as control parameters. In fact, $\Pi$ is  random, see \eqref{17012101},  containing $\im m_H$
 which is $\im m_{A\boxplus B}$ up to a random error that itself depends on $\Lambda$.
   Therefore an additional bootstrap  for a fixed  $z$ is necessary to conclude a deterministic bound on $\Lambda$.

\section{Entry-wise Green function subordination}  \label{s. Entrywise estimate}

 In this section, we prove a subordination property for the Green function entries.  From this section to Appendix \ref{appendix C},  without loss of generality, we assume that 
\begin{align}
\ntr A=\ntr B=0\,.  \label{17072620}
\end{align}
 We define the {\it approximate subordination functions} as
\begin{align}
\omega_A^c(z)\deq z-\frac{\ntr AG(z)}{m_H(z)}\,,\quad \qquad \omega_B^c(z)\deq z-\frac{ \ntr \wt{B}G}{ m_{H}(z)}, \qquad\qquad z\in \mathbb{C}^+. \label{17072550} 
\end{align}
It will be seen that the functions $\omega_A^c$ and $\omega_B^c$ are good approximations of $\omega_A$ and $\omega_B$ defined in (\ref{le prop 1}) with $(\mu_1,\mu_2)=(\mu_A, \mu_B)$.  Switching the roles of $A$ and $B$, and also the roles of $U$ and $U^*$,  we introduce the following analogues of $\wt{B}$, $H$, and $G(z)$, respectively,
\begin{align}\label{the tilda guys}
\wt{A}\deq U^*AU\,,\qquad \qquad \mathcal{H}\deq B+\wt{A}\,,\qquad \qquad \mathcal{G}\equiv \mathcal{G}(z)\deq(\mathcal{H}-z)^{-1}\,.
\end{align}
Observe that, by the cyclicity of the trace,
\begin{align*}
\omega_A^c(z)=z-\frac{\ntr \wt{A}\mathcal{G}(z)}{m_H(z)}\,.
\end{align*}
 From (\ref{17072550}) and the identity $(A+\wt{B}-z)G=I$, it is easy to check that
\begin{align}
\omega_A^c(z)+\omega_B^c(z)-z=-\frac{1}{m_H(z)}\,, \qquad\qquad z\in \mathbb{C}^+\,.  \label{170725130}
\end{align}
Recall the  quantities $S_i$ and $T_i$ defined in (\ref{17072580}). 
We will also need their variants
\begin{align}
\mathring{S}_i\equiv \mathring{S}_i(z)\deq\mathring{\mathbf{h}}_i^* \wt{B}^{\la i\ra} G\mathbf{e}_i=S_i-h_{ii}b_iG_{ii}\,, \qquad \quad \mathring{T}_i\equiv \mathring{T}_i(z)\deq\mathring{\mathbf{h}}_i^* G\mathbf{e}_i=T_i-h_{ii}G_{ii}\,,  \label{17072581}
\end{align}
where the $\chi$ random variable $h_{ii}$ is kicked out.

Further,  we denote (dropping the $z$-dependence from the notation for brevity)
\begin{align}
\Lambda_{{\rm d} i}\deq\Big|G_{ii}-\frac{1}{a_i-\omega_B}\Big|\,,\qquad \qquad \Lambda_{{\rm d}}\deq\max_i \Lambda_{{\rm d} i}\,,\qquad \qquad\Lambda_T\deq\max_i|T_{i}|\,. \label{17072571}
\end{align} 
We also define $\Lambda_{{\rm d} i}^c$ and $\Lambda_{{\rm d}}^c$ analogously by replacing $\omega_B$ by $\omega_B^c$ in the definitions of $\Lambda_{{\rm d} i}$ and $\Lambda_{{\rm d}}$, respectively, ~\eg

\begin{align}
\Lambda_{{\rm d} i}^c\deq\Big|G_{ii}-\frac{1}{a_i-\omega_B^c}\Big|\, , \qquad \Lambda_{{\rm d}}^c\deq \max_i \Lambda_{{\rm d} i}^c.  \label{17080305}
\end{align}

In addition, we  use the notations $\wt{\Lambda}_{{\rm d} i}, \wt{\Lambda}_{{\rm d}}, \wt{\Lambda}_T, \wt{\Lambda}_{{\rm d} i}^c, \wt{\Lambda}_{{\rm d}}^c$ to represent their analogues, obtained by switching the roles of $A$ and $B$, and the roles of $U$ and $U^*$, in the definitions of   $\Lambda_{{\rm d} i}, \Lambda_{{\rm d}}, \Lambda_T, \Lambda_{{\rm d} i}^c, \Lambda_{{\rm d}}^c$,~\eg

\begin{align}\label{tildedef}
\wt{\Lambda}_{{\rm d} i}\deq\Big|\mathcal{G}_{ii}-\frac{1}{b_i-\omega_A}\Big|\,, \qquad \wt{\Lambda}_{{\rm d} i}^c\deq\Big|\mathcal{G}_{ii}-\frac{1}{b_i-\omega_A^c}\Big|\,.
\end{align}

Recall $P_{i}$, $K_i$, and $\Upsilon$ defined in~\eqref{17011301},~\eqref{17072420} and~\eqref{17020511}. Note that all these quantities have analogues with tilde when the roles of $A$ and $B$, and also the roles of $U$ and $U^*$ are switched. 

We further observe the elementary identities 
\begin{align}
\wt{B}G=I-(A-z)G\,, \qquad\qquad G\wt{B}=I-G(A-z)\,. \label{17020508}
\end{align}
Using the first identity in (\ref{17020508}), we can rewrite $\Upsilon$ defined in~\eqref{17020511} as
\begin{align}
\Upsilon=\ntr AG\; \ntr \wt{B}G-\ntr G\;\ntr \wt{B}G A=\frac{1}{N}\sum_{i=1}^N   a_i \Big(G_{ii} \ntr \wt{B}G-(\wt{B}G)_{ii} \ntr G\Big)\,. \label{17011302}
\end{align}

To ease the presentation, we further introduce the control parameter
\begin{align}
\Pi_i\equiv \Pi_i(z)\deq\sqrt{\frac{\Im (G_{ii}(z)+\mathcal{G}_{ii}(z))}{N\eta}}\,,\qquad\qquad i\in\llbracket 1,N\rrbracket \,.\label{17020550}
\end{align}
Note that since $\|H\|< \mathcal{K}$ (\cf (\ref{17072840})), it is easy to see that $\Im G_{ii}(z)\gtrsim \eta$ and $\Im \mathcal{G}_{ii}(z)\gtrsim \eta$ for all $z\in \mathcal{D}_\tau(0,\eta_\mathrm{M})$, by spectral decomposition. This implies 
\begin{align}
\frac{1}{\sqrt{N}}\lesssim \Pi_i(z)\,, \qquad \qquad\forall z\in \mathcal{D}_\tau(0,\eta_\mathrm{M})\,. \label{17020530}
\end{align}

In this section, we derive the following  Green function subordination property. Recall the definitions of $P_i$ and $K_i$ in~\eqref{17011301} and~\eqref{17072420}, as well as the definition of the control parameter $\Psi$ in~\eqref{17012101}.

\begin{pro} \label{pro.17020310} Suppose that the assumptions of Theorem \ref{thm. strong law at the edge} hold.  Fix $z\in \mathcal{D}_\tau(\eta_{\rm m},\eta_\mathrm{M})$. Assume~that 
\begin{align}
\Lambda_{{\rm d}} (z)\prec N^{-\frac{\gamma}{4}}\,, \qquad \wt{\Lambda}_{{\rm d}}(z)\prec N^{-\frac{\gamma}{4}}\,, \qquad \Lambda_{T}(z)\prec 1\,, \qquad \wt{\Lambda}_T(z)\prec 1\,. \label{17020501}
\end{align} 
Then we have, for all $i\in \llbracket 1, N\rrbracket$, that
\begin{align}
| P_i(z)|\prec \Psi(z)\,, \qquad\qquad | K_i(z)|\prec \Psi(z)\,. \label{17020301}
\end{align}
In addition, we also have that
\begin{align}
|\Upsilon(z) |\prec \Psi(z) \label{17020302}
\end{align}
and, for all $i\in \llbracket 1, N\rrbracket$, that
\begin{align}
\Lambda_{{\rm d} i}^c(z)\prec \Psi(z),\qquad\qquad |T_{i}|\prec \Psi(z)\,. \label{17020303}
\end{align}
The same statements hold if we switch the roles of $A$ and $B$, and also the roles of $U$ and $U^*$.
\end{pro}

Before the actual proof of Proposition \ref{pro.17020310}, we establish several bounds that follow from the assumption in (\ref{17020501}). From the definitions in (\ref{17072571}), the assumptions in  (\ref{17020501}), together with (\ref{17020502}), we see~that 
\begin{align}
\max_{i\in \llbracket 1, N\rrbracket }|G_{ii}|\prec 1\,,\qquad \qquad \max_{i\in \llbracket 1, N\rrbracket }|T_{i}|\prec 1\,. \label{17020505}
\end{align}
Analogously, we also have $\max_{i\in \llbracket 1, N\rrbracket }|\mathcal{G}_{ii}|\prec 1$. Hence, under (\ref{17020501}), we see that 
\begin{align*}
\max_{i\in \llbracket 1, N\rrbracket }\Pi_i(z)\prec \Psi(z). 
\end{align*}
Moreover, using the identities in (\ref{17020508}), 
we also get from the first bound in (\ref{17020505}) that
\begin{align}
\max_{i\in \llbracket 1, N\rrbracket} |(XGY)_{ii}|\prec 1, \qquad\qquad X, Y=I \;\text{or}\; \wt{B}. \label{170726100}
\end{align}
In addition, from (\ref{170730100}) we see that
\begin{align}
\frac{1}{N}\sum_{i=1}^N   \frac{1}{a_i-\omega_B(z)}=m_{\mu_A}(\omega_B(z))=m_{\mu_A\boxplus\mu_B} (z). \label{17020507}
\end{align}
Then, the first bound in (\ref{17020501}), together with (\ref{17020507}), (\ref{17020508}), (\ref{17020503}) and  (\ref{17020502}), leads to the following estimates
\begin{align}
\ntr G&= m_{\mu_A\boxplus\mu_B} +O_\prec(N^{-\frac{\gamma}{4}})\,,\nonumber\\ 
\ntr \wt{B} G&= (z-\omega_B) m_{\mu_A\boxplus \mu_B} +O_\prec(N^{-\frac{\gamma}{4}})\,, \nonumber\\
 \ntr \wt{B} G\wt{B}&=(\omega_B-z) \big(1+(\omega_B-z) m_{\mu_A\boxplus\mu_B}\big)+O_\prec(N^{-\frac{\gamma}{4}})\,. \label{17020535}
\end{align}
Furthermore, by (\ref{17020502}), (\ref{17020503}),  and (\ref{17020507}), we see that all the above tracial quantities are $O_\prec(1)$
. This also implies that $|\Upsilon|\prec 1$, (\cf (\ref{17020511})).  Moreover, from  
(\ref{17072550}) and the first two equations in (\ref{17020535}), we can get the following rough estimate under  (\ref{17020501}) and (\ref{17020502}),
\begin{align}
\omega_B^c=\omega_B+O_\prec(N^{-\frac{\gamma}{4}})\,.  \label{170725110}
\end{align}

Further, we make the following convention in the rest of the paper: the notation $O_\prec(\Psi^k)$, for any given integer $k$, represents some generic (possibly) $z$-dependent random variable $X\equiv X(z)$ which satisfies 
\begin{align*}
|X|\prec \Psi^k, \qquad \text{and}\qquad \mathbb{E}|X|^q\prec \Psi^{qk}\,,
\end{align*}
for any given positive integer $q$. The first bound above follows from the original definition of  the notation $O_\prec(\cdot)$ directly. It turns out that it is more convenient to require the second one in our discussions below as well. It will be clear that the second bound always follows from the first one whenever this notation will be used. For more details, we refer to the paragraph above Proposition~6.1 in~\cite{BES16}.  Analogously, for all notation of the form $O_\prec(\Gamma)$ with some deterministic control parameter $\Gamma$, we make the same convention. 

 \begin{proof}[Proof of Proposition \ref{pro.17020310}]
To prove (\ref{17020301}), it suffices to show the high order moment estimates 
\begin{align}
\mathbb{E}\big[ | P_i|^{2p}\big]\prec \Psi^{2p}\,,\qquad \qquad \mathbb{E} \big[ | K_i|^{2p}\big]\prec \Psi^{2p}\,, \label{17072410}
\end{align}
for any fixed $p\in \mathbb{N}$. Let us introduce the notations
\begin{align}
\mathfrak{m}_i^{ (k,l)}\deq P_{i}^k\overline{ P_i^l}\,,\quad \quad \mathfrak{n}_i^{{ (k,l)}}\deq K_i^k\overline{ K_i^l},\qquad \qquad k,l\in \mathbb{N}\,,\quad\quad i\in\llbracket 1,N\rrbracket\,. \label{17072350}
\end{align}

With the definitions in (\ref{17072350}) and the convention made after (\ref{170725110}), we have the following recursive moment estimates. This type of estimates were used first in~\cite{LS16} to derive local laws for sparse Wigner matrices.
\begin{lem}[Recursive moment estimate for $ P_i$ and $ K_i$] \label{lem.17021230}  Suppose the assumptions of Proposition \ref{pro.17020310}. Then, for any fixed integer $p\geq 1$ and any $i\in \llbracket 1, N\rrbracket$, we have
\begin{align}
\mathbb{E}[\mathfrak{m}_i^{(p,p)}]&=\mathbb{E}[O_\prec(\Psi)\mathfrak{m}_i^{(p-1,p)}]+\mathbb{E}[O_\prec(\Psi^2) \mathfrak{m}_i^{(p-2,p)}]+\mathbb{E}[O_\prec(\Psi^2) \mathfrak{m}_i^{(p-1,p-1)}]\,,\label{17021020}\\
 \mathbb{E}[\mathfrak{n}_i^{(p,p)}]&=\mathbb{E}[O_\prec(\Psi)\mathfrak{n}_i^{(p-1,p)}]+\mathbb{E}[O_\prec(\Psi^2) \mathfrak{n}_i^{(p-2,p)}]+\mathbb{E}[O_\prec(\Psi^2) \mathfrak{n}_i^{(p-1,p-1)}]\,, \label{17021021}
\end{align} 
where we made the convention $\mathfrak{m}_i^{(0,0)}=\mathfrak{n}_i^{(0,0)}=1$ and $\mathfrak{m}_i^{(-1,1)}=\mathfrak{n}_i^{(-1,1)}=0$ if $p=1$.
\end{lem}
Although in the statements of Lemma~\ref{lem.17021230}, we use $\Psi$,  in the proof, we actually get better estimates in terms of $\Pi_i^2$ instead of $\Psi^2$ for some error terms. We will keep the stronger form of these estimates since the same errors will appear in the averaged bounds in Section  \ref{s. rough bound} as well.  The average of these errors  is typically smaller than $\Psi^2$.

\begin{proof}[Proof of Lemma \ref{lem.17021230} ] The proof is very similar to that of Lemma 7.3 of \cite{BES16b}, which is presented for the  block additive model in the bulk regime. It suffices to go through the strategy in \cite{BES16b} for our additive model again. The strategy also works well at the regular edge, provided (\ref{17020502}) and (\ref{17020503}) hold. In addition, instead of the control parameter  $\Psi$ used in the proof of Lemma 7.3 of \cite{BES16b}, we aim here at controlling many errors in terms of $\Pi_i$.  This requires a more careful estimate on the error terms. Due to the similarity to the proof of Lemma 7.3 of \cite{BES16b}, we only sketch the proof of  Lemma \ref{lem.17021230}  in the sequel. 

For each $i\in \llbracket 1, N\rrbracket$, we write
\begin{align}
\mathbb{E}[\mathfrak{m}_i^{(p,p)}] =\mathbb{E}[ P_i\mathfrak{m}_i^{(p-1,p)}]&=\mathbb{E}[(\wt{B}G)_{ii} \ntr G \mathfrak{m}_i^{(p-1,p)}]+\mathbb{E}\big[\big(-G_{ii}\ntr \wt{B}G+(G_{ii}+T_{i})\Upsilon\big)\mathfrak{m}_i^{(p-1,p)}\big]\,,\label{17020532}
\end{align}
respectively,
\begin{align}
\mathbb{E}[\mathfrak{n}_i^{(p,p)}] =\mathbb{E}[ K_i \mathfrak{n}_i^{(p-1,p)}]=\mathbb{E}[T_i \mathfrak{n}_i^{(p-1,p)}]+\mathbb{E}\big[\big( (b_iT_i+(\wt{B}G)_{ii})\ntr G- (G_{ii}+T_i) \ntr \wt{B}G\big)\mathfrak{n}_i^{(p-1,p)}\big]\,. \label{17020533}
\end{align}
Using the fact $\mathbf{e}_i^*R_i=-\mathbf{h}_i^*$ (\cf (\ref{17072573})), we can write
\begin{align}
(\wt{B}G)_{ii}&=\mathbf{e}_i^* R_i\wt{B}^{\la i\ra} R_i G\mathbf{e}_i=-\mathbf{h}_i^* \wt{B}^{\la i\ra} R_i G\mathbf{e}_i=-\mathbf{h}_i^* \wt{B}^{\la i\ra}G\mathbf{e}_i+\ell_i^2 \mathbf{h}_i^* \wt{B}^{\la i\ra}(\mathbf{e}_i+\mathbf{h}_i)(\mathbf{e}_i+\mathbf{h}_i)^*G\mathbf{e}_i\nonumber\\
&=-S_i+\ell_i^2(b_i h_{ii}+\mathbf{h}_i^* \wt{B}^{\la i\ra}\mathbf{h}_i) (G_{ii}+T_i)= -\mathring{S}_{i}+\varepsilon_{i1}\,, \label{17020531}
\end{align} 
where $S_i$ and $\mathring{S}_i$ are defined in (\ref{17072580}) and (\ref{17072581}), respectively, $\ell_i$ is defined in~\eqref{17072590} and
\begin{align}
\varepsilon_{i1}\deq\big((\ell_i^2-1)b_i h_{ii}+\ell_i^2 \mathbf{h}_i^* \wt{B}^{\la i\ra} \mathbf{h}_i\big) G_{ii}+\ell_i^2 \big(b_i h_{ii}+\mathbf{h}_i^* \wt{B}^{\la i\ra}\mathbf{h}_i\big) T_i\,. \label{17071805}
\end{align}
With the aid of Lemma \ref{lem.091720}, it is elementary to check 
\begin{align}
|h_{ii}|\prec \frac{1}{\sqrt{N}}\,,\qquad\quad |\ell_i^2-1|\prec \frac{1}{\sqrt{N}}\,,\quad \qquad |\mathbf{h}_i^* \wt{B}^{\la i\ra}\mathbf{h}_i|\prec \frac{1}{\sqrt{N}}\,,  \label{17021540}
\end{align}
where in the last inequality we also used the fact that $\ntr \wt{B}^{\la i\ra}=\ntr B=0$, under the convention (\ref{17072620}). 
Applying the bounds in (\ref{17020505}) and (\ref{17021540}),
 it is easy to see that 
\begin{align}
|\varepsilon_{i1}|\prec \frac{1}{\sqrt{N}}\,. \label{17020534}
\end{align}

Substituting (\ref{17020531}) and (\ref{17020534}) into the first term on the right hand side of (\ref{17020532}), we have
\begin{align}
\mathbb{E}[(\wt{B}G)_{ii} \ntr G \mathfrak{m}_i^{(p-1,p)}]=-\mathbb{E}[\mathring{S}_i \ntr G \mathfrak{m}_i^{(p-1,p)}]+\mathbb{E}[O_\prec(N^{-\frac12})\mathfrak{m}_i^{(p-1,p)}]\,, \label{17020536}
\end{align}
where for the second term on the right hand side above we also used $\ntr G=O_\prec(1)$; \cf (\ref{17020535}). We recall the definition  of $\mathring{S}_i$ from (\ref{17072581}) and rewrite 
\begin{align*}
\mathring{S}_i=\sum_{k}^{(i)} \bar{g}_{ik} \frac{1}{\|\mathbf{g}_i\| }\mathbf{e}_k^*\wt{B}^{\la i\ra}G\mathbf{e}_i.
\end{align*}
Hereafter, we use the notation $\sum_{k}^{(i)}$ to represent the sum over $k\in \llbracket 1, N\rrbracket\setminus\{i\}$.
Thus, the first term on the right of (\ref{17020536}) is of the form $\mathbb{E}[\sum_{k}^{(i)} \bar{g}_{ik} \langle \cdots \rangle]$, where $\langle \cdots\rangle$ can be regarded as a function of the $\bar{g}_{ik}$'s and the $g_{ik}$'s.  Recall the following integration by parts formula for complex centered Gaussian variables,
\begin{align}
\int_{\mathbb{C}} \bar{g} f(g,\bar{g}) \e{-\frac{|g|^2}{\sigma^2}} {\rm d}^2 g=\sigma^2 \int_{\mathbb{C}} \partial_g f(g,\bar{g}) \e{-\frac{|g|^2}{\sigma^2}}{\rm d}^2 g \,, \label{17021237}
\end{align}
for any differentiable function $f: \mathbb{C}^2\to \mathbb{C}$. Applying (\ref{17021237}) to the first term on the right of (\ref{17020536}), we~get
\begin{align}
&\mathbb{E}[\mathring{S}_i \ntr G \mathfrak{m}_i^{(p-1,p)}]=\frac{1}{N} \sum_k^{(i)} \mathbb{E}\Big[ \frac{1}{\|\mathbf{g}_i\| }\frac{\partial (\mathbf{e}_k^*\wt{B}^{\la i\ra} G\mathbf{e}_i)}{\partial g_{ik}}\ntr G \mathfrak{m}_i^{(p-1,p)}\Big]\nonumber\\
&\qquad+ \frac{1}{N} \sum_k^{(i)} \mathbb{E}\Big[ \frac{\partial \|\mathbf{g}_i\| ^{-1}}{\partial g_{ik}} \mathbf{e}_k^*\wt{B}^{\la i\ra} G\mathbf{e}_i\ntr G \mathfrak{m}_i^{(p-1,p)}\Big]+\frac{1}{N} \sum_k^{(i)} \mathbb{E}\Big[ \frac{\mathbf{e}_k^*\wt{B}^{\la i\ra} G\mathbf{e}_i}{\|\mathbf{g}_i\| }  \frac{\partial \ntr G}{\partial g_{ik}}\mathfrak{m}_i^{(p-1,p)} \Big]\nonumber\\
&\qquad+ \frac{p-1}{N} \sum_k^{(i)} \mathbb{E}\Big[ \frac{\mathbf{e}_k^*\wt{B}^{\la i\ra} G\mathbf{e}_i}{\|\mathbf{g}_i\| }  \ntr G \frac{\partial  P_i}{\partial g_{ik}}\mathfrak{m}_i^{(p-2,p)}\Big]+ \frac{p}{N} \sum_k^{(i)} \mathbb{E} \Big[ \frac{\mathbf{e}_k^*\wt{B}^{\la i\ra} G\mathbf{e}_i}{\|\mathbf{g}_i\| }  \ntr G \frac{\partial \overline{ P_i}}{\partial g_{ik}}\mathfrak{m}_i^{(p-1,p-1)}\Big]. \label{17020540}
\end{align} 
Analogously, by $T_i=\mathring{T}_i+h_{ii}G_{ii}$, (\ref{17072581}),  the first bound in (\ref{17020505}), the first bound in (\ref{17021540}), and also (\ref{17020530}), we can write the first term on the right hand side of (\ref{17020533}) as
\begin{align}
\mathbb{E}[T_i \mathfrak{n}_i^{(p-1,p)}]=\mathbb{E}[\mathring{T}_i \mathfrak{n}_i^{(p-1,p)}]+\mathbb{E}[O_\prec(N^{-\frac12}) \mathfrak{n}_i^{(p-1,p)}]\,. \label{17071802}
\end{align}
Similarly to (\ref{17020540}), applying the integration by parts formula, we obtain 
\begin{align}
&\mathbb{E}[\mathring{T}_i  \mathfrak{n}_i^{(p-1,p)}]=\frac{1}{N} \sum_k^{(i)} \mathbb{E}\Big[ \frac{1}{\|\mathbf{g}_i\| }\frac{\partial (\mathbf{e}_k^*G\mathbf{e}_i)}{\partial g_{ik}} \mathfrak{n}_i^{(p-1,p)}\Big]+ \frac{1}{N} \sum_k^{(i)} \mathbb{E}\Big[ \frac{\partial \|\mathbf{g}_i\| ^{-1}}{\partial g_{ik}} \mathbf{e}_k^* G\mathbf{e}_i \mathfrak{n}_i^{(p-1,p)}\Big]\nonumber\\
&\;\;+ \frac{p-1}{N} \sum_k^{(i)} \mathbb{E}\Big[ \frac{\mathbf{e}_k^* G\mathbf{e}_i}{\|\mathbf{g}_i\| }   \frac{\partial  K_i}{\partial g_{ik}}\mathfrak{n}_i^{(p-2,p)}\Big]+ \frac{p}{N} \sum_k^{(i)} \mathbb{E} \Big[ \frac{ \mathbf{e}_k^* G\mathbf{e}_i}{\|\mathbf{g}_i\| }  \frac{\partial \overline{ K_i}}{\partial g_{ik}}\mathfrak{n}_i^{(p-1,p-1)}\Big]\,. \label{17021011111}
\end{align}

First, we consider the first term on the right side of (\ref{17020540}).  Recall $\ell_i$ from (\ref{17072590}). For brevity, we set
\begin{align}
c_i\deq\frac{\ell_i^2}{\|\mathbf{g}_i\| }. \label{170725102}
\end{align}
 It is elementary to derive  that 
\begin{align}
&\frac{\partial G}{\partial g_{ik}}= c_i\big(G\mathbf{e}_k (\mathbf{e}_i+\mathbf{h}_i)^* \wt{B}^{\la i\ra} R_i G+GR_i\wt{B}^{\la i\ra}\mathbf{e}_k (\mathbf{e}_i+\mathbf{h}_i)^*G\big)+\Delta_G(i,k)\,. \label{17071801}
\end{align}
Here $\Delta_G(i,k)$ is a small remainder, defined as 
\begin{align}
\Delta_G(i,k)\deq-G\Delta_R(i,k)\wt{B}^{\la i\ra} R_i G-GR_i\wt{B}^{\la i\ra}\Delta_R(i,k)G, \label{17022801}
\end{align}
where
\begin{align}
\Delta_R(i,k)\deq\frac{\ell_i^2}{2\|\mathbf{g}_i\| ^2} \bar{g}_{ik}\big(\mathbf{e}_i\mathbf{h}_i^*+\mathbf{h}_i\mathbf{e}_i^*+2\mathbf{h}_i\mathbf{h}_i^*\big)-\frac{\ell_i^4}{2\|\mathbf{g}_i\| ^3} g_{ii}\bar{g}_{ik}\big(\mathbf{e}_i+\mathbf{h}_i\big)\big(\mathbf{e}_i+\mathbf{h}_i\big)^*\,.  \label{170729100}
\end{align}
The $\Delta_G(i,k)$'s are irrelevant error terms. We handle quantities with $\Delta_G(i,k)$ separately in Appendix~\ref{appendix B}.

Analogously to~(7.55) of \cite{BES16b}, using (\ref{17071801}), we can get 
\begin{multline}
\frac{1}{N} \sum_k^{(i)} \frac{\partial (\mathbf{e}_k^* \wt{B}^{\la i\ra}G\mathbf{e}_i)}{\partial g_{ik}}=-c_i \frac{1}{N} \sum_k^{(i)}  \mathbf{e}_k^* \wt{B}^{(i)} G\mathbf{e}_k (b_i T_i +(\wt{B}G)_{ii})\\
 +c_i\frac{1}{N} \sum_k^{(i)} \mathbf{e}_k^* \wt{B}^{\la i\ra} G R_i \wt{B}^{\la i\ra} \mathbf{e}_k (G_{ii}+T_i)+\frac{1}{N} \sum_k^{(i)} \mathbf{e}_k^*\wt{B}^{\la i\ra} \Delta_G(i,k) \mathbf{e}_i\,. \label{17020551}
\end{multline}
Note that $T_i$ naturally appears in the first term of  (\ref{17020540}) after integrating by parts the $\mathring{S}_i$ term. This explains why we need to study the high moments of $K_i$ to get another equation. 
Now, we claim that 
\begin{align}
\frac{1}{N} \sum_k^{(i)}  \mathbf{e}_k^* \wt{B}^{(i)} G\mathbf{e}_k=\ntr \wt{B}G+O_\prec(\Pi_i^2)\,,\qquad \frac{1}{N} \sum_k^{(i)} \mathbf{e}_k^* \wt{B}^{\la i\ra} G R_i \wt{B}^{\la i\ra} \mathbf{e}_k=\ntr \wt{B}G\wt{B}+O_\prec(\Pi_i^2)\,, \label{17072401}
\end{align}
with $\Pi_i$ given in~\eqref{17020550}. We state the proof for the first estimate in (\ref{17072401}).  Note that
\begin{align}
\frac{1}{N} \sum_k^{(i)}  \mathbf{e}_k^* \wt{B}^{(i)} G\mathbf{e}_k= \ntr \wt{B}^{\la i\ra} G-\frac{1}{N} (\wt{B}^{\la i\ra} G)_{ii}=\ntr \wt{B}^{\la i\ra} G+O_\prec(\frac{1}{N})\,, \label{17021510}
\end{align}
where the last step follows from the identity $(\wt{B}^{\la i\ra} G)_{ii}=b_i G_{ii}$ and (\ref{17020505}).  Then, using that $\wt{B}^{\la i\ra}=R_i\wt{B} R_i$ and $R_i=I-\mathbf{r}_i\mathbf{r}_i^*$ (\cf (\ref{17072593})), we see that
\begin{align*}
\ntr \wt{B}G-\ntr \wt{B}^{\la i\ra} G= \ntr \wt{B} G- \ntr R_i\wt{B} R_iG=\frac{1}{N} \mathbf{r}_i^* \wt{B}G\mathbf{r}_i+\frac{1}{N}\mathbf{r}_i^*G\wt{B} \mathbf{r}_i-\frac{1}{N}\mathbf{r}_i^*\wt{B}\mathbf{r}_i \mathbf{r}_i^*G\mathbf{r}_i\,.
\end{align*}
Using (\ref{170725100}), $\ell_i=1+O_\prec(\frac{1}{\sqrt{N}})$ and $\|\mathbf{r}_i^* \wt{B}\|\lesssim 1$, we get by Cauchy-Schwarz that 
\begin{align*}
\big|\mathbf{r}_i^* \wt{B}G\mathbf{r}_i\big|\lesssim\Big( \|G\mathbf{e}_i\|^2+\|G\mathbf{h}_i\|^2\Big)^{\frac{1}{2}}=\Big(\frac{\Im (G_{ii}+ \mathbf{h}_i^* G\mathbf{h}_i)}{\eta}\Big)^{\frac{1}{2}}= \Big(\frac{\Im (G_{ii}+\mathcal{G}_{ii})}{\eta}\Big)^{\frac{1}{2}},
\end{align*}
with $\mathcal{G}$ given in~\eqref{the tilda guys}, where in the last step we used
\begin{align}
\mathbf{h}_i^* G\mathbf{h}_i=\mathbf{u}_i^* G\mathbf{u}_i=\mathbf{e}_i^*U^*GU\mathbf{e}_i=\mathcal{G}_{ii} \label{170726140}
\end{align}
and the identities  $|G|^2=\frac{1}{\eta} \Im G$ and $ |\mathcal{G}|^2=\frac{1}{\eta} \Im \mathcal{G}$. 
Similarly, we have
\begin{align*}
\big|\mathbf{r}_i^*G\wt{B} \mathbf{r}_i\big|\lesssim \Big(\frac{\Im (G_{ii}+\mathcal{G}_{ii})}{\eta}\Big)^{\frac{1}{2}},\qquad  \big|\mathbf{r}_i^*G \mathbf{r}_i\big|\lesssim \Big(\frac{\Im (G_{ii}+\mathcal{G}_{ii})}{\eta}\Big)^{\frac{1}{2}}.
\end{align*}
Hence,  we have
\begin{align}
\big|\ntr \wt{B}G-\ntr \wt{B}^{\la i\ra} G\big|\lesssim  \frac{1}{N} \Big(\frac{\Im (G_{ii}+\mathcal{G}_{ii})}{\eta}\Big)^{\frac{1}{2}}\lesssim \frac{\Im (G_{ii}+\mathcal{G}_{ii})}{N \eta}=O_\prec(\Pi_i^2)\,, \label{17021511}
\end{align}
where in the second step, we used the fact $\Im G_{ii}, \Im \mathcal{G}_{ii}\gtrsim \eta$. Combining  (\ref{17021510}) with (\ref{17021511}) we obtain the first estimate of (\ref{17072401}).  The second estimate in (\ref{17072401}) is proved in the same way.  

Hence, using (\ref{17072401}) and the first estimate in (\ref{17030105}), we obtain from (\ref{17020551}) that
\begin{align}
\frac{1}{N} \sum_k^{(i)} \frac{\partial (\mathbf{e}_k^* \wt{B}^{\la i\ra}G\mathbf{e}_i)}{\partial g_{ik}}=-c_i \ntr \wt{B}G \big(b_i T_i+(\wt{B}G)_{ii}\big)+c_i \ntr \wt{B} G\wt{B} \big( G_{ii}+T_i\big)+O_\prec({\Pi}_i^2)\,. \label{17020801}
\end{align}
Analogously, we can show that
\begin{align}
\frac{1}{N}\sum_k^{(i)} \frac{\partial (\mathbf{e}_k^* G\mathbf{e}_i)}{\partial g_{ik}}=-c_i\ntr G\big( b_i T_i+(\wt{B}G)_{ii}\big)+c_i\ntr \wt{B}G \big( G_{ii}+T_i\big)+O_\prec({\Pi}_i^2)\,. \label{17020802}
\end{align}
Using (\ref{17020533}), (\ref{17071802}), (\ref{17021011111}) and (\ref{17020802}) and the estimate $\frac{c_i}{\|\mathbf{g}_i\|}=1+O_\prec(\frac{1}{\sqrt{N}})$, we obtain 
\begin{multline}
\mathbb{E}[ \mathfrak{n}_i^{(p,p)}]=\mathbb{E}\Big[ O_\prec(\Psi)\mathfrak{n}_i^{(p-1,p)}\Big]+ \frac{1}{N} \sum_k^{(i)} \mathbb{E}\Big[ \frac{\partial \|\mathbf{g}_i\| ^{-1}}{\partial g_{ik}} \mathbf{e}_k^* G\mathbf{e}_i \mathfrak{n}_i^{(p-1,p)}\Big]\\+ \frac{p-1}{N} \sum_k^{(i)} \mathbb{E}\Big[ \frac{\mathbf{e}_k^* G\mathbf{e}_i}{\|\mathbf{g}_i\| }   \frac{\partial  K_i}{\partial g_{ik}}\mathfrak{n}_i^{(p-2,p)}\Big]+ \frac{p}{N} \sum_k^{(i)} \mathbb{E} \Big[ \frac{\mathbf{e}_k^* G\mathbf{e}_i}{\|\mathbf{g}_i\| }   \frac{\partial \overline{ K_i}}{\partial g_{ik}}\mathfrak{n}_i^{(p-1,p-1)}\Big]\,. \label{17021011}
\end{multline}

Then, combining (\ref{17020801}) with (\ref{17020802}), we  obtain
\begin{multline}
\frac{1}{N}\sum_k^{(i)} \frac{\partial (\mathbf{e}_k^* \wt{B}^{\la i\ra}G\mathbf{e}_i)}{\partial g_{ik}} \ntr G =-c_i(G_{ii}+T_i) \big(\ntr \wt{B}G-\Upsilon\big)+\frac{1}{N}\sum_k^{(i)} \frac{\partial (\mathbf{e}_k^* G\mathbf{e}_i)}{\partial g_{ik}} \ntr \wt{B}G+O_\prec({\Pi}_i^2)\\=-c_i(G_{ii}+T_i) \big(\ntr \wt{B}G-\Upsilon\big)+\mathring{T}_i \ntr \wt{B}G+\Big(\frac{1}{N}\sum_k^{(i)} \frac{\partial (\mathbf{e}_k^* G\mathbf{e}_i)}{\partial g_{ik}}-\mathring{T}_i \Big)\ntr \wt{B}G+O_\prec({\Pi}_i^2)\,. \label{17020812}
\end{multline}
Recall the definition of $c_i$ from (\ref{170725102}). It is elementary to check that
\begin{align}
c_i=\|\mathbf{g}_i\| -h_{ii}-\big( \|\mathbf{g}_i\| ^2-1\big)+O_\prec(\frac{1}{N})\,. \label{17020811}
\end{align}
Plugging (\ref{17020811}) into (\ref{17020812}) and also using the second equation in (\ref{17072581}), we can write
\begin{multline}
\frac{1}{N}\sum_k^{(i)} \frac{\partial (\mathbf{e}_k^* \wt{B}^{\la i\ra}G\mathbf{e}_i)}{\partial g_{ik}} \ntr G= -\|\mathbf{g}_i\| \big( G_{ii}\ntr \wt{B}G-(G_{ii}+T_i)\Upsilon\big)\\
+\Big(\frac{1}{N}\sum_k^{(i)} \frac{\partial (\mathbf{e}_k^* G\mathbf{e}_i)}{\partial g_{ik}}-\|\mathbf{g}_i\| \mathring{T}_i \Big)\ntr \wt{B}G+\varepsilon_{i2}+O_\prec({\Pi}_i^2), \label{17021008}
\end{multline}
where $\varepsilon_{i2}$ collects irrelevant terms
\begin{align}
\varepsilon_{i2}\deq &\big(\|\mathbf{g}_i\| -c_i\big) \big(G_{ii}\ntr \wt{B}G-(G_{ii}+T_i)\Upsilon\big)+\big(\|\mathbf{g}_i\| \mathring{T}_i-c_i T_i\big)\ntr \wt{B}G\nonumber\\
=& \big( \|\mathbf{g}_i\| ^2-1\big)G_{ii}\ntr \wt{B}G- \big(h_{ii}+\big( \|\mathbf{g}_i\| ^2-1\big)\big) (G_{ii}+T_i)\Upsilon\nonumber\\
&\qquad+\big(h_{ii}+\big( \|\mathbf{g}_i\| ^2-1\big)\big)T_i \ntr \wt{B}G+O_\prec\big(\frac{1}{N}\big)\,. \label{17021004}
\end{align}
From the estimates $|h_{ii}|\prec\frac{1}{\sqrt{N}}$, $\|\mathbf{g}_i\| =1+O_\prec(\frac{1}{\sqrt{N}})$,  (\ref{17020505}) and the observation that the tracial quantities are $O_\prec (1)$, we~see that
\begin{align}
\varepsilon_{i2}=O_\prec\big(\frac{1}{\sqrt{N}}\big)\,. \label{170729110}
\end{align}

Combining (\ref{17020532}), (\ref{17020531}), (\ref{17020540}) and (\ref{17021008}), we have 
\begin{align}
&\mathbb{E}[\mathfrak{m}_i^{(p,p)}] =-\mathbb{E}[(\mathring{S}_{i}+\varepsilon_{i1}) \ntr G \mathfrak{m}_i^{(p-1,p)}]+\mathbb{E}\big[\big(-G_{ii}\ntr \wt{B}G+(G_{ii}+T_{i})\Upsilon\big)\mathfrak{m}_i^{(p-1,p)}\big]\nonumber\\
&= \mathbb{E}\Big[ \Big(\mathring{T}_i-\frac{1}{\|\mathbf{g}_i\| }\frac{1}{N} \sum_k^{(i)}\frac{\partial (\mathbf{e}_k^*G\mathbf{e}_i)}{\partial g_{ik}}\Big)\ntr \wt{B}G \mathfrak{m}_i^{(p-1,p)}\Big]-\frac{1}{N} \sum_k^{(i)} \mathbb{E}\Big[ \frac{\partial \|\mathbf{g}_i\| ^{-1}}{\partial g_{ik}} \mathbf{e}_k^*\wt{B}^{\la i\ra} G\mathbf{e}_i\ntr G \mathfrak{m}_i^{(p-1,p)}\Big]\nonumber\\
& -\frac{1}{N} \sum_k^{(i)} \mathbb{E}\Big[ \frac{\mathbf{e}_k^*\wt{B}^{\la i\ra} G\mathbf{e}_i}{\|\mathbf{g}_i\| }  \frac{\partial \ntr G}{\partial g_{ik}}\mathfrak{m}_i^{(p-1,p)} \Big]- \frac{p-1}{N} \sum_k^{(i)} \mathbb{E}\Big[ \frac{\mathbf{e}_k^*\wt{B}^{\la i\ra} G\mathbf{e}_i}{\|\mathbf{g}_i\| }  \ntr G \frac{\partial  P_i}{\partial g_{ik}}\mathfrak{m}_i^{(p-2,p)}\Big]\nonumber\\
&-\frac{p}{N} \sum_k^{(i)} \mathbb{E} \Big[ \frac{\mathbf{e}_k^*\wt{B}^{\la i\ra} G\mathbf{e}_i}{\|\mathbf{g}_i\| }  \ntr G \frac{\partial \overline{ P_i}}{\partial g_{ik}}\mathfrak{m}_i^{(p-1,p-1)}\Big]+\mathbb{E}\Big[\Big(\varepsilon_{i1}\ntr G-\frac{1}{\|\mathbf{g}_i\| } \varepsilon_{i2}+O_\prec(\Pi_i^2) \Big) \mathfrak{m}_i^{(p-1,p)}\Big]. \label{17021023}
\end{align}
For the first term on the right of (\ref{17021023}), analogously to (\ref{17021011111}), applying (\ref{17021237}) to the $\mathring{T}_i$-term, we get
\begin{align}
 &\mathbb{E}\Big[\Big(\mathring{T}_i- \frac{1}{\|\mathbf{g}_i\| }\frac{1}{N} \sum_k^{(i)}\frac{\partial (\mathbf{e}_k^*G\mathbf{e}_i)}{\partial g_{ik}}\Big)\ntr \wt{B}G \mathfrak{m}_i^{(p-1,p)}\Big]\nonumber\\
 &=\frac{1}{N} \sum_k^{(i)} \mathbb{E}\Big[ \frac{1}{\|\mathbf{g}_i\| }\frac{\partial \ntr \wt{B}G}{\partial g_{ik}} \mathbf{e}_k^* G\mathbf{e}_i \ntr \wt{B}G \mathfrak{m}_i^{(p-1,p)}\Big] +\frac{1}{N} \sum_k^{(i)} \mathbb{E}\Big[ \frac{\partial \|\mathbf{g}_i\| ^{-1}}{\partial g_{ik}} \mathbf{e}_k^* G\mathbf{e}_i \ntr \wt{B}G \mathfrak{m}_i^{(p-1,p)}\Big]\nonumber\\
&\qquad+ \frac{p-1}{N} \sum_k^{(i)} \mathbb{E}\Big[ \frac{ \mathbf{e}_k^* G\mathbf{e}_i}{\|\mathbf{g}_i\| }  \frac{\partial  P_i}{\partial g_{ik}}\ntr \wt{B}G \mathfrak{m}_i^{(p-2,p)}\Big]+ \frac{p}{N} \sum_k^{(i)} \mathbb{E} \Big[ \frac{\mathbf{e}_k^* G\mathbf{e}_i}{\|\mathbf{g}_i\| }   \frac{\partial \overline{ P_i}}{\partial g_{ik}}\ntr \wt{B}G \mathfrak{m}_i^{(p-1,p-1)}\Big]. \label{17021024}
\end{align}

Recall the estimates of $\varepsilon_{i1}$ and $\varepsilon_{i2}$ in (\ref{17020534}) and (\ref{170729110}), respectively, which implies that $|\varepsilon_{i1}|\prec\Psi$ and $|\varepsilon_{i2}|\prec \Psi$.  Therefore, to show (\ref{17021020}), it suffices to estimate the four last terms in the right side of (\ref{17021023}), and all the terms on the right side of (\ref{17021024}). Then, in order  to show (\ref{17021021}), it suffices to estimate the last three  terms on the right side of  (\ref{17021011}).  All these terms can be estimated based on the following lemma.

\begin{lem} \label{lem.17021201} Suppose the assumptions in Proposition \ref{pro.17020310} hold. Set $X_i=I$ or $\wt{B}^{\la i\ra}$. Let $Q$ be any (possibly random) diagonal matrix satisfying $\|Q\|\prec 1$ and $X=I$ or $A$. We have the following estimates
\begin{align}
&\frac{1}{N} \sum_k^{(i)}  \frac{\partial \|\mathbf{g}_i\| ^{-1}}{\partial g_{ik}} \mathbf{e}_k^* X_i G\mathbf{e}_i=O_\prec(\frac{1}{N}), \qquad& &\frac{1}{N} \sum_k^{(i)} \mathbf{e}_i^*X \frac{\partial G}{\partial g_{ik}} \mathbf{e}_i\mathbf{e}_k^* X_i G\mathbf{e}_i=O_\prec({\Pi}_i^2), \nonumber\\
& \frac{1}{N}  \sum_k^{(i)} \frac{\partial T_{i}}{ \partial g_{ik}} \mathbf{e}_k^* X_i G\mathbf{e}_i= O_\prec({\Pi}_i^2),\qquad& &\frac{1}{N}  \sum_k^{(i)} \ntr \Big(Q X\frac{\partial G}{\partial g_{ik}}\Big) \mathbf{e}_k^* X_i G\mathbf{e}_i=O_\prec\big(\Psi^2\Pi_i^2\big),\nonumber\\
&\frac{1}{N}  \sum_k^{(i)} \ntr \Big(Q X\frac{\partial G}{\partial g_{ik}}\Big) \mathbf{e}_k^* X_i \mathring{\mathbf{g}}_i=O_\prec\big(\Psi^2\Pi_i^2\big).  \label{17021202}
\end{align}
In addition, the same estimates hold if we replace $\frac{\partial G}{\partial g_{ik}}$ and $\frac{\partial T_i}{\partial g_{ik}}$ by their complex conjugates $\frac{\partial \overline{G}}{\partial g_{ik}}$ and $\frac{\partial \overline{T}_i}{\partial g_{ik}}$ in the last four equations above.
\end{lem}  
The proof of Lemma \ref{lem.17021201} is postponed to Appendix \ref{appendix B}.  We also remark here that last equation in (\ref{17021202}) will not be used in the remaining proof of Lemma \ref{lem.17021230}, but it will be used in Section \ref{s. rough bound}.
With the aid of Lemma \ref{lem.17021201}, the remaining proof of Lemma \ref{lem.17021230}  is the same as the counterpart to the proof of Lemma 7.3 in \cite{BES16b}. The only difference is that we use the improved bounds in  Lemma \ref{lem.17021201} instead of those in Lemma 7.4 in \cite{BES16b}. Specifically, the estimates for the second term of (\ref{17021011}), the second term of (\ref{17021023}), and the second term of (\ref{17021024}) follow from the first equation in (\ref{17021202}). 
The third term of (\ref{17021023}) and the first term of (\ref{17021024}) can be estimated by the fourth  equation in (\ref{17021202}), after writing $\ntr \wt{B}G= 1-\ntr (A-z)G$.  All the other terms have $\frac{\partial K_i}{\partial g_{ik}}$ and $\frac{\partial P_i}{\partial g_{ik}}$ or their complex conjugate  involved. Recall the definitions in (\ref{17011301}) and (\ref{17072420}), and also the first equation in (\ref{17020508}). Then, by the chain rule, we see that all terms in (\ref{17021011}), (\ref{17021023}) and (\ref{17021024}),  with $\frac{\partial K_i}{\partial g_{ik}}$ and $\frac{\partial P_i}{\partial g_{ik}}$ or their complex conjugate counterparts  involved,  can be estimated by combining the second to the fourth equations in (\ref{17021202}). 
This completes the proof of Lemma \ref{lem.17021230}.
 \end{proof}
 With Lemma \ref{lem.17021230} , we can complete the proof of Proposition \ref{pro.17020310}.
 The proof is nearly the same as that for Theorem 7.2 in \cite{BES16b}. For the convenience of the reader, we sketch it below.  
 
 Fix $z\in \mathcal{D}_\tau(\eta_{\rm m},\eta_\mathrm{M})$. Using Young's inequality, we obtain from (\ref{17021020}) that for any given (small) $\varepsilon>0$,
 \begin{align*}
\mathbb{E}\big[\mathfrak{m}_i^{(p,p)}(z) \big]\leq \frac{1}{3} \frac{1}{2p} N^{2p\varepsilon}\Psi^{2p}+3\frac{2p-1}{2p} N^{-\frac{2p\varepsilon}{2p-1}} \mathbb{E} \big[\mathfrak{m}_i^{(p,p)}(z)\big]\,.  
 \end{align*}
 Since $\varepsilon>0$ was arbitrary, this implies the first bound in (\ref{17072410}). The second one then follows from (\ref{17021021}) in the same manner.  By Markov's inequality, we get  (\ref{17020301}). 
 
 Next, we show how (\ref{17020302}) and (\ref{17020303}) follow from (\ref{17020301}) and the assumption (\ref{17020501}). To this end, we first prove the following crude bound 
 \begin{align}
 \Lambda_T(z)\prec N^{-\frac{\gamma}{4}}\,.  \label{17072441}
 \end{align}
 From the definition in (\ref{17072420}), we can rewrite the second estimate in (\ref{17020301}) as
 \begin{align}
(1+b_i\ntr G-\ntr (\wt{B}G))T_{i}= G_{ii}\ntr (\wt{B}G)-(\wt{B}G)_{ii} \ntr G+O_\prec(\Psi)\,.  \label{17072432}
 \end{align}
  Using the identity
  \begin{align}
  (\wt{B}G)_{ii}=1-(a_i-z)G_{ii}(z)\,, \label{170725131}
  \end{align}
  and approximate $G_{ii}$ by $(a_i-\omega_B)^{-1}$,  we get from (\ref{17020501}) and (\ref{17020502}) that
 \begin{align}
 (\wt{B}G)_{ii}=\frac{z-\omega_B}{a_i-\omega_B}+O_\prec(N^{-\frac{\gamma}{4}})\,. \label{17072431}
 \end{align}
 We also recall the estimates of the tracial quantities in (\ref{17020535}) under the assumption (\ref{17020501}). Plugging (\ref{17072431}), (\ref{17020535}) and the first bound in the assumption (\ref{17020501}) into (\ref{17072432}), we~get
 \begin{align}
 \big(1+(b_i-z+\omega_B) m_{\mu_A\boxplus \mu_B}+O_\prec(N^{-\frac{\gamma}{4}})\big)T_{i}= O_\prec(N^{-\frac{\gamma}{4}})+O_\prec(\Psi)=O_\prec(N^{-\frac{\gamma}{4}})\,, \label{17072440}
 \end{align}
 where in the last step we used that  $\Psi\leq N^{-\frac{\gamma}{2}}$ for all $\eta\geq \eta_{\rm m}$. From the second line in (\ref{170730100}), we note~that 
 \begin{align*}
 1+(b_i-z+\omega_B) m_{\mu_A\boxplus \mu_B}=m_{\mu_A\boxplus \mu_B}\Big(\frac{1}{m_{\mu_A\boxplus \mu_B}}+b_i-z+\omega_B\Big)=  m_{\mu_A\boxplus \mu_B}(b_i-\omega_A)\,. 
 \end{align*}
 Using (\ref{17020502}) and $\|A\|, \|B\|\leq C$, we get $|m_{\mu_A\boxplus \mu_B}(b_i-\omega_A)|\gtrsim 1$.  This together with (\ref{17072440}) implies  (\ref{17072441}). 
 
To prove (\ref{17020302}), we recall the definition of $P_i$ in (\ref{17011301}), which implies that
 \begin{align}
 \frac{1}{N}\sum_{i=1}^N (G_{ii}+T_i) \Upsilon=\frac{1}{N}\sum_{i=1}^N P_i= O_\prec(\Psi)\,. \label{17072447}  
 \end{align} 
Using the facts $\frac{1}{N}\sum_{i=1}^NG_{ii}=m_{\mu_A\boxplus\mu_B}+O_{\prec}(N^{-\frac{\gamma}{4}})$ (\cf (\ref{17020535})), and $\frac{1}{N}\sum_{i=1}^N T_i=O_\prec(N^{-\frac{\gamma}{4}})$,   and also $|m_{\mu_A\boxplus\mu_B}|\gtrsim 1$, we get  (\ref{17020302}) from (\ref{17072447}). 

Then, combining (\ref{17020302}) with the first estimate in (\ref{17020301}), we get 
\begin{align}
(\wt{B}G)_{ii}\ntr G-G_{ii}\ntr \wt{B}G=O_\prec(\Psi)\,.  \label{17072450}
\end{align}
Applying the identity (\ref{170725131}) and the definition of $\omega_B^c$, we can rewrite (\ref{17072450}) as 
\begin{align*}
\big((a_i-\omega_B^c)G_{ii}-1\big) \ntr G=O_\prec(\Psi)\,.
\end{align*}
As shown above that $|\ntr G|\gtrsim 1$ with high probability under the assumption (\ref{17020501}), we get  
$(a_i-\omega_B^c)G_{ii}-1=O_\prec(\Psi)$. By (\ref{170725110}) and (\ref{17020502}), we also note that $|a_i-\omega_B^c|\gtrsim 1$ with high probability.  This  further implies the first estimate in  (\ref{17020303}). 

Finally, plugging (\ref{17072450}) back to (\ref{17072432}), we can improve the right hand side of  (\ref{17072440}) to $O_\prec(\Psi)$. Then the second estimate in (\ref{17020303}) follows.  This completes the proof of Proposition \ref{pro.17020310}.  
 \end{proof}

\section{Rough fluctuation averaging for general linear combinations} \label{s. rough bound} 
In this section, we prove a rough fluctuation averaging estimate for the basic quantities $Q_i$ defined in  (\ref{17021701}). 
From (\ref{17072450}), we see that 
\begin{align}
 |Q_i(z)|\prec \Psi\,,\qquad\qquad i\in\llbracket 1,N\rrbracket\,,\qquad z \in \mathcal{D}_\tau(\eta_{\rm m},\eta_\mathrm{M})\,,\label{17021737}
\end{align}
if the assumptions  of Proposition \ref{pro.17020310} hold.

Recall the definition of the control parameters $\Pi$ and $\Pi_i$ in~\eqref{17012101} and~\eqref{17020550}, respectively. The following proposition states that the average of the $Q_i$'s is typically smaller than an individual $Q_i$. 
\begin{pro} \label{lem. rough fluctuation averaging} Fix a $z\in \mathcal{D}_\tau(\eta_{\rm m},\eta_\mathrm{M})$. Suppose that the assumptions  of Proposition \ref{pro.17020310} hold. Set $X_i=I$ or $\wt{B}^{\la i\ra}$.   Let $d_1, \ldots, d_N\in \mathbb{C}$ be possibly $H$-dependent quantities satisfying $\max_j|d_j|\prec 1$. Assume that they depend only weakly on the randomness in the sense that the following hold, for all $i,j\in \llbracket 1, N\rrbracket$,
\begin{align}
 \frac{1}{N}  \sum_k^{(i)} \frac{\partial d_j}{\partial g_{ik}} \mathbf{e}_k^* X_i G\mathbf{e}_i=O_\prec\big(\Psi^2\Pi_i^2\big)\,,\qquad\qquad  \frac{1}{N} \sum_k^{(i)} \frac{\partial d_j}{\partial g_{ik}} \mathbf{e}_k^* X_i \mathring{\mathbf{g}}_i=O_\prec\big(\Psi^2\Pi_i^2\big)\,, 
\label{17022530}
\end{align}
and the same bounds hold when the $d_j$'s are replaced by their complex conjugates $\overline{d_j}$.
Suppose that $\Pi(z)\prec \hat{\Pi}(z)$ for some deterministic and positive function $\hat{\Pi}(z)$ that satisfies ${\frac{1}{\sqrt{N\sqrt{\eta}}}}+\Psi^2\prec \hat{\Pi}\prec \Psi$.
Then,
\begin{align}
\Big|\frac{1}{N} \sum_{i=1}^N d_i  Q_i\Big | \prec \Psi \hat{\Pi}\,. \label{170723113}
\end{align}
\end{pro}
We remark that whenever the $d_j$'s are deterministic, (\ref{17022530}) trivially holds. However, we will also need~(\ref{170723113}) with certain random $d_j$'s that satisfy (\ref{17022530}).

  For any $d_i$'s satisfying the assumption in Proposition \ref{lem. rough fluctuation averaging}, we introduce the notation 
\begin{align}
\mathfrak{m}^{(k,l)}\deq\Big(\frac{1}{N}\sum_{i=1}^N   d_i  Q_i\Big)^k\Big(\frac{1}{N}\sum_{i=1}^N   \overline{d_i}\; \overline{ Q_i}\Big)^l\,,\qquad\qquad k,l\in\N\,. \label{17071810}
\end{align}
Similarly to Lemma \ref{lem.17021230}, it suffices to prove the following recursive moment estimate.
\begin{lem} \label{lem.17021231} Fix a $z\in \mathcal{D}_\tau(\eta_{\rm m},\eta_\mathrm{M})$. Suppose that the assumptions  of Proposition \ref{lem. rough fluctuation averaging} hold. Then, for any fixed integer $p\geq 1$, we have
\begin{align}
\mathbb{E}\big[ \mathfrak{m}^{(p,p)}\big]=\mathbb{E}\big[O_\prec(\hat{\Pi}^2)\mathfrak{m}^{(p-1,p)}\big]+\mathbb{E}\big[O_\prec(\Psi^2\hat{\Pi}^2) \mathfrak{m}^{(p-2,p)}\big]+\mathbb{E}\big[O_\prec(\Psi^2\hat{\Pi}^2) \mathfrak{m}^{(p-1,p-1)}\big]. \label{17071833}
\end{align}
\end{lem}
\begin{proof}[Proof of Proposition \ref{lem. rough fluctuation averaging}]
Like the proof of (\ref{17020301}) from Lemma \ref{lem.17021230}, with Lemma \ref{lem.17021231}, we can get (\ref{170723113}) by applying Young's and Markov's inequalities. This completes the proof of Proposition~\ref{lem. rough fluctuation averaging}.
\end{proof}
\begin{proof}[Proof of Lemma \ref{lem.17021231}] We first claim that it suffices to prove the following statement: If $|\Upsilon(z)|\prec \hat{\Upsilon}(z)$ for any deterministic and positive function $\hat{\Upsilon}(z)\leq\Psi(z)$, then  
\begin{align}
\mathbb{E}\big[ \mathfrak{m}^{(p,p)}\big]=&\mathbb{E}\big[(O_\prec(\hat{\Pi}^2)+O_\prec(\Psi \hat{\Upsilon}))\mathfrak{m}^{(p-1,p)}\big]+\mathbb{E}\big[O_\prec(\Psi^2\hat{\Pi}^2) \mathfrak{m}^{(p-2,p)}\big]\nonumber\\
&\qquad+\mathbb{E}\big[O_\prec(\Psi^2\hat{\Pi}^2) \mathfrak{m}^{(p-1,p-1)}\big]\,. \label{17072903}
\end{align}
Indeed, the same as the proof of (\ref{17020301}) from Lemma \ref{lem.17021230},  we can  again apply Young's inequality and Markov's inequality to get, for any $d_i$'s satisfying the assumptions in Proposition \ref{lem. rough fluctuation averaging}, that~\eqref{17072903} implies
\begin{align}
\Big|\frac{1}{N}\sum_{i=1}^N d_i Q_i\Big|\prec \hat{\Pi}^2+\Psi \hat{\Upsilon}+ \Psi\hat{\Pi}\prec \Psi \hat{\Upsilon}+ \Psi\hat{\Pi}\,, \label{17072901}
\end{align}
where in the last step we used the assumption $\hat{\Pi}\prec\Psi$. 

Next, recall from (\ref{17011302}) that
\begin{align*}
\Upsilon=-\frac{1}{N}\sum_{i=1}^N a_i Q_i\,.
\end{align*}
Choosing $d_i=a_i$ for all $i$, we get from (\ref{17072901}) 
\begin{align}
|\Upsilon|\prec \Psi \hat{\Upsilon}+ \Psi\hat{\Pi}\prec N^{-\frac{\gamma}{4}} \hat{\Upsilon}+ \Psi\hat{\Pi}\,. \label{17072902}
\end{align}
Using the right hand side of (\ref{17072902}) as a new deterministic bound of  $\Upsilon$ instead of the initial $\hat{\Upsilon}$ in  (\ref{17072903}), and perform the above argument iteratively, we can finally get 
\begin{align}
|\Upsilon|\prec  \Psi\hat{\Pi}\,. \label{19030601}
\end{align}
Hence, at the end, we can choose $\hat{\Upsilon}= \Psi\hat{\Pi}$ in (\ref{17072903}) and get 
\begin{align}
\mathbb{E}\big[ \mathfrak{m}^{(p,p)}\big]=&\mathbb{E}\big[(O_\prec(\hat{\Pi}^2)+O_\prec(\Psi^2 \hat{\Pi}))\mathfrak{m}^{(p-1,p)}\big]+\mathbb{E}\big[O_\prec(\Psi^2\hat{\Pi}^2) \mathfrak{m}^{(p-2,p)}\big]\nonumber\\
&\qquad+\mathbb{E}\big[O_\prec(\Psi^2\hat{\Pi}^2) \mathfrak{m}^{(p-1,p-1)}\big].  \label{17072910}
\end{align}
 Observe that by the assumption ${\frac{1}{\sqrt{N\sqrt{\eta}}}}+\Psi^2\prec \hat{\Pi}$,  the $O_\prec(\Psi^2 \hat{\Pi})$ term can be absorbed in the $O_\prec(\hat{\Pi}^2)$ term in (\ref{17072910}). Hence, we conclude (\ref{17071833}) from (\ref{17072903}). Therefore, in the sequel, we will focus on proving~\eqref{17072903}. 

Denote by  $D\deq\text{diag}(d_i)_{i=1}^N$. 
We first write
\begin{align}
\frac{1}{N} \sum_{i=1}^N d_i  Q_i
= \frac{1}{N}\sum_{i=1}^N (\wt{B}G)_{ii} \big(d_i \ntr G- \ntr DG\big)=\frac{1}{N}\sum_{i=1}^N (\wt{B}G)_{ii}\ntr G \tau_{i1} , \label{17021232}
\end{align}
where we introduced the notation 
\begin{align}
\tau_{i1}\deq d_i-\frac{\ntr D G}{\ntr G}.  \label{17021305}
\end{align}

Similarly to the proof of (\ref{17020301}), we approximate $(\wt{B}G)_{ii}$ by $-\mathring{S}_i$ (\cf (\ref{17020531})), and then perform an integration by parts using~\eqref{17021237} with respect to $\mathring{\mathbf{g}}_i$ in $\mathring{S}_i$.   More specifically, we write
\begin{align}
\mathbb{E}\big[ \mathfrak{m}^{(p,p)}\big] &=\frac{1}{N}\sum_{i=1}^N  \mathbb{E}\Big[ (\wt{B}G)_{ii} \ntr G \tau_{i1}\mathfrak{m}^{(p-1,p)}\Big]\nonumber\\
&=-\frac{1}{N}\sum_{i=1}^N  \mathbb{E}\Big[ \mathring{S}_{i} \ntr G \tau_{i1}\mathfrak{m}^{(p-1,p)}\Big]+
\mathbb{E}\Big[ \varepsilon_1\mathfrak{m}^{(p-1,p)}\Big] \,, \label{17021240}
\end{align}
where we used the notation 
\begin{align}
\varepsilon_{1}\deq\frac{1}{N}\sum_{i=1}^N   \varepsilon_{i1} \ntr G \tau_{i1}. \label{17021320}
\end{align}
Here $\varepsilon_{i1}$ is defined in (\ref{17071805}). To ease the presentation, we further introduce the notation
\begin{align}
 \tau_{i2}\deq-\tau_{i1} \ntr \wt{B}G.  \label{1702130512}
\end{align}
 Using assumption (\ref{17020501}), (\ref{17020535}), and also~\eqref{17020502}, one checks that $|\tau_{i1}|\prec 1$, $|\tau_{i2}|\prec 1$, for all $i\in \llbracket 1, N\rrbracket $.

Analogously to  (\ref{17020540}), applying~(\ref{17021237}) to the first term on the right hand side of (\ref{17021240}), we obtain
\begin{align}
\frac{1}{N}\sum_{i=1}^N  \mathbb{E}\Big[ \mathring{S}_{i} &\ntr G \tau_{i1}\mathfrak{m}^{(p-1,p)}\Big]=\frac{1}{N^2} \sum_{i=1}^N  \sum_k^{(i)} \mathbb{E}\Big[ \frac{1}{\|\mathbf{g}_i\| }\frac{\partial (\mathbf{e}_k^*\wt{B}^{\la i\ra} G\mathbf{e}_i)}{\partial g_{ik}}\ntr G \tau_{i1} \mathfrak{m}^{(p-1,p)}\Big]\nonumber\\
&\quad+ \frac{1}{N^2} \sum_{i=1}^N  \sum_k^{(i)} \mathbb{E}\Big[ \frac{\partial \|\mathbf{g}_i\| ^{-1}}{\partial g_{ik}} \mathbf{e}_k^*\wt{B}^{\la i\ra} G\mathbf{e}_i\ntr G \tau_{i1} \mathfrak{m}^{(p-1,p)}\Big]\nonumber\\
&\quad +\frac{1}{N^2} \sum_{i=1}^N   \sum_k^{(i)} \mathbb{E}\Big[ \frac{1}{\|\mathbf{g}_i\| } \mathbf{e}_k^*\wt{B}^{\la i\ra} G\mathbf{e}_i \frac{\partial (\ntr G \tau_{i1})}{\partial g_{ik}}\mathfrak{m}^{(p-1,p)} \Big]\nonumber\\
&\quad+ \frac{p-1}{N^2} \sum_{i=1}^N   \sum_k^{(i)} \mathbb{E}\Big[ \frac{1}{\|\mathbf{g}_i\| } \mathbf{e}_k^*\wt{B}^{\la i\ra} G\mathbf{e}_i \ntr G \tau_{i1} \Big(\frac{1}{N}\sum_{j=1}^N  \frac{\partial   (d_jQ_j)}{\partial g_{ik}}\Big)\mathfrak{m}^{(p-2,p)}\Big]\nonumber\\
&\quad+ \frac{p}{N^2} \sum_{i=1}^N   \sum_k^{(i)} \mathbb{E} \Big[ \frac{1}{\|\mathbf{g}_i\| } \mathbf{e}_k^*\wt{B}^{\la i\ra} G\mathbf{e}_i \ntr G \tau_{i1} \Big(\frac{1}{N}\sum_{j=1}^N   \frac{\partial (\overline{d_j}\overline{ Q_j})}{\partial g_{ik}}\Big)\mathfrak{m}_i^{(p-1,p-1)}\Big]. \label{17021250}
\end{align}
First, we estimate the first term on the right hand side of (\ref{17021250}). Using (\ref{17021008}) and the bound
\begin{align*}
\frac{1}{N}\sum_{i=1}^N \Pi_i^2\leq 2 \Pi^2,
\end{align*}
 we have
\begin{align*}
&\frac{1}{N^2} \sum_{i=1}^N  \sum_k^{(i)} \frac{1}{\|\mathbf{g}_i\| }\frac{\partial (\mathbf{e}_k^*\wt{B}^{\la i\ra} G\mathbf{e}_i)}{\partial g_{ik}}\ntr G \tau_{i1}=-\frac{1}{N} \sum_{i=1}^N   \big( G_{ii}\ntr \wt{B}G-(G_{ii}+T_i)\Upsilon\big) \tau_{i1}\nonumber\\
 &\qquad\qquad \qquad\qquad+\frac{1}{N^2} \sum_{i=1}^N  \sum_k^{(i)} \Big(\mathring{T}_i-\frac{1}{\|\mathbf{g}_i\| }\frac{\partial (\mathbf{e}_k^* G\mathbf{e}_i)}{\partial g_{ik}} \Big) \tau_{i2}+\varepsilon_{2}+O_\prec(\Pi^2)\,,
 \end{align*}
 where we have introduced
 \begin{align}
 \varepsilon_{2}\deq\frac{1}{N} \sum_{i=1}^N  \frac{1}{\|\mathbf{g}_i\| } \tau_{i1}\varepsilon_{i2}\,; \label{17021310}
 \end{align}
 see~(\ref{17021004}) for the definition of $\varepsilon_{i2}$. 
 According to the definition in (\ref{17021305}), we observe that 
 \begin{align}
&\frac{1}{N} \sum_{i=1}^N   \big( G_{ii}\ntr \wt{B}G-(G_{ii}+T_i)\Upsilon\big)\tau_{i1}=\frac{1}{N^2} \sum_{i=1}^N   G_{ii} \tau_{i1}\big(\ntr \wt{B}G-\Upsilon\big) - \frac{1}{N} \sum_{i=1}^N   T_i \tau_{i1} \Upsilon  =O_\prec(\Psi \hat{\Upsilon})\,. \label{19030922}
 \end{align}
 Here in the last step we used the facts 
 \begin{align}
 \sum_{i=1}^N   G_{ii} \tau_{i1}=0\,,\qquad\qquad \frac{1}{N}\sum_{i=1}^N    T_i \tau_{i1} \Upsilon =O_\prec(\Psi \hat{\Upsilon})\,, \label{170728100}
 \end{align}
where the second estimate is implied by  the second estimate in (\ref{17020303}), and the assumption that $|\Upsilon|\prec \hat{\Upsilon}$.

 Therefore, for the first term on the right hand side of (\ref{17021250}), we have
 \begin{align}
 &\frac{1}{N^2} \sum_{i=1}^N  \sum_k^{(i)} \mathbb{E}\Big[ \frac{1}{\|\mathbf{g}_i\| }\frac{\partial (\mathbf{e}_k^*\wt{B}^{\la i\ra} G\mathbf{e}_i)}{\partial g_{ik}}\ntr G \tau_{i1}  \mathfrak{m}^{(p-1,p)}\Big]\nonumber\\
 &= \frac{1}{N^2} \sum_{i=1}^N  \sum_k^{(i)}\mathbb{E}\Big[ \Big(\mathring{T}_i-\frac{1}{\|\mathbf{g}_i\| }\frac{\partial (\mathbf{e}_k^* G\mathbf{e}_i)}{\partial g_{ik}} \Big)\tau_{i2} \mathfrak{m}^{(p-1,p)}\Big]+\mathbb{E}\big[(\varepsilon_{2}+O_\prec(\Pi^2)+O_\prec(\Psi \hat{\Upsilon}))\mathfrak{m}^{(p-1,p)}\big]\nonumber\\
 &= \frac{1}{N^2} \sum_{i=1}^N   \sum_k^{(i)} \mathbb{E}\Big[ \frac{\partial \|\mathbf{g}_i\| ^{-1}}{\partial g_{ik}} \mathbf{e}_k^* G\mathbf{e}_i \tau_{i2} \mathfrak{m}^{(p-1,p)}\Big] +\frac{1}{N^2} \sum_{i=1}^N   \sum_k^{(i)} \mathbb{E}\Big[ \frac{1}{\|\mathbf{g}_i\| }\frac{\partial \tau_{i2}}{\partial g_{ik}} \mathbf{e}_k^* G\mathbf{e}_i  \mathfrak{m}^{(p-1,p)}\Big]\nonumber\\
&\qquad+ \frac{p-1}{N^2} \sum_{i=1}^N   \sum_k^{(i)} \mathbb{E}\Big[ \frac{1}{\|\mathbf{g}_i\| } \mathbf{e}_k^* G\mathbf{e}_i  \tau_{i2} \Big(\frac{1}{N}\sum_{j=1}^N    \frac{\partial  (d_jQ_j)}{\partial g_{ik}}\Big) \mathfrak{m}^{(p-2,p)}\Big]\nonumber\\
&\qquad+ \frac{p}{N^2} \sum_{i=1}^N   \sum_k^{(i)} \mathbb{E} \Big[ \frac{1}{\|\mathbf{g}_i\| } \mathbf{e}_k^* G\mathbf{e}_i  \tau_{i2} \Big(\frac{1}{N}\sum_{j=1}^N   \frac{\partial (\overline{d_j}\overline{ Q_j})}{\partial g_{ik}}\Big)\mathfrak{m}^{(p-1,p-1)}\Big]\nonumber\\
&\qquad+\mathbb{E}\big[\big(\varepsilon_{2}+O_\prec(\Pi^2)+O_\prec(\Psi \hat{\Upsilon})\big)\mathfrak{m}^{(p-1,p)}\big], \label{17021252}
 \end{align}
 where the second equation is obtained analogously to  (\ref{17021024}), by writing $\mathring{T}_i=\sum_k^{(i)}\bar{g}_{ik} \mathbf{e}_k^*G\mathbf{e}_i/\|\mathbf{g}_i\| $ and performing integration by parts with respect to the $g_{ik}$'s.
 
 According to (\ref{17021240}),  (\ref{17021250}),  and (\ref{17021252}), it suffices to estimate the last term on the right side of (\ref{17021240}), the last four  terms on the right side of (\ref{17021250}), and all the terms on the right side of (\ref{17021252}). All the desired estimates can be derived from the following lemma.
  \begin{lem} \label{lem.17021301}  Fix a $z\in \mathcal{D}_\tau(\eta_{\rm m},\eta_\mathrm{M})$. Suppose that the assumptions  of Proposition \ref{lem. rough fluctuation averaging} hold, especially (\ref{17022530}) holds for $d_1, \ldots, d_N$ in the definition (\ref{17071810}). Let  $\tilde{d}_1, \ldots, \tilde{d}_N \in \mathbb{C}$ be any (possibly random) numbers with the bound $\max_i|\tilde{d}_i|\prec 1$. Let $Q$ be any  (possibly random) diagonal matrix that satisfies $\| Q\|\prec 1$.  Set $X=I$ or $A$, and set $X_i=I$ or $\wt{B}^{\la i\ra}$. Then we have 
 \begin{align}
 &\frac{1}{N^2} \sum_{i=1}^N   \sum_k^{(i)} \tilde{d}_i \frac{\partial \|\mathbf{g}_i\| ^{-1}}{\partial g_{ik}} \mathbf{e}_k^* X_iG\mathbf{e}_i=O_\prec(\frac{1}{N})\,,\label{17021302}\\
 &  \frac{1}{N^2} \sum_{i=1}^N   \sum_k^{(i)} \tilde{d}_i \ntr \Big(Q X\frac{\partial G}{\partial g_{ik}}\Big) \mathbf{e}_k^* X_iG\mathbf{e}_i=O_\prec(\Psi^2\Pi^2) \,,\label{17021303}
 \end{align} 
 and  the same estimate holds if we replace $\frac{\partial G}{\partial g_{ik}}$ by the complex conjugate $\frac{\partial \overline{G}}{\partial g_{ik}}$  in (\ref{17021303}).
 Further, we~have 
 \begin{multline}
\mathbb{E} \big[ \varepsilon_j\mathfrak{m}^{(p-1,p)}\big]=\mathbb{E}\big[O_\prec(\hat{\Pi}^2)\mathfrak{m}^{(p-1,p)}\big]\\+\mathbb{E}\big[O_\prec(\Psi^2\hat{\Pi}^2) \mathfrak{m}^{(p-2,p)}\big]+\mathbb{E}\big[O_\prec(\Psi^2\hat{\Pi}^2) \mathfrak{m}^{(p-1,p-1)}\big]\,, \qquad j=1,2.\label{17021311}
 \end{multline}
 \end{lem}
We postpone the proof of Lemma \ref{lem.17021301} for a moment and continue with the proof of  Lemma \ref{lem.17021231} instead. 

The second term of (\ref{17021250}) and the first term of (\ref{17021252}) are directly estimated by (\ref{17021302}).
Using the~definition of $\tau_{i1}$ in (\ref{17021305}) and of $\tau_{i2}$ in (\ref{1702130512}), the boundedness of the tracial quantities (\cf (\ref{17020535})), and the chain rule, we get the estimate on the third term of (\ref{17021250}) and the second term of (\ref{17021252}), using  (\ref{17021303}) and the assumption (\ref{17022530}).  For the last two terms of (\ref{17021250}), and the third and fourth terms of (\ref{17021252}), we~note~that
\begin{align*}
\frac{1}{N} \sum_{j=1}^N   d_j Q_j=\ntr D\wt{B}G\,\ntr G-\ntr\wt{B}G \,\ntr DG=\ntr D \,\ntr G-\ntr DG- \ntr DAG\,\ntr G+\ntr AG\, \ntr DG\,,
\end{align*}
where in the last step we used the first  identity of (\ref{17020508}). Hence, by the chain rule,  the fourth term of (\ref{17021250}) and the third term of (\ref{17021252}) are estimated with the aid of (\ref{17021303}) and (\ref{17022530}). The last term of (\ref{17021250}) and the fourth term of (\ref{17021252}) can be estimated analogously.
Finally, the  estimates of the second term of (\ref{17021240}) and  the last term of (\ref{17021252}) are given by  (\ref{17021311}). Thus we conclude the proof of Lemma~\ref{lem.17021231}.
\end{proof}

\begin{proof}[Proof of Lemma \ref{lem.17021301}]
Note that  (\ref{17021302}) and (\ref{17021303}) follow from the first and the second last estimates in (\ref{17021202}), respectively, by averaging over the index $i$.  Hence, it suffices to prove (\ref{17021311}). Recall the definition of $\varepsilon_1$ from (\ref{17021320}) and of $\varepsilon_{2}$ from (\ref{17021310}).

We first consider $\mathbb{E}[\varepsilon_1\mathfrak{m}^{(p-1,p)}]$. Recall the definition of $\varepsilon_{i1}$ from (\ref{17071805}). Using (\ref{17020302}),  (\ref{17020303}), the first bound in (\ref{17020505}), and (\ref{17021540}),
we have
\begin{align}
\varepsilon_{i1}=   \frac{\mathbf{h}_i^* \wt{B}^{\la i\ra} \mathbf{h}_i}{a_i-\omega_B^c}+O_\prec\big(\frac{\Psi}{\sqrt{N}}\big)= \frac{\mathring{\mathbf{h}}_i^* \wt{B}^{\la i\ra} \mathring{\mathbf{h}}_i}{a_i-\omega_B^c}+O_\prec(\hat{\Pi}^2)\,. \label{17021550}
\end{align}
Here the last step follows from  the assumption $\frac{1}{N\sqrt{\eta}} \prec \hat{\Pi}^2$, and that  $\mathbf{h}_i=\mathring{\mathbf{h}}_i+\frac{g_{ii}}{\|\mathbf{g}_i\|}\mathbf{e}_i$ with 
\begin{align*}
|g_{ii}|\prec \frac{1}{\sqrt{N}}\,, \qquad\qquad \mathring{\mathbf{h}}_i^* \wt{B}^{\la i\ra} \mathbf{e}_i=b_i\mathring{\mathbf{h}}_i^* \mathbf{e}_i=0\,. 
\end{align*}
Hence, by the definition of $\varepsilon_1$ in~(\ref{17021320}), we have
\begin{align*}
\varepsilon_1=\frac{1}{N}\sum_{i=1}^N \mathring{\mathbf{h}}_i^* \wt{B}^{\la i\ra} \mathring{\mathbf{h}}_i  \frac{d_i \ntr G-\ntr DG}{a_i-\omega_B^c}+O_\prec(\hat{\Pi}^2)= \frac{1}{N}\sum_{i=1}^N \mathring{\mathbf{h}}_i^* \wt{B}^{\la i\ra} \mathring{\mathbf{h}}_i  \tau_{i3}+O_\prec(\hat{\Pi}^2)\,,
\end{align*}
where  we   introduced the notation
\begin{align*}
\tau_{i3}\deq  \frac{d_i \ntr G-\ntr DG}{a_i-\omega_B^c}\,.
\end{align*}
Using the integration by parts formula~\eqref{17021237}, we obtain
\begin{align}
\frac{1}{N}\sum_{i=1}^N \mathbb{E}\big[ \mathring{\mathbf{h}}_i^* \wt{B}^{\la i\ra} \mathring{\mathbf{h}}_i  \tau_{i3}\mathfrak{m}^{(p-1,p)}\big]&= \frac{1}{N}\sum_{i=1}^N \sum_{k}^{(i)}\mathbb{E}\big[ \frac{1}{\|\mathbf{g}_i\| ^2} \bar{g}_{ik} \mathbf{e}_k^* \wt{B}^{\la i\ra} \mathring{\mathbf{g}}_i  \tau_{i3}\mathfrak{m}^{(p-1,p)}\big]\nonumber\\
&=\frac{1}{N^2}\sum_{i=1}^N \sum_{k}^{(i)}\mathbb{E}\Big[ \frac{\partial \big(\|\mathbf{g}_i\| ^{-2} \mathbf{e}_k^* \wt{B}^{\la i\ra} \mathring{\mathbf{g}}_i  \tau_{i3}\mathfrak{m}^{(p-1,p)}\big)}{\partial g_{ik}}\Big]. \label{17021531}
\end{align}
Note that
\begin{align}
&\frac{\partial \big(\|\mathbf{g}_i\| ^{-2} \mathbf{e}_k^* \wt{B}^{\la i\ra} \mathring{\mathbf{g}}_i  \tau_{i3}\mathfrak{m}^{(p-1,p)}\big)}{\partial g_{ik}}=\frac{ \partial \|\mathbf{g}_i\| ^{-2} }{ \partial g_{ik}} \mathbf{e}_k^* \wt{B}^{\la i\ra} \mathring{\mathbf{g}}_i  \tau_{i3}\mathfrak{m}^{(p-1,p)}+ \|\mathbf{g}_i\| ^{-2} \mathbf{e}_k^* \wt{B}^{\la i\ra} \mathbf{e}_k \tau_{i3}\mathfrak{m}^{(p-1,p)}\nonumber\\
&\quad +\|\mathbf{g}_i\| ^{-2} \mathbf{e}_k^* \wt{B}^{\la i\ra} \mathring{\mathbf{g}}_i  \frac{\partial \tau_{i3}}{\partial g_{ik}}\mathfrak{m}^{(p-1,p)}+ (p-1) \|\mathbf{g}_i\| ^{-2} \mathbf{e}_k^* \wt{B}^{\la i\ra} \mathring{\mathbf{g}}_i  \tau_{i3} \Big(\frac{1}{N}\sum_{j=1}^N   \frac{\partial (d_jQ_j)}{\partial g_{ik}}\Big)\mathfrak{m}^{(p-2,p)}\nonumber\\
&\quad + p \|\mathbf{g}_i\| ^{-2} \mathbf{e}_k^* \wt{B}^{\la i\ra} \mathring{\mathbf{g}}_i  \tau_{i3} \Big(\frac{1}{N}\sum_{j=1}^N   \frac{\partial (\overline{d_j}\overline{Q_j})}{\partial g_{ik}}\Big)\mathfrak{m}^{(p-1,p-1)}\,. \label{17021530}
\end{align}
Notice that $\frac{\partial \|\mathbf{g}_i\| ^{-2}}{\partial g_{ik}}=-\|\mathbf{g}_i\| ^{-4}\bar{g}_{ik}$ and that $\tau_{i3}=O_\prec (1)$. In addition, we also have that
\begin{align*}
\sum_k^{(i)}\bar{g}_{ik}\mathbf{e}_k=\mathring{\mathbf{g}}_i^*\,, \qquad\qquad \sum_{k}^{(i)} \mathbf{e}_k^* \wt{B}^{\la i\ra}\mathbf{e}_k=\text{Tr} B-b_i=b_i\,.  
\end{align*}
 Denoting by $\tilde{d}_1,\ldots, \tilde{d}_N\in \mathbb{C}$ generic (possibly random) numbers with $\max_i |\tilde{d}_i|\prec 1$, we see that the contributions from the first two terms on the right side of (\ref{17021530}) to (\ref{17021531}) follow from the~estimates  
\begin{align*}
\frac{1}{N^2}\sum_{i=1}^N  \tilde{d}_i\mathring{\mathbf{g}}_i^* \wt{B}^{\la i\ra} \mathring{\mathbf{g}}_i =O_\prec(\frac{1}{N})\,, \qquad\qquad
\frac{1}{N^2}\sum_{i=1}^N  \tilde{d}_i b_i \mathbf{e}_k^* \wt{B}^{\la i\ra} \mathbf{e}_k =O_\prec(\frac{1}{N}) \,.
\end{align*}
Here $\tilde{d}_i$ includes $\tau_{i3}$ and an appropriate power of $\|\mathbf{g}_i\|$. 
In addition, for the estimate of the remaining terms in (\ref{17021530}), 
  we claim that, for $X_i=I, \wt{B}^{\la i\ra}$,
\begin{align}
&\frac{1}{N^2}\sum_{i=1}^N \sum_k^{(i)} \tilde{d}_i \mathbf{e}_k^* X_i \mathring{\mathbf{g}}_i \frac{\partial \tau_{i3}}{\partial g_{ik}}=O_\prec(\Psi^2\Pi^2)\,,\\
 &\frac{1}{N^2}\sum_{i=1}^N \sum_k^{(i)} \tilde{d}_i \mathbf{e}_k^* X_i \mathring{\mathbf{g}}_i  \Big(\frac{1}{N}\sum_{j=1}^N   \frac{\partial (d_j Q_j)}{\partial g_{ik}}\Big)=O_\prec(\Psi^2\Pi^2)\,,\label{17022401}\\
  &\frac{1}{N^2}\sum_{i=1}^N \sum_k^{(i)} \tilde{d}_i \mathbf{e}_k^* X_i \mathring{\mathbf{g}}_i  \Big(\frac{1}{N}\sum_{j=1}^N   \frac{\partial (\overline{d_j} \overline{Q_j})}{\partial g_{ik}}\Big)=O_\prec(\Psi^2\Pi^2)\,. \label{17022402}
\end{align}
The  above three bounds follows from the last estimate in (\ref{17021202}) and the chain rule.   Hence, we  conclude the proof of (\ref{17021311}) with $j=1$. 

The proof of (\ref{17021311}) for $j=2$ is similar to $j=1$. Recall the definition of $\varepsilon_{i2}$ from (\ref{17021004}).  Using (\ref{17020302}),  (\ref{17020303}), the first bound in (\ref{17020505}), and also the bounds in (\ref{17021540}), we have
\begin{align*}
\varepsilon_{i2}= &\big( \|\mathbf{g}_i\| ^2-1\big)G_{ii}\ntr \wt{B}G+O_\prec\Big(\frac{\Psi}{\sqrt{N}}\Big)=\big( \mathring{\mathbf{g}}_i^*\mathring{\mathbf{g}}_i-1\big)\frac{\ntr \wt{B}G}{a_i-\omega_B^c}+O_\prec(\hat{\Pi}^2)\,,
\end{align*}
which possesses a very similar structure as (\ref{17021550}).   The remaining proof  is nearly the same as the case for $\varepsilon_{1}$; it suffices to replace $\mathring{\mathbf{g}}_i^*\wt{B}^{\la i\ra}\mathring{\mathbf{g}}_i$ by $\mathring{\mathbf{g}}_i^*\mathring{\mathbf{g}}_i$ throughout the proof. We thus omit the details.  Hence, we conclude the proof for Lemma \ref{lem.17021301}. 
\end{proof}

\section{Optimal fluctuation averaging} \label{s.optimal FL}
In this section, we establish the optimal fluctuation averaging estimate for a special linear combinations of the $Q_i$'s and their analogues, the $\mathcal{Q}_i$'s (see (\ref{17071820}) below), under assumption (\ref{17020501}).

Recall the definitions of the approximate subordination functions $\omega_A^c$ and $\omega_B^c$ in~\eqref{17072550}. We denote
\begin{align}
\Lambda_A\deq\omega_A^c-\omega_A\,,\qquad \Lambda_B\deq\omega_B^c-\omega_B\,, \qquad  \Lambda\deq|\Lambda_A|+|\Lambda_B|\,.\label{le gros lambda}
\end{align}
Recall $\mathcal{S}_{AB}$, $\mathcal{T}_A$ and $\mathcal{T}_B$ defined in (\ref{17080110}).  For brevity, in the sequel, we use the shorthand notation 
\begin{align*}
\mathcal{S}\equiv\mathcal{S}_{AB}. 
\end{align*}

\begin{pro} \label{pro.17021715}Fix a $z=E+\ii\eta\in \mathcal{D}_\tau(\eta_{\rm m},\eta_\mathrm{M})$. Suppose that the assumptions  of Proposition \ref{pro.17020310} hold.  Suppose that $\Lambda(z)\prec \hat{\Lambda}(z)$, for some deterministic and positive function $\hat{\Lambda}(z)\prec N^{-\frac{\gamma}{4}}$, then
\begin{align}
&\Big|\mathcal{S}\Lambda_\iota+\mathcal{T}_\iota\Lambda_\iota^2+O(\Lambda_\iota^3)\Big|\prec   \frac{\sqrt{(\Im m_{\mu_A\boxplus\mu_B}+\hat{\Lambda})(|\mathcal{S}|+\hat{\Lambda})}}{N\eta}+\frac{1}{(N\eta)^2},\qquad \iota=A, B\,. \label{17030301}
\end{align}
\end{pro}

Before commencing the proof of Proposition~\ref{pro.17021715}, we first claim that the control parameter
$\hat{\Pi}$ in Proposition~\ref{lem. rough fluctuation averaging}  can be chosen as the square root of the right side of~\eqref{17030301} as long as $\Lambda\prec \hat{\Lambda}$, \ie 
\begin{align}
\hat{\Pi}\deq  \Bigg(  \frac{\sqrt{(\Im m_{\mu_A\boxplus\mu_B}+\hat{\Lambda})(|\mathcal{S}|+\hat{\Lambda})}}{N\eta}+\frac{1}{(N\eta)^2}\Bigg)^{\frac12}  \,.\label{17072960}
\end{align}
Indeed, observe that  when $\Lambda\prec \hat{\Lambda}\prec N^{-\frac{\gamma}{4}}$, we obtain from the second line of (\ref{170730100}) that 
\begin{align}
|m_H-m_{\mu_A\boxplus\mu_B}| &=|m_{H}m_{\mu_A\boxplus\mu_B}|\Big|\frac{1}{m_H(z)}-\frac{1}{m_{\mu_A\boxplus\mu_B}(z)}\Big|\nonumber\\
&\leq  C|m_{H}m_{\mu_A\boxplus\mu_B}|\,\Lambda\, \leq C|m_{H}-m_{\mu_A\boxplus\mu_B}| \,\Lambda\,+C|m_{\mu_A\boxplus\mu_B}|^2 \,\Lambda\, . \label{17073101}
\end{align}
which together with the fact $|m_{\mu_A\boxplus \mu_B}|\leq C$ implies
\begin{align}
|m_H-m_{\mu_A\boxplus\mu_B}|\prec \Lambda\prec \hat{\Lambda} \,.\label{17073111}
\end{align} 
Therefore, recalling~\eqref{17012101}, we have
\begin{align*}
\Pi^2\prec \frac{\Im m_{\mu_A\boxplus\mu_B}+\hat{\Lambda}}{N\eta}\prec \frac{\sqrt{(\Im m_{\mu_A\boxplus\mu_B}+\hat{\Lambda})(|\mathcal{S}|+\hat{\Lambda})}}{N\eta}\prec \Psi^2,
\end{align*}
where in the last two steps, we used that $\Im m_{\mu_A\boxplus\mu_B}\lesssim |\mathcal{S}|\prec 1$; (\ref{17080120}) and (\ref{17080121}). In addition, from~(\ref{17080120}) and (\ref{17080121}), we also have $\Im m_{\mu_A\boxplus\mu_B} |\mathcal{S}|\gtrsim \eta $. Thus we also have 
\begin{align*}
\frac{1}{N\sqrt{\eta}}\prec   \frac{\sqrt{(\Im m_{\mu_A\boxplus\mu_B}+\hat{\Lambda})(|\mathcal{S}|+\hat{\Lambda})}}{N\eta}\,. 
\end{align*}
From the definition of $\Pi$ in (\ref{17012101}), we note that $\Pi\prec \sqrt{\frac{\Im m_{\mu_A\boxplus\mu_B}+\hat{\Lambda}}{N\eta}}$ when $\Lambda\prec\hat{\Lambda}$.  Hence, up to a $\frac{1}{N\eta}$ term,  $\hat{\Pi}$ defined in \eqref{17072960} is 
 a deterministic bound of $\Pi$ inside the spectrum but it can be much
larger than $\Pi$ in the outside regime where $\mathcal{S}\gg \Im m_{\mu_A\boxplus\mu_B}$ (\cf (\ref{17080120}) and (\ref{17080121})).

With the above notation, we can rewrite (\ref{17030301}) as 
\begin{align}
&\Big|\mathcal{S}\Lambda_\iota+\mathcal{T}_\iota\Lambda_\iota^2+O(\Lambda_\iota^3)\Big|\prec   \hat{\Pi}^2,\qquad\qquad \iota=A, B. 
\end{align}

Recall the definition of $Q_i$ from (\ref{17021701}). We also introduce their analogues 
\begin{align}
 \mathcal{Q}_i\equiv  \mathcal{Q}_i(z)\deq(\wt{A}\mathcal{G})_{ii}\ntr \mathcal{G}-\mathcal{G}_{ii} \ntr \wt{A}\mathcal{G}\,,\qquad\qquad i\in\llbracket 1,N\rrbracket\,. \label{17071820}
\end{align}
with $\wt A$ and $\mathcal{G}$ given in~\eqref{the tilda guys}. To prove Proposition \ref{pro.17021715}, we need an optimal fluctuation averaging for a very special combination of the $Q_i$'s and the $\mathcal{Q}_i$'s.  To this end, we define the functions $\Phi_1,\Phi_2\,:\, (\C^+)^3\longrightarrow \C$, 
\begin{align}
\Phi_1(\omega_1,\omega_2,z)\deq F_A(\omega_2)-\omega_1-\omega_2+z\,,\qquad\qquad\Phi_2(\omega_1,\omega_2,z)\deq F_B(\omega_1)-\omega_1-\omega_2+z\,,  \label{17073115}
\end{align}
where $F_A(\,\cdot\,)\equiv F_{\mu_A}(\,\cdot\,)$ and $F_B(\,\cdot\,)\equiv F_{\mu_B}(\,\cdot\,)$ denote the negative reciprocal Stieltjes transforms of $\mu_A$ and $\mu_B$.  
From the subordination equation (\ref{170730100}), we have $\Phi_1(\omega_A, \omega_B,z)=\Phi_2(\omega_A, \omega_B,z)=0$, with $\omega_A\equiv \omega_A(z)$ and $\omega_B\equiv\omega_B(z)$. For brevity, we use the shorthand notations
\begin{align}
\Phi_1^c\deq\Phi_1(\omega_A^c,\omega_B^c,z)\,,\qquad \qquad \Phi_2^c\deq \Phi_2(\omega_A^c,\omega_B^c,z)\,. \label{072960}
\end{align}
Further, we define the quantities
\begin{align}
\mathcal{Z}_1\deq \Phi_1^c+(F_A'(\omega_B)-1)\Phi_2^c\,,
\qquad \qquad
\mathcal{Z}_2\deq\Phi_2^c+(F_B'(\omega_A)-1)\Phi_1^c\,. \label{17021710}
\end{align}
We are going to show that $\mathcal{Z}_1$ and $\mathcal{Z}_2$ are actually certain linear combinations of the $Q_i$'s and the $\mathcal{Q}_i$'s.
We start with the identities
\begin{align}
\Phi_1^c=  -\frac{F_A(\omega_B^c)}{(m_H(z))^2} \frac{1}{N} \sum_{i=1}^N  \frac{1}{a_i-\omega_B^c} Q_i\,,\qquad\qquad \Phi_2^c
= -\frac{F_B(\omega_A^c)}{(m_H(z))^2} \frac{1}{N} \sum_{i=1}^N  \frac{1}{b_i-\omega_A^c} \mathcal{Q}_i\,,  \label{17021711}
\end{align}
which can be derived by combining  (\ref{17072550}), (\ref{170725130}) and (\ref{170725131}).
For all $i\in \llbracket 1, N\rrbracket $, we set
\begin{align}
& \mathfrak{d}_{i,1}\deq -\frac{F_A(\omega_B^c)}{(m_H(z))^2}\frac{1}{a_i-\omega_B^c}\,,\qquad\qquad \mathfrak{d}_{i,2}\deq -(F_A'(\omega_B)-1)\frac{F_B(\omega_A^c)}{(m_H(z))^2}\frac{1}{b_i-\omega_A^c}\,. \label{17022001}
\end{align}
According to the definition in (\ref{17021710}), (\ref{17021711}), and also (\ref{17022001}), we can write
\begin{align}
\mathcal{Z}_1=\frac{1}{N} \sum_{i=1}^N   \mathfrak{d}_{i,1}  Q_i+ \frac{1}{N} \sum_{i=1}^N   \mathfrak{d}_{i,2}  \mathcal{Q}_i\,, \label{17012202}
\end{align}
and $\mathcal{Z}_2$ can be represented in a similar way. 

Now, we choose $d_i=\mathfrak{d}_{i,1}, i\in \llbracket 1, N\rrbracket$, in Proposition \ref{lem. rough fluctuation averaging}. Observe that $\mathfrak{d}_{i,1}$ can be regarded as a smooth function of $\ntr \wt{B}G=1-\ntr(A-z)G$ and $m_H(z)=\ntr G$, according to the definition in (\ref{17022001}) and that of $\omega_B^c$ in (\ref{17072550}). Then, using the chain rule and the estimates of the tracial quantities in (\ref{17020535}), one can check that the first equation in  assumption   (\ref{17022530}) is satisfied for the choice  $d_i=\mathfrak{d}_{i,1}, i\in \llbracket 1, N\rrbracket$, by using~(\ref{17021202}).  The second equation can be checked analogously.  Hence, applying  Proposition \ref{lem. rough fluctuation averaging}, we get 
\begin{align}
|\Phi_1^c|\prec \Psi\hat{\Pi}\,,\qquad\qquad |\Phi_2^c|\prec \Psi\hat{\Pi}\,,  \label{170724501}
\end{align}
where $\hat{\Pi}$ is chosen as in (\ref{17072960}).

 The main technical task in this section is to establish the following estimates for $\mathcal{Z}_1$ and $\mathcal{Z}_2$, where the previous order $\Psi\hat{\Pi}$ bounds from (\ref{170723113}) are strengthened. 
\begin{pro} \label{pro. 17021720}
Fix $z\in \mathcal{D}_\tau(\eta_{\rm m},\eta_\mathrm{M})$. Suppose that the assumptions  of Proposition \ref{pro.17020310} hold and that $\Lambda(z)\prec \hat{\Lambda}(z)$ for some deterministic and positive function $\hat{\Lambda}(z)\leq N^{-\frac{\gamma}{4}}$. Choose $\hat{\Pi}(z)$ as (\ref{17072960}).  Then,
\begin{align}
|\mathcal{Z}_1|\prec  \hat{\Pi}^2\,, \qquad\qquad |\mathcal{Z}_2|\prec  \hat{\Pi}^2 \,. \label{17021740}
\end{align}
\end{pro}

We postpone the proof of Proposition \ref{pro. 17021720} and first prove Proposition \ref{pro.17021715} with the aid of Proposition~\ref{pro. 17021720}.
\begin{proof}[Proof of Proposition \ref{pro.17021715}] By assumption, we see that  $|\Lambda_A|, |\Lambda_B|\prec N^{-\frac{\gamma}{4}}$. First of all, expanding $\Phi_1^c$ and $\Phi_2^c$ around $(\omega_A, \omega_B)$ and using the subordination equations $\Phi_1(\omega_A, \omega_B,z)=\Phi_2(\omega_A, \omega_B,z)=0$, we get
\begin{align} 
&\Phi_1^c=-\Lambda_A+(F'_A(\omega_B)-1)\Lambda_B+\frac{1}{2}F''_A(\omega_B) \Lambda_B^2+O(\Lambda_B^3)\,,\nonumber\\
&  \Phi_2^c=-\Lambda_B+(F'_B(\omega_A)-1)\Lambda_A+\frac{1}{2}F''_B(\omega_A) \Lambda_A^2+O(\Lambda_A^3)\,.  \label{17021730}
\end{align}
We rewrite the second equation in (\ref{17021730}) as
\begin{align}
\Lambda_B=-\Phi_2^c+(F'_B(\omega_A)-1)\Lambda_A+\frac{1}{2}F''_B(\omega_A) \Lambda_A^2+O(\Lambda_A^3)\,. \label{17021731}
\end{align}
Substituting (\ref{17021731}) into the first equation in (\ref{17021730}) yields
\begin{align*}
\Phi_1^c &=-(F'_A(\omega_B)-1)\Phi_2^c+\mathcal{S}\Lambda_A+\mathcal{T}_A\Lambda_A^2+O((\Phi_2^c)^2)+O(\Phi_2^c\Lambda_A)+O(\Lambda_A^3)\,,
\end{align*}
where  $\mathcal{T}_A$ is defined in  (\ref{17080110}). 
In light of the definition in (\ref{17021710}), we have
\begin{align}
\mathcal{Z}_1=\mathcal{S}\Lambda_A+\mathcal{T}_A\Lambda_A^2+O((\Phi_2^c)^2)+O(\Phi_2^c\Lambda_A)+ O(\Lambda_A^3)\,. \label{17021741}
\end{align}
Combination of (\ref{170724501}), (\ref{17021740}) with 
(\ref{17021741}) leads to 
\begin{align}
\big|\mathcal{S}\Lambda_A+\mathcal{T}_A\Lambda_A^2+O(\Lambda_A^3)\big|\prec \hat{\Pi}^2 +\Psi \hat{\Pi} \hat{\Lambda}\,. \label{1702175511}
\end{align}
The second term on the right hand side of (\ref{1702175511}) can be absorbed into the first term, in light of the fact that $\Psi\hat{\Lambda}\prec \hat{\Pi}$ (\cf (\ref{17072960})). Hence, we have 
\begin{align}
\big|\mathcal{S}\Lambda_A+\mathcal{T}_A\Lambda_A^2+O(\Lambda_A^3)\big|\prec \hat{\Pi}^2\,. \label{17021755}
\end{align}
Analogously, we also have
\begin{align}
\big|\mathcal{S}\Lambda_B+\mathcal{T}_B\Lambda_B^2+O(\Lambda_B^3)\big|\prec \hat{\Pi}^2\,. \label{17021756}
\end{align}
This completes the proof of Proposition \ref{pro.17021715}.
\end{proof}

It remains to prove Proposition \ref{pro. 17021720}.
We state the proof for $\mathcal{Z}_1$, $\mathcal{Z}_2$ is handled similarly.
We set
\begin{align*}
\mathfrak{l}^{(k,l)}\equiv\mathfrak{l}^{(k,l)}(z)\deq \mathcal{Z}_1^k\overline{\mathcal{Z}_1^l}\,,\qquad\qquad k,l\in \N\,.
\end{align*}  
We can now prove a stronger estimate one $\mathbb{E}[\mathfrak{l}^{(p,p)}]$ than the estimate obtained from Lemma \ref{lem.17021231} by improving the error terms from $O_\prec(\Psi\hat{\Pi})$ to $O_\prec(\hat{\Pi}^2)$.
\begin{lem} \label{lem.17022410}
Fix a $z\in \mathcal{D}_\tau(\eta_{\rm m},\eta_\mathrm{M})$. Suppose that the assumptions  of Proposition \ref{pro. 17021720} hold. For any fixed integer $p\geq 1$, we have
\begin{align*}
\mathbb{E}\big[ \mathfrak{l}^{(p,p)} (z)\big]=\mathbb{E}\big[O_\prec(\hat{\Pi}^2)\mathfrak{l}^{(p-1,p)} (z)\big]+\mathbb{E}\big[O_\prec(\hat{\Pi}^4) \mathfrak{l}^{(p-2,p)} (z)\big]+\mathbb{E}\big[O_\prec(\hat{\Pi}^4) \mathfrak{l}^{(p-1,p-1)}(z)\big]\,.
\end{align*}
\end{lem}
Now, with Lemma \ref{lem.17022410},  we can prove Proposition \ref{pro. 17021720}.
\begin{proof}[Proof of Proposition \ref{pro. 17021720}] Similarly to the proof of (\ref{17020301}) from Lemma \ref{lem.17021230}, with Lemma \ref{lem.17022410}, we can get (\ref{17021740}) by applying Young's and Markov's inequalities. This completes the proof of Proposition \ref{pro. 17021720}.
\end{proof}

In the sequel, we prove Lemma \ref{lem.17022410}.
\begin{proof} [Proof of Lemma \ref{lem.17022410}] Recall the definition of $\mathcal{Z}_1$ in (\ref{17012202}). We can write
\begin{align*}
\mathbb{E}\big[ \mathfrak{l}^{(p,p)}\big]= \frac{1}{N} \sum_{i=1}^N  \mathbb{E}\big[ \mathfrak{d}_{i,1}  Q_i \mathfrak{l}^{(p-1,p)}\big]+ \frac{1}{N} \sum_{i=1}^N   \mathbb{E}\big[\mathfrak{d}_{i,2}  \mathcal{Q}_i \mathfrak{l}^{(p-1,p)}\big].
\end{align*}
We only state the estimate for the first term on the right hand side above. The second term can be estimated in a similar way.  By (\ref{17021232}), we can write
\begin{align*}
\frac{1}{N} \sum_{i=1}^N   \mathfrak{d}_{i,1} Q_i= \frac{1}{N}\sum_{i=1}^N   (\wt{B}G)_{ii} \ntr G\tau_{i1},
\end{align*}
where we chose $d_i=\mathfrak{d}_{i,1}, i\in\llbracket 1,N\rrbracket$, in the definition of $\tau_{i1}$ in (\ref{17021305}). 

  Then, analogously to (\ref{17021240}), we can also write 
\begin{align}
&\frac{1}{N} \sum_{i=1}^N  \mathbb{E}\big[ \mathfrak{d}_{i,1}  Q_i \mathfrak{l}^{(p-1,p)}\big] =\frac{1}{N}\sum_{i=1}^N  \mathbb{E}\Big[ (\wt{B}G)_{ii} \ntr G \tau_{i1} \mathfrak{l}^{(p-1,p)}\Big]
\end{align}
with  $d_i=\mathfrak{d}_{i,1}, i\in\llbracket 1,N\rrbracket$.  Analogously to (\ref{17071833}), we can show 
\begin{align*}
&\frac{1}{N} \sum_{i=1}^N  \mathbb{E}\big[ \mathfrak{d}_{i,1}  Q_i \mathfrak{l}^{(p-1,p)}\big]=\mathbb{E}\big[O_\prec(\hat{\Pi}^2) \mathfrak{l}^{(p-1,p)}\big]+\mathbb{E}\big[O_\prec(\Psi^2\hat{\Pi}^2) \mathfrak{l}^{(p-2,p)}\big]+\mathbb{E}\big[O_\prec(\Psi^2\hat{\Pi}^2) \mathfrak{l}^{(p-1,p-1)}\big],
\end{align*}
where the last two terms come from the estimates of the analogues of the last two terms of (\ref{17021250}),  the third and fourth  terms in the right side of (\ref{17021252}), and also the terms in (\ref{17022401}) and (\ref{17022402}),  but with $\frac{1}{N}\sum_{j=1}^N   d_jQ_j$ replaced by $\mathcal{Z}_1$. It suffices to improve the estimates of these terms. All these terms contain a derivative $\frac{\partial \mathcal{Z}_1}{\partial g_{ik}}$ or  $\frac{\partial \overline{\mathcal{Z}}_1}{\partial g_{ik}}$, which is smaller than the derivative of an arbitrary linear combination $\partial (\frac{1}{N}\sum_i d_i Q_i)/\partial g_{ik}$ or $\partial (\frac{1}{N}\sum_i d_i \mathcal{Q}_i)/\partial g_{ik}$, due to the special choice of $\mathfrak{d}_{i,1}$'s and $\mathfrak{d}_{i,2}$'s. 
Specifically, we shall show the following lemma, which contains the estimates of all necessary terms. 
\begin{lem} \label{lem. partial Z} Fix a $z\in \mathcal{D}_\tau(\eta_{\rm m},\eta_\mathrm{M})$. Suppose that the assumptions  of Proposition \ref{pro.17020310} hold.  Let  $\tilde{d}_1, \ldots, \tilde{d}_N \in \mathbb{C}$ be (possibly random) numbers with $\max_i|\tilde{d}_i|\prec 1$. Let  $X_i=I$ or $\wt{B}^{\la i\ra}$. Then we~have
\begin{align}
&\frac{1}{N^2}\sum_{i=1}^N \sum_k^{(i)} \tilde{d}_i \mathbf{e}_k^* X_i G\mathbf{e}_i  \frac{\partial \mathcal{Z}_1}{\partial g_{ik}}=O_\prec(\hat{\Pi}^4)\,,\qquad & & \frac{1}{N^2}\sum_{i=1}^N \sum_k^{(i)} \tilde{d}_i \mathbf{e}_k^* X_i G\mathbf{e}_i  \frac{\partial \overline{\mathcal{Z}_1}}{\partial g_{ik}}=O_\prec(\hat{\Pi}^4)\,,\nonumber\\
&\frac{1}{N^2}\sum_{i=1}^N \sum_k^{(i)} \tilde{d}_i \mathbf{e}_k^* X_i \mathring{\mathbf{g}}_i  \frac{\partial  \mathcal{Z}_1}{\partial g_{ik}}=O_\prec(\hat{\Pi}^4)\,, \qquad & &  \frac{1}{N^2}\sum_{i=1}^N \sum_k^{(i)} \tilde{d}_i \mathbf{e}_k^* X_i \mathring{\mathbf{g}}_i \frac{\partial \overline{\mathcal{Z}_1}}{\partial g_{ik}}=O_\prec(\hat{\Pi}^4)\,.\label{17022420}
\end{align}
\end{lem}

\begin{proof}[Proof of Lemma \ref{lem. partial Z}]
We give the proof for the first estimate in (\ref{17022420}). The third one is analogous, and the other two are just their complex conjugates. From the definitions in (\ref{072960}) and (\ref{17021710}), we~get
\begin{align*}
\frac{\partial \mathcal{Z}_1}{\partial g_{ik}}&=\frac{\partial \Phi_1^c}{\partial g_{ik}}+(F_A'(\omega_B)-1)\frac{\partial \Phi_2^c}{\partial g_{ik}}\nonumber\\
&=\Big(\big(F_A'(\omega_B)-1\big)\big(F_B'(\omega_A^c)-1\big)-1\Big)\frac{\partial \omega_A^c}{\partial g_{ik}}+\big(F_A'(\omega_B^c)-F_A'(\omega_B)\big)\frac{\partial \omega_B^c}{\partial g_{ik}}.
\end{align*}
Note that by the regularity of $F_A$ and $F_B$, we have 
\begin{align*}
\big(F_A'(\omega_B)-1\big)\big(F_B'(\omega_A^c)-1\big)-1=\mathcal{S}+O(|\Lambda_A|)\,,\qquad F_A'(\omega_B^c)-F_A'(\omega_B)=O(|\Lambda_B|)\,. 
\end{align*}
The smallness of these coefficients carry the gain.  
According to the definition of $\hat{\Pi}$ in (\ref{17072960}),  we see that 
\begin{align*}
(|\mathcal{S}|+\Lambda)\Psi^2\Pi^2 \leq  \hat{\Pi}^4 
\end{align*} if $\Lambda\leq \hat{\Lambda}$. Hence,  for the first estimate in (\ref{17022420}), 
 it suffices to show that
\begin{align}
\frac{1}{N^2}\sum_{i=1}^N \sum_k^{(i)} \tilde{d}_i \mathbf{e}_k^* X_i G\mathbf{e}_i  \frac{\partial \omega_\iota^c}{\partial g_{ik}}=O_\prec(\Psi^2\Pi^2)\,, \qquad \iota=A,B\,.  \label{17073001}
\end{align}
This follows from (\ref{17021303}), the fact that $\omega_B^c$ is a tracial quantity, and the chain rule. The other terms in (\ref{17022420}) can be estimated similarly. 
This concludes the proof of Lemma \ref{lem. partial Z}. 
\end{proof}
With the aid of Lemma \ref{lem. partial Z}, we can conclude the proof of Lemma \ref{lem.17022410}. 
\end{proof}

\section{Weak local law} \label{s.weak law}

In Sections \ref{s. Entrywise estimate}, \ref{s. rough bound} and \ref{s.optimal FL}, we established the subordination property for the Green function entries and the rough and optimal fluctuation averaging for the linear combinations of them, but all for a fixed $z\in \mathcal{D}_\tau(\eta_{\rm m},\eta_\mathrm{M})$, under the a priori input (\ref{17020501}). In this section, based on some cutoff versions of the conclusions in Sections \ref{s. Entrywise estimate} and \ref{s. rough bound} (\cf (\ref{19030501111}), (\ref{19030501})), we will establish a weak local law,  uniformly in $z\in \mathcal{D}_\tau(\eta_{\rm m},\eta_\mathrm{M})$,  by using a continuity argument. The weak local law will guarantee that the input in (\ref{17020501}) hold uniformly true on $\mathcal{D}_\tau(\eta_{\rm m},\eta_\mathrm{M})$, and thus the conclusions in Sections \ref{s. Entrywise estimate}, \ref{s. rough bound} and \ref{s.optimal FL} are also uniformly true on $\mathcal{D}_\tau(\eta_{\rm m},\eta_\mathrm{M})$. 

Our main result in this section is the following weak local law for the quantities $P_i$, $K_i$, $T_i$, $\Lambda_{{\rm d} i}^c(z)$, $\Upsilon$, $\Lambda_{\rm d}$, $\Lambda$, defined in (\ref{17011301}), (\ref{17072420}), (\ref{17072580}), (\ref{17080305}), (\ref{17011302}), (\ref{17072571}), (\ref{le gros lambda}), respectively. 

\begin{thm}[Weak local law at the regular edge] \label{thm. weak law at the edge} Suppose that Assumptions \ref{a.regularity of the measures} and \ref{a. levy distance} hold. Let $\tau>0$ be a sufficiently small constant and fix any (small) constants $\gamma>0$. 
 Then, for all $i\in \llbracket 1, N\rrbracket$, we have
\begin{align}
| P_i(z)|\prec \Psi(z)\,, \quad | K_i(z)|\prec \Psi(z)\,, \quad \Lambda_{{\rm d} i}^c(z)\prec \Psi(z)\,,\quad |T_{i}|\prec \Psi(z)\,, \quad  |\Upsilon(z) |\prec \Psi(z)\,,  \label{19031901}
\end{align}
uniformly in $z\in \mathcal{D}_\tau(\eta_{\rm m},\eta_\mathrm{M})$. In addition, we have 
\begin{align}
\Lambda_{{\rm d}}(z)\prec \frac{1}{(N\eta)^{\frac13}}\,, \qquad \Lambda(z)\prec \frac{1}{(N\eta)^{\frac13}}\,, \label{19031902}
\end{align}
uniformly in $z\in \mathcal{D}_\tau(\eta_{\rm m},\eta_\mathrm{M})$.

The same statements hold
 for the analogous quantities with tildes (see their definitions around \eqref{tildedef}), i.e.
 if we switch the roles of $A$ and $B$, and also the roles of $U$ and $U^*$.
 
 \end{thm}

In order to prove Theorem \ref{thm. weak law at the edge}, we first need the following lemma. Recall from~\eqref{le gros lambda} the definitions $\Lambda_\iota(z)=\omega_\iota(z)-\omega_\iota^c(z)$, $\iota=A,B$. Further recall the definitions of $\mathcal{T}_\iota$,  $\iota=A,B$, and $\mathcal{S}_{AB}$ from~\eqref{17080110} and that we abbreviate $\mathcal{S}=\mathcal{S}_{AB}$.

\begin{lem}  \label{lem.19031101}
Fix $z\in \mathcal{D}_\tau(\eta_{\rm m},\eta_\mathrm{M})$.   Let $\varepsilon\in (0, \frac{\gamma}{100})$. Let $\hat{\Lambda}\equiv\hat{\Lambda}(z)$ be some deterministic control parameter satisfying $\frac{N^{3\varepsilon}}{(N\eta)^{\frac13}}\leq\hat\Lambda(z)\leq  N^{-\frac{\gamma}{4}}$. Suppose that $\Lambda\leq \hat{\Lambda}$ and 
\begin{align}
&\Big|\mathcal{S}\Lambda_\iota+\mathcal{T}_\iota\Lambda_\iota^2+O(\Lambda_\iota^3)\Big|\leq   N^\varepsilon  \frac{|\mathcal{S}|+ \hat{\Lambda}}{(N\eta)^{\frac13}},   \qquad \iota=A, B.  \label{19030801}
\end{align}
hold on some event $\wt{\Omega}(z)$. Then we have, for $N$ sufficiently large,

\noindent $(i)$: If $
\sqrt{\kappa+\eta}> N^{-\varepsilon}  \hat{\Lambda} 
$, there is a sufficiently large constant $K_0>0$  independent of $z$, such that 
\begin{align}
\mathbf{1}\Big(\Lambda\leq \frac{|\mathcal{S}|}{K_0}\Big) |\Lambda_A| \leq  N^{-2\varepsilon} \hat{\Lambda}\,,\qquad\qquad\mathbf{1}\Big(\Lambda\leq \frac{|\mathcal{S}|}{K_0}\Big) |\Lambda_B| \leq  N^{-2\varepsilon} \hat{\Lambda}\, \qquad \text{on}\quad \wt{\Omega}(z)\,, \label{19031102}
\end{align}
where $\mathbf{1}$ denotes the indicator function.

\noindent $(ii)$: If $
\sqrt{\kappa+\eta}\leq  N^{-\varepsilon}  \hat{\Lambda} 
$, we have
\begin{align}
|\Lambda_A|\leq N^{-\varepsilon}  \hat{\Lambda}\,, \qquad\qquad|\Lambda_B|\leq N^{-\varepsilon}  \hat{\Lambda}\, \qquad \text{on}\quad \wt{\Omega}(z).  \label{19031103}
\end{align}
\end{lem}
 \begin{proof}[Proof of Lemma \ref{lem.19031101}] We first recall  (\ref{17080121}).  Then, from the assumptions $|\Lambda_\iota|\leq\hat{\Lambda}\leq N^{-\frac{\gamma}{4}}$ and (\ref{19030801}), we have on the event $\wt{\Omega}(z)$ that
\begin{align}
\mathcal{S}\Lambda_\iota+\mathcal{T}_\iota\Lambda_\iota^2=O  \Big( N^{\varepsilon}\frac{|\mathcal{S}|+ \hat{\Lambda}}{(N\eta)^{\frac13}}+N^{-\frac{\gamma}{4}} \hat{\Lambda}^2\Big),\qquad \iota=A, B. \label{17030302}
\end{align}  
If $
\sqrt{\kappa+\eta}> N^{-\varepsilon} \hat{\Lambda} 
$, we have for $\iota=A, B$,  and sufficiently large constant $K_0>0$,
\begin{align}
\mathbf{1}\Big(\Lambda\leq \frac{|\mathcal{S}|}{K_0}\Big) |\Lambda_\iota|\leq C |\mathcal{S}|^{-1}  \Big(N^{\varepsilon} \frac{|\mathcal{S}|+ \hat{\Lambda}}{(N\eta)^{\frac13}}+N^{-\frac{\gamma}{4}} \hat{\Lambda}^2\Big)\leq C \frac{N^\varepsilon}{(N\eta)^{\frac13}}+  C N^{\varepsilon-\frac{\gamma}{4}} \hat{\Lambda}\leq C N^{-2\varepsilon} \hat{\Lambda}\,.  \label{17030303}
\end{align}
Here we absorbed the quadratic term on the left side of (\ref{17030302}) into the linear term and used that $\mathcal{S}\sim\sqrt{\kappa+\eta}$ and $|\mathcal{T}_\iota|\lesssim1$; see Proposition~\ref{le proposition 3.1}. Hence, we proved $(i)$.
From (\ref{17030303}), we also see that if $ \sqrt{\kappa+\eta}> N^{-\varepsilon}  \hat{\Lambda} $, then
\begin{align}
\mathbf{1}\Big(\Lambda\leq \frac{|\mathcal{S}|}{K_0}\Big) |\Lambda_\iota| \leq C  N^{-\varepsilon} |\mathcal{S}|\,,\qquad \qquad \iota=A, B. \label{17071901}
\end{align}

Next, we prove $(ii)$.  If $ \sqrt{\kappa+\eta}\leq  N^{-\varepsilon}  \hat{\Lambda} $, from (\ref{17080121}) and (\ref{17080122}), we see that $\mathcal{T}_\iota\sim 1$. Hence, we solve the quadratic equation (\ref{17030302}) directly, then we get 
\begin{align*}
|\Lambda_\iota|\leq   C|\mathcal{S}|+ C\Big(N^{\varepsilon}\frac{|\mathcal{S}|+ \hat{\Lambda}}{(N\eta)^{\frac13}}+N^{-\frac{\gamma}{4}} \hat{\Lambda}^2\Big)^{\frac{1}{2}}\leq CN^{-\varepsilon}  \hat{\Lambda}\,, \qquad \iota=A, B \,,
\end{align*}
under the assumption that $ \hat{\Lambda}\geq \frac{N^{3\varepsilon}}{(N\eta)^{\frac13}}$. This concludes the proof of Lemma \ref{lem.19031101}.
\end{proof}

Recall the definitions  of $\Lambda_{\rm d}$, $\wt{\Lambda}_{\rm d}$, ${\Lambda}_T$, $ \wt{\Lambda}_T$ in~\eqref{17072571}. For any  $z\in \mathcal{D}_\tau(\eta_{\rm m},\eta_\mathrm{M})$ and any $\delta, \delta'\in [0,1]$, we  define the event 
\begin{align}
\Theta(z, \delta, \delta')\deq \Big\{ \Lambda_{\rm d}(z)\leq \delta, \; \wt{\Lambda}_{\rm d}(z) \leq \delta,\;
 \Lambda(z)\leq \delta,  \; \Lambda_T(z)\leq \delta',\; \wt{\Lambda}_T(z)\leq \delta'\Big\}.  \label{19031130}
\end{align}  
We further decompose the domain $\mathcal{D}_\tau(\eta_{\rm m},\eta_\mathrm{M})$ into the following two disjoint parts:
\begin{align}
&\mathcal{D}_{>}\deq \Big\{z\in \mathcal{D}_\tau(\eta_{\rm m},\eta_\mathrm{M}): \sqrt{\kappa+\eta}> \frac{N^{2\varepsilon}}{(N\eta)^{\frac13}} \Big\},\quad\mathcal{D}_{\leq}\deq\Big\{z\in \mathcal{D}_\tau(\eta_{\rm m},\eta_\mathrm{M}): \sqrt{\kappa+\eta}\leq \frac{N^{2\varepsilon}}{(N\eta)^{\frac13}} \Big\}\,.  \label{190312100}
\end{align}
For $z\in \mathcal{D}_{>}$, any $\delta, \delta'\in [0,1]$ and any $\varepsilon'\in [0,1]$, we define the event $\Theta_>(z, \delta, \delta', \varepsilon')\subset \Theta(z, \delta, \delta')$ as
\begin{align}
\Theta_>(z, \delta, \delta', \varepsilon')\deq \Big\{ \Lambda_{\rm d}(z)\leq \delta, \; \wt{\Lambda}_{\rm d}(z) \leq \delta,\;
 \Lambda(z)\leq \min\{\delta, N^{-\varepsilon'} |\mathcal{S}| \},  \; \Lambda_T(z)\leq \delta',\; \wt{\Lambda}_T(z)\leq \delta'\Big\}\,. 
\end{align}

\begin{lem} \label{lem.19030901} Suppose that the assumptions in Theorem \ref{thm. strong law at the edge} hold.  For any fixed $z\in \mathcal{D}_\tau(\eta_{\rm m},\eta_\mathrm{M})$,  any  $\varepsilon\in (0, \frac{\gamma}{100})$ and  any $D>0$, there exists a positive integer $N_1(D, \varepsilon)$ and an event $\Omega(z)\equiv \Omega(z, D,\varepsilon)$ with 
\begin{align}
\mathbb{P}(\Omega(z))\geq 1- N^{-D}\,,\qquad \forall N\geq N_1(D, \varepsilon) \,, \label{19031220}
\end{align} 
such that the following hold:  

(i) If $z\in \mathcal{D}_{>}$, we  have 
\begin{align}
\Theta_{>} \Big(z, {\frac{N^{3\varepsilon}}{(N\eta)^{\frac13}}}, {\frac{N^{3\varepsilon}}{\sqrt{N\eta}}}, \frac{\varepsilon}{10}\Big) \cap \Omega(z) \subset \Theta_{>} \Big(z, {\frac{N^{\frac52\varepsilon}}{(N\eta)^{\frac13}}}, {\frac{N^{\frac52\varepsilon}}{\sqrt{N\eta}}},  \frac{\varepsilon}{2}\Big).  \label{19031201}
\end{align}

(ii) If $z\in \mathcal{D}_{\leq}$, we have
\begin{align}
\Theta \Big(z, {\frac{N^{3\varepsilon}}{(N\eta)^{\frac13}}}, {\frac{N^{3\varepsilon}}{\sqrt{N\eta}}}\Big) \cap \Omega(z) \subset \Theta \Big(z, {\frac{N^{\frac52\varepsilon}}{(N\eta)^{\frac13}}}, {\frac{N^{\frac52\varepsilon}}{\sqrt{N\eta}}}\Big)\,. \label{19031202}
\end{align}
\end{lem}

\begin{proof}[Proof of Lemma \ref{lem.19030901}]   
 The proof relies on a cutoff version of Lemma \ref{lem.17021230} and  Lemma \ref{lem.17021231}, where we will introduce some smooth cutoff to $\mathfrak{m}_i$ and $\mathfrak{m}$ to guarantee that the a priori inputs needed for the estimates hold. 
The same idea has already been used in Lemma 5.5 of \cite{BES16b}, but for completeness
we repeat the arguments.
 With these cutoff versions of Lemmas \ref{lem.17021230} and  \ref{lem.17021231}, the proof of Lemma \ref{lem.19030901} is accomplished in three steps, corresponding to what we did in Sections \ref{s. Entrywise estimate}, \ref{s. rough bound} and \ref{s.optimal FL}, respectively.

 \vspace{2ex}
\noindent {\bf Step 1:}
 In this step we  establish the cutoff version of Lemma \ref{lem.17021230} and use it to prove an estimate of  for $G_{ii}$'s, $T_{i}$'s and their tilde analogues.  Let $\varphi:\mathbb{R}\to \mathbb{R}$ be a smooth cutoff function equal to $1$ on $[-\mathcal{L},\mathcal{L}]$ and vanishing on $[-2\mathcal{L},2\mathcal{L}]^c$, such that  $\sup_{x\in \mathbb{R}}|\varphi'(x)|\leq C\mathcal{L}^{-1}$
for some sufficiently large constant $\mathcal{L}>0$. Let 
\begin{align} 
\Gamma_i\equiv \Gamma_i(z)\deq |G_{ii}|^2+|\mathcal{G}_{ii}|^2+|T_{i}|^2+ |\wt{T}_i|^2+|\ntr G|^2+|\ntr \wt{B}G|^2+|\ntr \wt{B}G\wt{B}|^2\,,  \label{19030401}
\end{align}
where we denote by $\wt{T}_i$ the analogue of $T_i$, obtained via switching the roles of $A$ and $B$ and also the roles of $U$ and $U^*$ in the definition of $T_i$ (\cf (\ref{17072580})). 
For a given $i$, observe that all the a priori inputs we needed in the proof of Lemma \ref{lem.17021230} are the $O_\prec(1)$ bound for the summands on the right side of (\ref{19030401}). Recall the definitions of $\mathfrak{m}_i^{(k,\ell)}$ and $\mathfrak{n}_i^{(k,\ell)}$ from (\ref{17072350}), and set 
\begin{align}
\wt{\mathfrak{m}}_i^{(k,\ell)}:= \mathfrak{m}_i^{(k,\ell)} (\varphi(\Gamma_i))^{k+\ell}, \qquad \wt{\mathfrak{n}}_i^{(k,\ell)}:=  \mathfrak{n}_i^{(k,\ell)} (\varphi(\Gamma_i))^{k+\ell} \label{19030905}
\end{align}
In addition, for any $\varepsilon_1>0$, let $\widehat{\Omega}_1(z)\equiv\widehat{\Omega}_1(z, \varepsilon_1)$ be the event that all concentration estimates of the components or quadratic forms of 
the Gaussian vectors $\mathbf{g}_i$'s in the proof of Lemma \ref{lem.17021230} hold with precision $N^{\varepsilon_1}$. For instance, we used the $O_\prec(\frac{1}{\sqrt{N}})$ bound for $\mathbf{h}_i^* \wt{B}^{\la i\ra}\mathbf{h}_i$ in  (\ref{17021540}). Now we can bound it more quantitatively by $\frac{N^{\varepsilon_1}}{\sqrt{N}}$ on $\widehat{\Omega}_1(z)$.  
 Due to the Gaussian tail, for any $D_1>0$, there exists an $N(D_1,\varepsilon_1)$, such that if $N\geq N(D_1, \varepsilon_1)$, then
\begin{align*}
\mathbb{P}(\widehat{\Omega}_1(z))\geq 1-N^{-D_1}. 
\end{align*}
Furthermore, we claim that 
\begin{align}
\mathbb{E}[\wt{\mathfrak{m}}_i^{(p,p)}]= \mathbb{E}[\mathfrak{c}_{i1} \wt{\mathfrak{m}}_i^{(p-1,p)}]+ \mathbb{E}[\mathfrak{c}_{i2} \wt{\mathfrak{m}}_i^{(p-2,p)}]+ \mathbb{E}[\mathfrak{c}_{i3} \wt{\mathfrak{m}}_i^{(p-1,p-1)}], \label{19030501111}
\end{align}
with some random variables $\mathfrak{c}_{i1}, \mathfrak{c}_{i2}$ and $\mathfrak{c}_{i3}$, satisfying 
\begin{align*}
|\mathfrak{c}_{i1}|\leq C\frac{N^{\varepsilon_1}}{\sqrt{N\eta}}, \qquad  |\mathfrak{c}_{i2}|\leq C\frac{N^{2\varepsilon_1}}{N\eta}, \qquad |\mathfrak{c}_{i3}|\leq C\frac{N^{2\varepsilon_1}}{N\eta}, \qquad \text{on}\quad \widehat{\Omega}_1(z),
\end{align*}
for some positive constant $C$ which may depend on $\mathcal{L}$, \ie  the parameter in the definition of the cutoff~$\varphi$. Moreover, the $\mathfrak{c}_{ia}$'s  also admit the moment bound  $\mathbb{E}|\mathfrak{c}_{ia}|^k=O(1)$ for any given $k>0$. Note that (\ref{19030501111}) is the same as (\ref{17021020}) but with a cutoff, since $\widehat{\Omega}_1(z)$ holds with high probability. The proof of (\ref{19030501111}) can be done in the same way as the proof of the non-cutoff one in (\ref{17021020}). Essentially, the only modification we need to accommodate is estimating the additional terms in the integration by parts that are created by introducing $\varphi(\Gamma_i)$ into $\wt{\mathfrak{m}}_i$. But it will be clear that these additional terms can be absorbed into the first term on the right side of (\ref{19030501111}). For instance, in the analogue of the step  (\ref{17020540}), apart from replacing $\mathfrak{m}_i$ by $\wt{\mathfrak{m}}_i$, we will have an additional term
\begin{align}
\frac{1}{N} \sum_k^{(i)} \mathbb{E}\Big[ \frac{\mathbf{e}_k^*\wt{B}^{\la i\ra} G\mathbf{e}_i}{\|\mathbf{g}_i\| }  \ntr G \frac{\partial  \varphi(\Gamma_i)}{\partial g_{ik}}\wt{\mathfrak{m}}_i^{(p-1,p)}\Big]. \label{19030502}
\end{align}
According to the definition of $\varphi(\Gamma_i)$, the derivative $\frac{\partial  \varphi(\Gamma_i)}{\partial g_{ik}}$ is written as a sum of several terms. For instance, one term is 
\begin{align*}
\varphi'(\Gamma_i)\frac{\partial |G_{ii}|^2}{\partial g_{ik}}= \varphi'(\Gamma_i)\overline{G_{ii}} \frac{\partial G_{ii}}{\partial g_{ik}}+   \varphi'(\Gamma_i) G_{ii}\frac{\partial \overline{G_{ii}}}{\partial g_{ik}}\,.
\end{align*}
Applying a quantitative version of  the second estimate in (\ref{17021202}), one obtains
\begin{align*}
\frac{1}{N}\sum_{k}^{(i)}  \frac{\mathbf{e}_k^*\wt{B}^{\la i\ra} G\mathbf{e}_i}{\|\mathbf{g}_i\| }  \ntr G  \frac{\partial |G_{ii}|^2 }{\partial g_{ik}}=O\big(\frac{N^{\varepsilon_1}}{\sqrt{N\eta}}\big)\,, \qquad \text{on }\; \{\varphi'(\Gamma_i)\neq 0\}\cap \widehat{\Omega}_1(z)\,.
\end{align*}
The contribution of the other terms from  $\frac{\partial  \varphi(\Gamma_i)}{\partial g_{ik}}$ to the term (\ref{19030502}) can be bounded in the same way. We omit the details. Hence, we have (\ref{19030501111}). 

Applying Young's inequality to (\ref{19030501111}), we have 
\begin{align*}
\mathbb{E} [\wt{\mathfrak{m}}_i^{(p,p)}]\leq C_p N^{2p\varepsilon_1}\Big(\mathbb{E} |\mathfrak{c}_1|^{2p}+ \mathbb{E}|\mathfrak{c}_2|^{p}+\mathbb{E} |\mathfrak{c}_3|^p \Big)\leq C_p N^{2p\varepsilon_1} \Big( \big(\frac{N^{\varepsilon_1}}{\sqrt{N\eta}}\big)^{2p}+N^{-\frac{D_1}{2}}\Big). 
\end{align*}
Further,  by Markov's inequality, we have 
\begin{align}
\mathbb{P} \Big(|P_i\varphi(\Gamma_i)|\geq \frac{N^{\frac{\varepsilon}{4}}}{\sqrt{N\eta}}\Big)\leq C_p \Big(\frac{N^{\frac{\varepsilon}{4}}}{\sqrt{N\eta}}\Big)^{-2p} N^{2p\varepsilon_1} \Big( \big(\frac{N^{\varepsilon_1}}{\sqrt{N\eta}}\big)^{2p}+N^{-\frac{D_1}{2}}\Big). \label{19030503}
\end{align}
For the given $\varepsilon>0$ in Lemma \ref{lem.19030901}, we can first choose $\varepsilon_1=\varepsilon_1(\varepsilon)$ to be smaller than $\frac{\varepsilon}{8}$, and then choose  $p=p(\varepsilon, D)$ to be sufficiently large, then we can get 
\begin{align}
C_p \Big(\frac{N^{\frac{\varepsilon}{4}}}{\sqrt{N\eta}}\Big)^{-2p} N^{2p\varepsilon_1}  \Big(\frac{N^{\varepsilon_1}}{\sqrt{N\eta}}\Big)^{2p}\leq \frac{1}{10} N^{-D}. \label{19030504}
\end{align}
Then by choosing $D_1= D_1(\varepsilon, D)$ sufficiently large, we also have 
\begin{align}
C_p \Big(\frac{N^{\frac{\varepsilon}{4}}}{\sqrt{N\eta}}\Big)^{-2p} N^{2p\varepsilon_1} N^{-\frac{D_1}{2}}\leq \frac{1}{10} N^{-D}.  \label{19030505}
\end{align}
We can thus denote $\wt{N}_1(D, \varepsilon):= N(D_1, \varepsilon_1)$, according to our choice of $\varepsilon_1, D_1$.  Then by (\ref{19030503})-(\ref{19030505}), there exists an event $\Omega_1(z)$, such that 
\begin{align*}
\mathbb{P}(\Omega_1(z))\geq 1-\frac{1}{5} N^{-D}, \qquad N\geq \wt{N}_1(D, \varepsilon) 
\end{align*}
and 
\begin{align}
|P_i\varphi(\Gamma_i)|\leq {\frac{N^\frac\varepsilon4}{\sqrt{N\eta}}}, \qquad \text{on  }\; \Omega_1(z).  \label{19030510}
\end{align}
Observing  that $\varphi(\Gamma_i)=1 $ on $\Theta(z,{\frac{N^{3\varepsilon}}{(N\eta)^{\frac13}}},  {\frac{N^{3\varepsilon}}{\sqrt{N\eta}}})$, we further get from (\ref{19030510}) that 
\begin{align}
|P_i|\leq {\frac{N^\frac\varepsilon4}{\sqrt{N\eta}}}, \qquad \text{on  }\; \Theta\Big(z,{\frac{N^{3\varepsilon}}{(N\eta)^{\frac13}}},  {\frac{N^{3\varepsilon}}{\sqrt{N\eta}}}\Big)\cap\Omega_1(z). \label{19030511}
\end{align}
By working with $\wt{\mathfrak{n}}_i$ instead of $\wt{\mathfrak{m}}_i$, we can also get 
\begin{align}
|K_i|\leq {\frac{N^\frac\varepsilon4}{\sqrt{N\eta}}}, \qquad \text{on  }\; \Theta \Big(z,{\frac{N^{3\varepsilon}}{(N\eta)^{\frac13}}},  {\frac{N^{3\varepsilon}}{\sqrt{N\eta}}}\Big)\cap\Omega_1(z) \,, \label{19030512}
\end{align}
 where we redefine $\Omega_1(z)$ to include the ``good" events for both estimates for $P_i$ and $K_i$.
Similarly to the proof of Proposition \ref{pro.17020310}, based on the estimates in (\ref{19030511}), (\ref{19030512}) and also their analogues by switching the roles of $A$ and $B$, and also the roles of $U$ and $U^*$, we can derive the  following  quantitative estimates:
\begin{align}
\Lambda_{{\rm d}}^c (z)\leq {\frac{N^\frac\varepsilon2}{\sqrt{N\eta}}}, \quad \wt{\Lambda}_{{\rm d}}^c(z)\leq  {\frac{N^\frac\varepsilon2}{\sqrt{N\eta}}}, \quad \Lambda_{T}(z)\leq  {\frac{N^\frac\varepsilon2}{\sqrt{N\eta}}}, \quad \wt{\Lambda}_T(z)\leq  {\frac{N^\frac\varepsilon2}{\sqrt{N\eta}}},\quad |\Upsilon(z)|\leq {\frac{N^\frac\varepsilon2}{\sqrt{N\eta}}}\,, \label{17030511}
\end{align}
hold on  $\Theta(z,{\frac{N^{3\varepsilon}}{(N\eta)^{\frac13}}},  {\frac{N^{3\varepsilon}}{\sqrt{N\eta}}})\cap \Omega_1(z)$,  where  $\Omega_1(z)$ shall be redefined further to include the ``good" events for both estimates of the analogues of $P_i$ and $K_i$. 

\vspace{2ex}
\noindent{\bf Step 2:}
In this step, we  derive a cutoff version of  Lemma \ref{lem.17021231}, but with weaker bounds, and use it to estimate the linear combinations of $G_{ii}$. Analogously to (\ref{19030905}), we introduce the variant of   (\ref{17071810})
\begin{align}
\wt{\mathfrak{m}}^{(k,l)}\deq\Big(\frac{1}{N}\sum_{i=1}^N   d_i  Q_i \varphi(\Gamma_i)\varphi(\Gamma)\Big)^k\Big(\frac{1}{N}\sum_{i=1}^N   \overline{d_i}\; \overline{ Q_i}\varphi(\Gamma_i)\varphi(\Gamma)\Big)^l \,,\qquad\qquad k,l\in\N\,,
\end{align}
where $\Gamma_i$ is defined in (\ref{19030401}) and $\Gamma$ is defined as the following 
\begin{align}
\Gamma:=(c\,\Im m_{\mu_A\boxplus\mu_B}+\hat{\Lambda})^{-2}\big(|\Lambda_A|^2+|\Lambda_B|^2\big)+\Big(\frac{N^{5\varepsilon}}{(N\eta)^{\frac13}}\Big)^{-2}|\Upsilon|^2+\Big(\frac{N^{5\varepsilon}}{\sqrt{N\eta}}\Big)^{-1} \frac{1}{N}\sum_{i=1}^N (| T_{i}|^2+N^{-1})^{\frac12}, \label{19030920}
\end{align}
for some sufficiently small constant $c>0$. In the rest of the proof, we choose 
\begin{align}
\hat{\Lambda}(z)=\frac{N^{3\varepsilon}}{(N\eta)^{\frac13}}. \label{190311200}
\end{align}
The boundedness of the first term in (\ref{19030920}) is used to control $\Pi$ by $\hat{\Pi}$, see  (\ref{17072960}) for its definition. The second and the third terms in (\ref{19030920}) are used to bound the analogue of the term $\frac{1}{N}\sum_{i}T_{i}\tau_{i1}\Upsilon$ in (\ref{19030922}).  

Similarly to $\widehat{\Omega}_1(z)$,  for any $\varepsilon_1>0$, let $\widehat{\Omega}_2(z)\equiv \widehat{\Omega}_2(z, \varepsilon_1)$ be the event that all concentration estimates of the components or quadratic forms of 
the Gaussian vectors $\mathbf{g}_i$'s in the proof of Lemma  \ref{lem.17021231} hold with precision $N^{\varepsilon_1}$.   Again, due to the Gaussian tails, for any $D_1>0$, there exists $\wt{N}(D_1,\varepsilon_1)$, such that if $N\geq \widetilde{N}(D_1, \varepsilon_1)$
\begin{align*}
\mathbb{P}(\widehat{\Omega}_2(z))\geq 1-N^{-D_1}. 
\end{align*}
Analogously to (\ref{19030501111}), we now claim that  
\begin{align}
\mathbb{E}[\wt{\mathfrak{m}}^{(p,p)}]= \mathbb{E}[\mathfrak{c}_{1} \wt{\mathfrak{m}}^{(p-1,p)}]+ \mathbb{E}[\mathfrak{c}_{2} \wt{\mathfrak{m}}^{(p-2,p)}]+ \mathbb{E}[\mathfrak{c}_{3} \wt{\mathfrak{m}}^{(p-1,p-1)}]\,, \label{19030501}
\end{align}
with some random variables $\mathfrak{c}_{1}, \mathfrak{c}_{2}$ and $\mathfrak{c}_{3}$, satisfying 
\begin{align}
|\mathfrak{c}_{1}|\leq C\hat{\Pi}\,, \qquad  |\mathfrak{c}_{2}|\leq C\hat{\Pi}^2\,, \qquad |\mathfrak{c}_{3}|\leq C\hat{\Pi}^2\,, \qquad \text{on}\quad \widehat{\Omega}_2(z)\,, \label{19030950}
\end{align}
for some positive constant $C$. Moreover, the $\mathfrak{c}_{i}$'s  also admit the moment bound  $\mathbb{E}|\mathfrak{c}_{i}|^k=O(1)$, for any given $k>0$. 
Note that (\ref{19030501}) is similar to (\ref{17071833}), but with a weaker bounds for $\mathfrak{c}_i$'s. The weakness of the bounds is partially due to the weak a priori input in the cutoffs $\varphi(\Gamma_i)$ and $\varphi(\Gamma)$, and also partially due to the additional terms  involving the derivatives of the cutoffs which are generated by the integration by parts.  In Appendix \ref{appendix C}, we show more details on how to slightly modify the proof of  (\ref{17071833}) to get (\ref{19030501}).   Similarly to (\ref{19030510}), we can show from  (\ref{19030501}) that  there exists an event $\Omega_2(z)$, such that 
\begin{align*}
\mathbb{P}(\Omega_2(z))\geq 1-\frac{1}{5} N^{-D}\,, \qquad N\geq \wt{N}_2(D, \varepsilon) \,,
\end{align*}
for some sufficiently large constant $\wt{N}_2(D, \varepsilon)>0$, and 
\begin{align*}
\Big|\frac{1}{N}\sum_{i=1}^N   d_i  Q_i \varphi(\Gamma_i)\varphi(\Gamma)\Big|\leq N^{\frac{\varepsilon}{4}}\hat{\Pi}, \qquad \text{on}\quad \Omega_2(z)\,.
\end{align*}
Note that $\varphi(\Gamma_i)=\varphi(\Gamma)=1$ for all $i$ on $\Theta(z,{\frac{N^{3\varepsilon}}{(N\eta)^{\frac13}}},  {\frac{N^{3\varepsilon}}{\sqrt{N\eta}}})$. Hence, we have 
\begin{align*}
\Big|\frac{1}{N}\sum_{i=1}^N   d_i  Q_i \Big|\leq  N^{\frac{\varepsilon}{3}}  \hat{\Pi}\,, \qquad \text{on}\quad \Theta\Big(z,{\frac{N^{3\varepsilon}}{(N\eta)^{\frac13}}}\,,  {\frac{N^{3\varepsilon}}{\sqrt{N\eta}}} \Big)\cap \Omega_2(z)\,.
\end{align*}

 \vspace{2ex}
\noindent{\bf Step 3:}  In the last step, we perform an estimate of $\Lambda_A$ and $\Lambda_B$, by 
choosing  $d_i=\mathfrak{d}_{i1}$ in (\ref{17022001}) and also considering the analogues of $\frac{1}{N}\sum_{i=1}^N   \mathfrak{d}_{i1}  Q_i $. Repeating the argument of the proof of Proposition \ref{pro. 17021720} but with the cutoff versions, where the error terms due to the cutoff are estimated analogously as in Appendix  \ref{appendix C},  we can then finally show that there exists an event $\Omega_3(z)$, such that 
\begin{align*}
\mathbb{P}(\Omega_3(z))\geq 1-\frac{4}{5} N^{-D}, \qquad N\geq \wt{N}_3(D, \varepsilon) \,,
\end{align*}
for some sufficiently large constant $\wt{N}_3(D, \varepsilon)>0$, and 
\begin{align*}
\big|\mathcal{Z}_\iota\big|\leq   N^{\frac{\varepsilon}{3}}  \hat{\Pi} \leq  N^{\frac{\varepsilon}{2}}  \frac{|\mathcal{S}|+ \hat{\Lambda}}{(N\eta)^{\frac13}}, \quad \iota=1,2 \qquad \text{on}\quad \Theta\Big(z,{\frac{N^{3\varepsilon}}{(N\eta)^{\frac13}}},  {\frac{N^{3\varepsilon}}{\sqrt{N\eta}}}
\Big)\cap\Omega_3(z),
\end{align*}
where $\mathcal{Z}_\iota$'s are defined in (\ref{17021710}). Here in the second step above we used the fact $\Im m_{\mu_A\boxplus\mu_B}\lesssim |\mathcal{S}|$ (\cf (\ref{17080120}), (\ref{17080121})),  and  thus
\begin{align*}
\hat{\Pi}\leq  C\Big(\frac{|\mathcal{S}|+\hat{\Lambda}}{N\eta}\Big)^{\frac12} \leq C \Big(\frac{|\mathcal{S}|+\hat{\Lambda}}{(N\eta)^{\frac13}}+\frac{1}{(N\eta)^{\frac23}} \Big)\leq C' \frac{|\mathcal{S}|+\hat{\Lambda}}{(N\eta)^{\frac13}} 
\end{align*}
under the choice of $\hat{\Lambda}$ in (\ref{190311200}). Similarly to the proof of (\ref{17030301}), we can then get the following weaker but quantitative version of (\ref{17030301})
\begin{align*}
\Big|\mathcal{S}\Lambda_\iota+\mathcal{T}_\iota\Lambda_\iota^2+O(\Lambda_\iota^3)\Big|\leq   N^{{\varepsilon}} \frac{|\mathcal{S}|+ \hat{\Lambda}}{(N\eta)^{\frac13}},\quad \iota=1,2 \qquad \text{on}\quad \Theta(z,{\frac{N^{3\varepsilon}}{(N\eta)^{\frac13}}}\,,  {\frac{N^{3\varepsilon}}{\sqrt{N\eta}}})\cap\Omega_3(z)\,.
\end{align*}
Then, applying Lemma \ref{lem.19031101}, we have (\ref{19031102}) and (\ref{19031103}) with the choice $\wt{\Omega}(z)= \Theta(z,{\frac{N^{3\varepsilon}}{(N\eta)^{\frac13}}},  {\frac{N^{3\varepsilon}}{\sqrt{N\eta}}})\cap\Omega_3(z)$. 
In addition, from the conclusion of (\ref{19031102}) and recalling  that   $|\mathcal{S}|\sim \sqrt{\kappa+\eta}$, we note that if $\sqrt{\kappa+\eta}> N^{-\varepsilon}  \hat{\Lambda} $, then
\begin{align}
\mathbf{1}\Big(\Lambda\leq \frac{|\mathcal{S}|}{K_0}\Big) |\Lambda_\iota| \leq N^{-\varepsilon} |\mathcal{S}|,\qquad \qquad \iota=A, B \,,\label{17071901}
\end{align}
with $K_0$ chosen in Lemma \ref{lem.19031101}. 

Therefore, with the choice in (\ref{190311200}), by (\ref{19031102}), (\ref{19031103}) and (\ref{17071901}), we have  
\begin{align}
\Lambda \leq \min \Big\{\frac{N^{\frac52\varepsilon}}{(N\eta)^{\frac13}}, N^{-\varepsilon} |\mathcal{S}|\Big\}, \qquad \text{on}\quad  \Theta_{>} \Big(z, {\frac{N^{3\varepsilon}}{(N\eta)^{\frac13}}}, {\frac{N^{3\varepsilon}}{\sqrt{N\eta}}}, \frac{\varepsilon}{10}\Big) \cap \Omega_3(z) \label{19031110}
\end{align}
if $ \sqrt{\kappa+\eta}> N^{-\varepsilon}  \hat{\Lambda} $, respectively,
\begin{align}
\Lambda \leq \frac{N^{\frac52\varepsilon}}{(N\eta)^{\frac13}}, \qquad \text{on}\quad  \Theta \Big(z, {\frac{N^{3\varepsilon}}{(N\eta)^{\frac13}}}, {\frac{N^{3\varepsilon}}{\sqrt{N\eta}}}\Big) \cap \Omega_3(z) \label{19031111}
\end{align}
if $ \sqrt{\kappa+\eta}\leq N^{-\varepsilon}  \hat{\Lambda} $. Further, applying (\ref{19031110}) and (\ref{19031111}) to (\ref{17030511}), we can also conclude that 
\begin{align}
\Lambda_{{\rm d}} (z)\leq {\frac{N^{\frac52\varepsilon}}{(N\eta)^{\frac13}}}, \qquad \wt{\Lambda}_{{\rm d}}(z)\leq  {\frac{N^{\frac52\varepsilon}}{(N\eta)^{\frac13}}}, \qquad \Lambda_{T}(z)\leq  {\frac{N^\frac\varepsilon2}{\sqrt{N\eta}}}, \qquad \wt{\Lambda}_T(z)\leq  {\frac{N^\frac\varepsilon2}{\sqrt{N\eta}}} \label{19031120}
\end{align}
on $\Theta_{>} \Big(z, {\frac{N^{3\varepsilon}}{(N\eta)^{\frac13}}}, {\frac{N^{3\varepsilon}}{\sqrt{N\eta}}}, \frac{\varepsilon}{10}\Big) \cap \Omega(z)$  if 
$\sqrt{\kappa+\eta}> N^{-\varepsilon}  \hat{\Lambda} $, and on ${ \Theta} \Big(z, {\frac{N^{3\varepsilon}}{(N\eta)^{\frac13}}}, {\frac{N^{3\varepsilon}}{\sqrt{N\eta}}}\Big) \cap \Omega(z)$ if $ \sqrt{\kappa+\eta}\leq N^{-\varepsilon}  \hat{\Lambda} $, 
where 
\begin{align*}
\Omega(z):=\Omega_1(z)\cap \Omega_3(z). 
\end{align*}
Combining (\ref{19031110}), (\ref{19031111}) and (\ref{19031120}), we conclude the proof of Lemma \ref{lem.19030901}. 
\end{proof}

With Lemma \ref{lem.19030901}, we can now prove Theorem \ref{thm. weak law at the edge} by using a continuity argument.  

\begin{proof}[Proof of Theorem \ref{thm. weak law at the edge}]
We start with an entry-wise Green function subordination estimate on global scale, \ie $\eta=\eta_\mathrm{M}$ for some sufficiently large constant $\eta_\mathrm{M}>0$. Recall $Q_i$ from (\ref{17021701}). We regard $Q_i$ as a function of the random unitary matrix $U$. Then,  for $z=E+\mathrm{i}\wt{\eta}_M$ with any fixed $E$ and any  $\wt{\eta}_M\geq \eta_\mathrm{M}$, we apply the Gromov-Milman concentration inequality (\cf (6.2) in \cite{BES16b}), and get 
\begin{align}
|Q_i(E+\mathrm{i}\wt{\eta}_M)-\mathbb{E} Q_i(E+\mathrm{i} \wt{\eta}_M)|\prec\frac{1}{\sqrt{N \wt{\eta}_M^4}}\,; \label{170731101}
\end{align}
see Section 6.2 of \cite{BES16b} for similar estimates for the Green function entries of the block additive model. 

Next, using the invariance of the Haar measure, one can check the equation 
\begin{align}
\mathbb{E} (\wt{B}G\otimes G-G\otimes \wt{B}G)=0\,; \label{170731100}
\end{align}
see Proposition 3.2 of~\cite{VP}. 
Taking the $(i,i)$-th entry for the first component and the normalized trace for the second component in the tensor product, we obtain from (\ref{170731100}) that 
\begin{align}
\mathbb{E} Q_i= \mathbb{E} \big((\wt{B}G)_{ii}\ntr G-G_{ii}\ntr \wt{B}G\big)=0\,. \label{170731102}
\end{align}

We claim that, for sufficiently large $\eta_\mathrm{M}>1$, we have
\begin{align}
\sup_{z: \Im z\geq \eta_\mathrm{M}}|Q_i(z)|\prec\frac{1}{\sqrt{N}}\,,\qquad\qquad \forall i\in\llbracket 1,N\rrbracket\,,  \label{170731120}
\end{align}
where we used  (\ref{170731101}), (\ref{170731102}), the Lipschitz continuity of $Q_i$ in the regime $|z|\leq \sqrt{N}$ and the deterministic bound $|Q_i(z)|\leq\frac{C}{\sqrt{N}}$ when $|z|\geq \sqrt{N}$. 
In addition,  using that $\|H\|\leq \|A\|+\|B\|< \mathcal{K}$ and the convention $\ntr \wt{B}=\ntr B=0$ (\cf (\ref{17072620})),  we have, for $z=E+\mathrm{i}\wt{\eta}_M$ with fixed $E$ and any  $\wt{\eta}_M\geq \eta_\mathrm{M}$,  the expansions
\begin{align}
&\ntr G(z)=-\frac{1}{z}+O(\frac{1}{|z|^2})=\frac{\mathrm{i}}{\wt{\eta}_M}+O\big(\frac{1}{\wt{\eta}_M^2}\big), \qquad \quad\ntr \wt{B}G(z)=-\frac{\ntr \wt{B}}{z}+O(\frac{1}{|z|^2})=O(\frac{1}{\wt{\eta}_M^2})\,, \label{170731125}
\end{align}
where we used $\ntr B=0$ in the second equality.  Hence, by the definition of $\omega_B^c$ in (\ref{17072550}), we see  that, 
\begin{align}
\omega_B^c(z)= z+O(\frac{1}{\wt{\eta}_M}), \qquad\qquad z=E+\mathrm{i}\wt{\eta}_M. \label{170731130}
\end{align}
Using the identity $(\wt{B}G)_{ii}=1-(a_i-z)G_{ii}$, we can rewrite (\ref{170731120}) as 
\begin{align*}
(1-(a_i-\omega_B^c) G_{ii}) \ntr G=O_\prec(\frac{1}{\sqrt{N}}),\qquad \qquad z=E+\mathrm{i}\wt{\eta}_M. 
\end{align*}
 From the first line of (\ref{170731125}) and (\ref{170731130}) we get
\begin{align}
\Lambda_{{\rm d}}^c (z )\prec {\frac{1}{\sqrt N }}\,,\qquad \qquad z=E+\mathrm{i}\wt{\eta}_M\,.   \label{170731140}
\end{align}
Analogously, we also have 
\begin{align}
\wt{\Lambda}_{{\rm d}}^c(z )\prec\frac{1}{\sqrt{N}}\,,\qquad \qquad z=E+\mathrm{i}\wt{\eta}_M\,.   \label{170731141}
\end{align}
Averaging over the index $i$ in the definition of $\Lambda_{di}^c$ and $\wt{\Lambda}_{di}^c$ (\cf(\ref{17080305})), using (\ref{170731140}) and (\ref{170731141}) and using the fact $\ntr G=\ntr \mathcal{G}=m_H$ yields
\begin{align}
\sup_{z: \Im z\geq \eta_\mathrm{M}}\big|m_H(z)- m_A(\omega_B^c(z))\big|\prec \frac{1}{\sqrt{N}}\,,\qquad \qquad \sup_{z: \Im z\geq \eta_\mathrm{M}}\big|m_H(z)- m_B(\omega_A^c(z))\big|\prec \frac{1}{\sqrt{N}} \label{170731153}
\end{align}
where in the large $z$ regime these bounds  even hold deterministically, similarly to (\ref{170731120}). 
This together with (\ref{170725130}) gives us the system
\begin{align}
\sup_{z: \Im z\geq \eta_\mathrm{M}}|\Phi_1(\omega_A^c(z), \omega_B^c(z),z)|\prec \frac{1}{\sqrt{N}}\,,\qquad \sup_{z: \Im z\geq \eta_\mathrm{M}}|\Phi_2(\omega_A^c(z), \omega_B^c(z),z)|\prec \frac{1}{\sqrt{N}},   \label{170731150}
\end{align}
where $\Phi_1$ and $\Phi_2$ are defined in  (\ref{17073115}).  We regard (\ref{170731150}) as a perturbation of $\Phi_1(\omega_A(z), \omega_B(z),z)=0$, $\Phi_2(\omega_A(z), \omega_B(z),z)=0$. The stability of this system in the large $\eta$ regime is analyzed in Lemma~\ref{lem. stability for large eta}.
Choosing $(\mu_1, \mu_2)=(\mu_A, \mu_B)$, $(\wt{\omega}_1(z), \wt{\omega}_2(z))= (\omega_A^c(z), \omega_B^c(z))$ in Lemma \ref{lem. stability for large eta} below, and using  the fact that (\ref{170731150}) and (\ref{170731130}) hold for any  sufficiently large $\wt{\eta}_M$, we obtain from the stability Lemma \ref{lem. stability for large eta} that
\begin{align}
|\Lambda_\iota(z)|=|\omega_\iota^c(z)-\omega_\iota(z)|\prec \frac{1}{\sqrt{N}}\,, \qquad \qquad \iota=A,B\,, \qquad z=E+\mathrm{i}\eta_\mathrm{M}\,, \label{17080105}
\end{align}
for any sufficiently large constant $\eta_\mathrm{M}>1$, say.

Substituting (\ref{17080105}) into (\ref{170731140}) and (\ref{170731141}) gives
\begin{align}
\Lambda_{{\rm d}} (E+\mathrm{i}\eta_\mathrm{M} )\prec \frac{1}{\sqrt{N}}\,, \qquad \qquad \wt{\Lambda}_{{\rm d}}(E+\mathrm{i}\eta_\mathrm{M} )\prec \frac{1}{\sqrt{N}}\,,\label{17072101}
\end{align} 
for  any fixed $E\in \mathbb{R}$. Using the bound $\|G\|\leq \frac{1}{\eta}$ and the inequality $|\mathbf{x}^*G\mathbf{y}|\leq \|G\|\|\mathbf{x}\| \|\mathbf{y}\| $, we also~get
\begin{align}
\Lambda_T(E+\mathrm{i}\eta_\mathrm{M} )\leq \frac{1}{\eta_\mathrm{M}}\,, \qquad\qquad \wt{\Lambda}_T(E+\mathrm{i}\eta_\mathrm{M} )\leq \frac{1}{\eta_\mathrm{M}} \,,\label{17072102}
\end{align}
for any fixed $E\in \mathbb{R}$.  Since (\ref{17072101}) and (\ref{17072102}) guarantee assumption~(\ref{17020501}),   we can apply Proposition \ref{pro.17020310} to get, for any fixed $E\in \mathbb{R}$, that
\begin{align}
& \Lambda_{T}(E+\mathrm{i}\eta_\mathrm{M} )\prec \frac{1}{\sqrt{N }}\,,  \qquad \qquad \wt{\Lambda}_T(E+\mathrm{i}\eta_\mathrm{M} )\prec \frac{1}{\sqrt{N} }\,. \label{17072111}
\end{align}
 Also observe that $E+\mathrm{i}\eta_\mathrm{M} \in \mathcal{D}_{>}$, for any fixed $E$, and that $|\mathcal{S}(E+\mathrm{i}\eta_\mathrm{M} )|\gtrsim 1$. Hence $\Lambda(E+\mathrm{i}\eta_\mathrm{M} )\prec N^{-\varepsilon} |\mathcal{S}(E+\mathrm{i} \eta_\mathrm{M})|$.  From (\ref{17072101}), we can also conclude 
 \begin{align}
 \Lambda(E+\mathrm{i}\eta_\mathrm{M} ) \prec \frac{1}{\sqrt{N}}. \label{17072110}
 \end{align}
 Combining (\ref{17072101}), (\ref{17072111}), (\ref{17072110}) with the fact $\Lambda(E+\mathrm{i}\eta_\mathrm{M} )\prec N^{-\varepsilon} |\mathcal{S}(E+\mathrm{i}\eta_\mathrm{M} )|$, we see that the event   $\Theta_{>} (E+\mathrm{i}\eta_\mathrm{M} , {\frac{N^{3\varepsilon}}{N^{\frac13}}},{\frac{N^{3\varepsilon}}{\sqrt N}},  {\frac{\varepsilon}{10}})$ holds with high probability.   More quantitively, we have for any fixed $E$ that  
 \begin{align}
\mathbb{P} \Big(\Theta_{>} (E+\mathrm{i}\eta_\mathrm{M} , {\frac{N^{3\varepsilon}}{N^{\frac13}}},{\frac{N^{3\varepsilon}}{\sqrt N}},  {\frac{\varepsilon}{10}})\Big)\geq 1-N^{-D}\,,  \label{17072115}
\end{align}
for all $D>0$ and $N\geq N_2(D, \varepsilon)$ with some threshold $N_2(D, \varepsilon)$. 

Now  we take (\ref{17072115}) as the initial input, and use a continuity argument based on Lemma \ref{lem.19030901}, to control the probability of the ``good" events $\Theta_{>}$ for $z\in \mathcal{D}_{>}$ and $\Theta$  for $z\in \mathcal{D}_{\leq}$.  To this end, we first recall the event $\Omega(z)$ in Lemma \ref{lem.19030901}.  The main task is to show for any $z=E+\mathrm{i}\eta\in \mathcal{D}_{>}$, 
\begin{align}
\Theta_{>} \Big(E+\mathrm{i}\eta, {\frac{N^{\frac{5}{2}\varepsilon}}{(N\eta)^{\frac13}}},{\frac{N^{\frac{5}{2}\varepsilon}}{\sqrt{N\eta}}}, \frac{\varepsilon}{2}\Big)\cap \Omega\big(E+\mathrm{i}(\eta-N^{-5})\big)\subset  \Theta_{>} \Big(E+\mathrm{i}(\eta-N^{-5}), {\frac{N^{\frac{5}{2}\varepsilon}}{(N\eta)^{\frac13}}}, {\frac{N^{\frac{5}{2}\varepsilon}}{\sqrt{N\eta}}}, \frac{\varepsilon}{2}\Big), \label{17072123}
\end{align}
and, for any $z=E+\mathrm{i}\eta\in \mathcal{D}_{\leq}$, 
\begin{align}
\Theta \Big(E+\mathrm{i}\eta, {\frac{N^{\frac{5}{2}\varepsilon}}{(N\eta)^{\frac13}}},  {\frac{N^{\frac{5}{2}\varepsilon}}{\sqrt{N\eta}}}\Big)\cap \Omega\big(E+\mathrm{i}(\eta-N^{-5})\big)\subset  \Theta \Big(E+\mathrm{i}(\eta-N^{-5}), {\frac{N^{\frac{5}{2}\varepsilon}}{(N\eta)^{\frac13}}}, {\frac{N^{\frac{5}{2}\varepsilon}}{\sqrt{N\eta}}}\Big). \label{17072124}
\end{align}
The inclusions (\ref{17072123}) and (\ref{17072124}) are analogous to (7.20) of \cite{BES15b}. The only difference here is that we decompose the domain $\mathcal{D}_\tau(\eta_{\rm m},\eta_\mathrm{M})$ into  $\mathcal{D}_{>}$ and $\mathcal{D}_{\leq}$, and in $\mathcal{D}_{>}$ we also keep monitoring the event $\Lambda\leq N^{-\frac{\varepsilon}{2}}|\mathcal{S}|$ in order to use  Lemma 
\ref{lem.19030901}~$(i)$. As we are gradually reducing $\Im z$, once $z$ enters into the domain $\mathcal{D}_{\leq}$, we do not need to monitor $\mathcal{S}$ anymore.  

The proofs of (\ref{17072123}) and (\ref{17072124}) rely on the Lipschitz continuity of the Green function, $\|G(z)-G(z')\|\leq N^2|z-z'|$, and of the subordination functions and $\mathcal{S}$ in~\eqref{le lipschitz stuff}. Using the Lipschitz continuity of these functions, it is not difficult to see that
\begin{align}
&\Theta_{>} \Big(E+\mathrm{i}\eta, {\frac{N^{\frac{5}{2}\varepsilon}}{(N\eta)^{\frac13}}},{\frac{N^{\frac{5}{2}\varepsilon}}{\sqrt{N\eta}}}, \frac{\varepsilon}{2}\Big)\subset \Theta_{>} \Big(E+\mathrm{i}(\eta-N^{-5}), {\frac{N^{3\varepsilon}}{(N\eta)^{\frac13}}},{\frac{N^{3\varepsilon}}{\sqrt{N\eta}}}, \frac{\varepsilon}{10}\Big),\qquad  & & z=E+\mathrm{i}\eta\in \mathcal{D}_{>}\,, \label{17072130} \\
&\Theta \Big(E+\mathrm{i}\eta, {\frac{N^{\frac{5}{2}\varepsilon}}{(N\eta)^{\frac13}}},{\frac{N^{\frac{5}{2}\varepsilon}}{\sqrt{N\eta}}}\Big)\subset \Theta \Big(E+\mathrm{i}(\eta-N^{-5}), {\frac{N^{3\varepsilon}}{(N\eta)^{\frac13}}},{\frac{N^{3\varepsilon}}{\sqrt{N\eta}}}\Big), \qquad & & z=E+\mathrm{i}\eta\in \mathcal{D}_{\leq}\,. \label{17072131}
\end{align}
Then, (\ref{17072130}) together with (\ref{19031201}) implies  (\ref{17072123}). Similarly, (\ref{17072131}) together with  (\ref{19031202}) implies (\ref{17072124}). Applying (\ref{17072123}) and (\ref{17072124}) recursively and  using the simple fact that the domains $\mathcal{D}_{>}$ and $\mathcal{D}_{\leq}$ are connected, one can go from $\eta=\eta_\mathrm{M} $ to $\eta=\eta_{\rm m}$, step by step of size $N^{-5}$. 
Consequently, we obtain for any $\eta\in [\eta_{\rm m},\eta_\mathrm{M}]\cap N^{-5}\mathbb{Z}$  that, if $E+\mathrm{i}\eta\in \mathcal{D}_{>}$ then
\begin{multline}
\Theta_{>} (E+\mathrm{i}\eta_\mathrm{M} , {\frac{N^{\frac52\varepsilon}}{N^{\frac13}}},{\frac{N^{\frac52\varepsilon}}{\sqrt N}},  {\frac{\varepsilon}{2}})\cap \Omega(E+\mathrm{i}(\eta_\mathrm{M}-N^{-5}))\cap\ldots\cap \Omega(E+\mathrm{i}\eta)\\
\subset  \Theta_{>} \Big(E+\mathrm{i}\eta,{\frac{N^{\frac{5}{2}\varepsilon}}{(N\eta)^{\frac13}}}, {\frac{N^{\frac{5}{2}\varepsilon}}{\sqrt{N\eta}}}, \frac{\varepsilon}{2}\Big)\subset   \Theta \Big(E+\mathrm{i}\eta, {\frac{N^{\frac{5}{2}\varepsilon}}{(N\eta)^{\frac13}}}, {\frac{N^{\frac{5}{2}\varepsilon}}{\sqrt{N\eta}}}\Big)\,, \label{17072301}
\end{multline}
respectively, if $E+\mathrm{i}\eta\in \mathcal{D}_{\leq}$ then
\begin{align}
&\Theta_{>} (E+\mathrm{i}\eta_\mathrm{M} , {\frac{N^{\frac52\varepsilon}}{N^{\frac13}}},{\frac{N^{\frac52\varepsilon}}{\sqrt N}},  {\frac{\varepsilon}{2}})\cap \Omega(E+\mathrm{i}(\eta_\mathrm{M}-N^{-5}))\cap\ldots\cap \Omega(E+\mathrm{i}\eta)\subset   \Theta \Big(E+\mathrm{i}\eta, {\frac{N^{\frac{5}{2}\varepsilon}}{(N\eta)^{\frac13}}}, {\frac{N^{\frac{5}{2}\varepsilon}}{\sqrt{N\eta}}}\Big)\,.\label{17072302}
\end{align}

Combining (\ref{19031220}), (\ref{17072115}), (\ref{17072301}) and (\ref{17072302}), we have
\begin{align}
\mathbb{P}\Big(\Theta \Big(E+\mathrm{i}\eta, {\frac{N^{\frac{5}{2}\varepsilon}}{(N\eta)^{\frac13}}}, {\frac{N^{\frac{5}{2}\varepsilon}}{\sqrt{N\eta}}}\Big)\Big)\geq 1-N^{-D} (1+N^5(\eta_\mathrm{M}-\eta))\,,\label{17073030} 
\end{align}
uniformly for all $\eta\in [\eta_{\rm m}, \eta_\mathrm{M}]\cap N^{-5}\mathbb{Z}$, when $N\geq \max\{N_1(D, \varepsilon), N_2(D, \varepsilon)\}$.  Finally, by the Lipschitz continuity of  the Green function and also that of the subordination functions in~\eqref{le lipschitz stuff}, we can extend the bounds from $z$ in the discrete lattice to the entire domain $\mathcal{D}_\tau(\eta_{\rm m},\eta_\mathrm{M})$.

By the definition of the event $\Theta$ in (\ref{19031130}), we obtain from (\ref{17073030}) that
\begin{align}
\Lambda_{\rm d}(z)\leq  {\frac{N^{\frac{5}{2}\varepsilon}}{(N\eta)^{\frac13}}}\,, \quad \wt{\Lambda}_{\rm d}(z)\leq  {\frac{N^{\frac{5}{2}\varepsilon}}{(N\eta)^{\frac13}}}\,, \quad |\Lambda(z)|\leq {\frac{N^{\frac{5}{2}\varepsilon}}{(N\eta)^{\frac13}}} \quad 
|\Lambda_T(z)|\leq  {\frac{N^{\frac{5}{2}\varepsilon}}{\sqrt{N\eta}}}\quad |\wt{\Lambda}_T(z)|\leq  {\frac{N^{\frac{5}{2}\varepsilon}}{\sqrt{N\eta}}}\,, \label{17072320}
\end{align}
uniformly on $\mathcal{D}_\tau(\eta_{\rm m},\eta_\mathrm{M})$ with high probability.

Further, by (\ref{17072320}), we see that $\varphi(\Gamma_i)=1$ and $\varphi(\Gamma)=1$ hold uniformly on $\mathcal{D}_\tau(\eta_{\rm m},\eta_\mathrm{M})$, with high probability. Then it is easy to show that the conclusions in  (\ref{19030511})-(\ref{17030511}) also hold uniformly on $\mathcal{D}_\tau(\eta_{\rm m},\eta_\mathrm{M})$, with high probability. This concludes the proof of   Theorem \ref{thm. weak law at the edge}. 
\end{proof}

\section{Strong local law} \label{s. strong local law}
 In this section, we will prove the strong local law,~\ie Theorem \ref{thm. strong law at the edge}. 
From the weak local law in  Theorem \ref{thm. weak law at the edge}, we have the following rewriting of Proposition \ref{pro.17021715}, valid uniformly on $\mathcal{D}_\tau(\eta_{\rm m},\eta_\mathrm{M})$. 
\begin{pro}  Suppose that Assumptions \ref{a.regularity of the measures} and \ref{a. levy distance} hold. Fix any small $\gamma>0$. Suppose that $\Lambda(z)\prec \hat{\Lambda}(z)$, for some deterministic and positive function $\hat{\Lambda}(z)\prec N^{-\frac{\gamma}{4}}$, then
\begin{align}
&\Big|\mathcal{S}\Lambda_\iota+\mathcal{T}_\iota\Lambda_\iota^2+O(\Lambda_\iota^3)\Big|\prec   \frac{\sqrt{(\Im m_{\mu_A\boxplus\mu_B}+\hat{\Lambda})(|\mathcal{S}|+\hat{\Lambda})}}{N\eta}+\frac{1}{(N\eta)^2},\qquad \iota=A, B\,.  \label{19031250}
\end{align}
holds uniformly on $\mathcal{D}_\tau(\eta_{\rm m},\eta_\mathrm{M})$. 
\end{pro}
\begin{proof} The proof is the same as Proposition \ref{pro.17021715}. But now we have the weak local law Theorem \ref{thm. weak law at the edge}, which guarantees that the assumptions in Proposition \ref{pro.17020310} hold uniformly on $\mathcal{D}_\tau(\eta_{\rm m},\eta_\mathrm{M})$. Hence we do not need additional inputs in (\ref{17020501}), and the conclusion holds uniformly on $\mathcal{D}_\tau(\eta_{\rm m},\eta_\mathrm{M})$. This concludes the proof. 
\end{proof}
With the improved bound (\ref{19031250}) instead of the weaker one in (\ref{19030801}), we obtain the following improvement of Lemma \ref{lem.19031101}. 

\begin{lem} \label{19031270}  
 Let $\varepsilon\in (0, \frac{\gamma}{100})$. Let $\hat{\Lambda}=\hat{\Lambda}(z)\leq  N^{-\frac{\gamma}{4}}$ be some deterministic control parameter.  Suppose that $\Lambda\leq \hat{\Lambda}$.   Then we have the following estimates uniformly on $\mathcal{D}_\tau(\eta_{\rm m},\eta_\mathrm{M})$:

\noindent $(i)$: If $
\sqrt{\kappa+\eta}> N^{-\varepsilon}  \hat{\Lambda} 
$, there is a sufficiently large constant $K_0>0$, such that 
\begin{align}
\mathbf{1}\Big(\Lambda\leq \frac{|\mathcal{S}|}{K_0}\Big) |\Lambda_A| \prec  N^{-2\varepsilon} \hat{\Lambda}\,,\qquad\qquad\mathbf{1}\Big(\Lambda\leq \frac{|\mathcal{S}|}{K_0}\Big) |\Lambda_B| \prec  N^{-2\varepsilon} \hat{\Lambda}\,;\label{190312101}
\end{align}

\noindent $(ii)$: If $
\sqrt{\kappa+\eta}\leq  N^{-\varepsilon}  \hat{\Lambda} 
$, we have 
\begin{align}
|\Lambda_A|\prec  N^{-\varepsilon}  \hat{\Lambda}\,, \qquad\qquad|\Lambda_B|\prec N^{-\varepsilon}  \hat{\Lambda}.  \label{190312102}
\end{align}
\end{lem}
\begin{proof}  By (\ref{19031250}) and the fact $\Im m_{\mu_A\boxplus \mu_B}\lesssim |\mathcal{S}|$, we have 
\begin{align}
&\Big|\mathcal{S}\Lambda_\iota+\mathcal{T}_\iota\Lambda_\iota^2+O(\Lambda_\iota^3)\Big|\prec   \frac{|\mathcal{S}|+\hat{\Lambda}}{N\eta}+\frac{1}{(N\eta)^2},\qquad \iota=A, B\,.  \label{19031256}
\end{align}
Using (\ref{19031256}) instead of (\ref{19030801}), the remaining proof  is the same as the proof of Lemma \ref{lem.19031101}. 
\end{proof}
With the aid of Lemma \ref{19031270}, we can now prove (\ref{17072330}) and  (\ref{17011304}).

\begin{proof}[Proof of (\ref{17072330}) and  (\ref{17011304}) in Theorem \ref{thm. strong law at the edge}]  We first prove (\ref{17011304}).  Observe that by the weak local law in Theorem \ref{thm. weak law at the edge}, we have $\Lambda\prec \frac{1}{(N\eta)^{\frac13}}$. Hence, for any $z\in \mathcal{D}_\tau(\eta_{\rm m},\eta_\mathrm{M})$  we can start with choosing $\hat{\Lambda}(z)= \frac{N^{3\varepsilon}}{(N\eta)^{\frac13}}\ll N^{-\frac{\gamma}{4}}$ and apply Lemma \ref{19031270}.    Now, if $z\in \mathcal{D}_{>}$ (c.f. (\ref{190312100})), we can use (\ref{190312101}) iteratively (but for finitely many, $O(\varepsilon^{-1})$ times) to conclude that $\Lambda(z)\prec \frac{1}{N\eta}$.  For $z\in \mathcal{D}_{\leq }$, if $\sqrt{\kappa+\eta}\leq \frac{N^{2\varepsilon}}{N\eta}$, we  use (\ref{190312102}) iteratively until we get $\Lambda(z)\prec \frac{1}{N\eta}$.  If $z\in \mathcal{D}_{\leq }$ and $\sqrt{\kappa+\eta}> \frac{N^{2\varepsilon}}{N\eta}$, we shall first use (\ref{190312102}) iteratively until we get a bound $\hat{\Lambda}$ which satisfies $\sqrt{\kappa+\eta}>N^{-\varepsilon}\hat{\Lambda}$, then we use  (\ref{190312101})  for the iteration until we get $\Lambda(z)\prec \frac{1}{N\eta}$. Using (\ref{17073111}) we can get  (\ref{17011304}). 

Next, with the weak local law in Theorem \ref{thm. weak law at the edge}, it is also easy to see that Proposition \ref{lem. rough fluctuation averaging} holds uniformly on $\mathcal{D}_\tau(\eta_{\rm m},\eta_\mathrm{M})$. For any deterministic $d_1, \ldots, d_N\in \mathbb{C}$, we further write
\begin{align}
\frac{1}{N}\sum_{i=1}^N d_i \Big(G_{ii}-\frac{1}{a_i-\omega_B^c}\Big)= \frac{1}{N}\sum_{i=1}^N \frac{d_i}{\ntr G (a_i-\omega_B^c)}Q_i\,, \label{17072501}
\end{align}
which can easily be checked from the definition of $\omega_B^c$, $Q_i$ and the equation $(a_i-z)G_{ii}+(\wt{B}G)_{ii}=1$. Regarding $\frac{d_i}{\ntr G (a_i-\omega_B^c)}$ as the random coefficients $d_i$ in (\ref{170723113}), it is not difficult to check that (\ref{17022530}) holds, similarly to the last two equations in (\ref{17021202}). Hence, we have by Proposition~\ref{lem. rough fluctuation averaging} that
\begin{align}
\Big|\frac{1}{N}\sum_{i=1}^N d_i \Big(G_{ii}-\frac{1}{a_i-\omega_B^c}\Big)\Big|\prec \Psi\hat{\Pi}\,.  \label{17072321}
\end{align}
 Then combining the estimate $\Lambda(z)\prec \frac{1}{N\eta}$ with (\ref{17072321}) implies (\ref{17072330}). This concludes the proof of (\ref{17072330}) and  (\ref{17011304}) in (\ref{thm. strong law at the edge}).
\end{proof}

\section{Rigidity of the  eigenvalues} \label{s.rigidity}
In this section, we prove Theorem \ref{thm. rigidity of eigenvalues}, and also (\ref{17072847}) in Theorem \ref{thm. strong law at the edge}. We first  decompose the domain $\mathcal{D}_\tau(\eta_{\rm m},\eta_\mathrm{M})$ into the following two disjoint parts. Fix a small $\epsilon>0$ and set
\begin{align}
&\widehat{\mathcal{D}}_{>}\deq \Big\{z\in \mathcal{D}_\tau(\eta_{\rm m},\eta_\mathrm{M}): \sqrt{\kappa+\eta}> \frac{N^{2\varepsilon}}{N\eta} \Big\},\quad \widehat{\mathcal{D}}_{\leq}\deq\Big\{z\in \mathcal{D}_\tau(\eta_{\rm m},\eta_\mathrm{M}): \sqrt{\kappa+\eta}\leq \frac{N^{2\varepsilon}}{N\eta} \Big\}\,. 
\end{align}
We start by improving the estimate of $\Lambda$ defined in~\eqref{le gros lambda} in the following subdomain of $\widehat{\mathcal{D}}_>$,
\begin{align}
\wt{\mathcal{D}}_{>}\deq \{z=E+\mathrm{i}\eta\in\widehat{\mathcal{D}}_{>}: E<E_-\}\,, \label{17073110}
\end{align}
where $E_-$ is the lower endpoint of the support of the measure $\mu_\alpha\boxplus\mu_\beta$; see~\eqref{17080330}.

\begin{lem} \label{lem. away from the support} Suppose that the assumptions in Theorem~\ref{thm. strong law at the edge} hold. 
Then, we have the following uniform estimate for all $z\in\wt{\mathcal{D}}_{>}$,
\begin{align}
\Lambda(z)\prec \frac{1}{N\sqrt{(\kappa+\eta)\eta}}+\frac{1}{\sqrt{\kappa+\eta}}\frac{1}{(N\eta)^2}\,.\label{17072802}
\end{align}
\end{lem}
\begin{proof} First, from (\ref{17072320}), we see that $\Lambda\prec\frac{1}{N\eta}$ on $\mathcal{D}_\tau(\eta_{\rm m},\eta_\mathrm{M})$. Now, suppose that $\Lambda\prec  \hat{\Lambda}$ for some deterministic $ \hat{\Lambda}\equiv  \hat{\Lambda}(z)$ that satisfies
\begin{align}
N^{\varepsilon} \Big(\frac{1}{N\sqrt{(\kappa+\eta)\eta}}+\frac{1}{\sqrt{\kappa+\eta}}\frac{1}{(N\eta)^2}\Big)\leq  \hat{\Lambda}(z)\leq \frac{N^{\varepsilon}}{N\eta}\,. \label{17072801}
\end{align}
Observe that such $\hat{\Lambda}$ always exists on $\widehat{\mathcal{D}}_>$.
 From (\ref{17030301}), (\ref{17080120}) and (\ref{17080121}),  we have  for  $\iota=A, B$, and $z\in \wt{\mathcal{D}}_{>}$, 
\begin{align}
\Big|\mathcal{S}\Lambda_\iota+\mathcal{T}_\iota\Lambda_\iota^2\Big| &\prec   \frac{\sqrt{(\frac{\eta}{\sqrt{\kappa+\eta}}+ \hat{\Lambda})(\sqrt{\kappa+\eta}+ \hat{\Lambda})}}{N\eta}+\frac{1}{(N\eta)^2}\prec \frac{\sqrt{ \hat{\Lambda}\sqrt{\kappa+\eta}}}{N\eta}+\frac{\sqrt{\eta}}{N\eta}+\frac{1}{(N\eta)^2}\,, \label{170727100}
\end{align}
where we used that $ \hat{\Lambda} \prec \frac{N^{\varepsilon}}{N\eta}\leq N^{-\varepsilon}\sqrt{\kappa+\eta}$ for all $z\in \wt{\mathcal{D}}_{>}$. 
Moreover, for $z\in \wt{\mathcal{D}}_{>}$, we see that
\begin{align*}
|\Lambda_\iota|\prec \frac{1}{N\eta}\leq N^{-2\varepsilon}\sqrt{\kappa+\eta}\sim N^{-2\varepsilon}|\mathcal{S}|\,,
\end{align*} 
 for  $\iota=A, B$. Hence, according to the fact $\mathcal{T}_\iota\leq C$ (\cf (\ref{17080121})), we can absorb the second term on the left side of (\ref{170727100}) into the first term, and thus we have for  $\iota=A, B$
\begin{align*}
|\Lambda_\iota|\prec \frac{1}{\sqrt{\kappa+\eta}}  \bigg(\frac{\sqrt{ \hat{\Lambda}\sqrt{\kappa+\eta}}}{N\eta}+\frac{\sqrt{\eta}}{N\eta}+\frac{1}{(N\eta)^2}\bigg)\leq \frac{1}{N\eta(\kappa+\eta)^{\frac14}} \hat{\Lambda}^{\frac12}+N^{-\varepsilon}  \hat{\Lambda} \leq N^{-\frac{\varepsilon}{4}} \hat{\Lambda}\,, 
\end{align*}
where in the second step we used the lower bound in (\ref{17072801}) directly, and in the last step we used the fact $(N\eta)^{-1}(\kappa+\eta)^{-\frac14}\leq N^{-\frac{\varepsilon}{2}}\hat{\Lambda}^{\frac12}$ which again follows from the lower bound in (\ref{17072801}). 

Hence, we improved the bound from $\Lambda\leq  \hat{\Lambda}$ to $\Lambda\leq N^{-\frac{\varepsilon}{4}} \hat{\Lambda}$ as long as the lower bound in (\ref{17072801}) holds. Performing the above improvement iteratively,  one finally gets (\ref{17072802}). Hence, we complete the proof.
\end{proof}

With the aid of Lemma \ref{lem. away from the support}, we can now prove  Theorem \ref{thm. rigidity of eigenvalues}.
\begin{proof}[Proof of Theorem \ref{thm. rigidity of eigenvalues}] We first show  (\ref{17072845a}) for the smallest eigenvalue $\lambda_1$, \ie
\begin{align}
|\lambda_1-\gamma_1|\prec N^{-\frac23}\,.  \label{17072820}
\end{align}
Recall $\mathcal{K}$ defined in (\ref{17072840}). For any (small) constant $\varepsilon>0$, we define the line segment. 
\begin{align}
\wt{\mathcal{D}}(\varepsilon)\deq\{z=E+\mathrm{i}\eta: E\in [-\mathcal{K}, E_--N^{-\frac23+6\varepsilon}]\,, \,\eta=N^{-\frac{2}{3}+\varepsilon}\}.
\end{align}
Then it is easy to check that $\wt{\mathcal{D}}(\varepsilon)\subset \wt{\mathcal{D}}_{>}$ (\cf (\ref{17073110})). Applying \eqref{17072802}, we  obtain $\Lambda\prec \frac{N^{-\varepsilon}}{N\eta}$ uniformly on $\wt{\mathcal{D}}(\varepsilon)$, which together with (\ref{17073111}) implies
\begin{align}
|m_H(z)-m_{\mu_A\boxplus\mu_B}(z)|\prec \frac{N^{-\varepsilon}}{N\eta}\,, \label{17073130}
\end{align}
uniformly on $\wt{\mathcal{D}}(\varepsilon)$. Moreover, by (\ref{17080120}), we have 
\begin{align}
\Im m_{\mu_A\boxplus\mu_B}(z)\sim \frac{\eta}{\sqrt{\kappa+\eta}}\leq   \frac{N^{-\varepsilon}}{N\eta}\,,\label{17073131}
\end{align}
uniformly on $\wt{\mathcal{D}}(\varepsilon)$. Combining (\ref{17073130}) with (\ref{17073131}) yields
\begin{align}
\Im m_H(z)\prec  \frac{N^{-\varepsilon}}{N\eta}\,,  \label{17080315}
\end{align}
uniformly on $\wt{\mathcal{D}}(\varepsilon)$.  Since $\|H\|< \mathcal{K}$, to see (\ref{17072820}), it suffices to show 
 that with high probability $\lambda_1$ is not in the interval $[-\mathcal{K}, E_--N^{-\frac23+6\varepsilon}]$. We prove it by contradiction. Suppose that  $\lambda_1\in [-\mathcal{K}, E_--N^{-\frac23+6\varepsilon}]$.  Then clearly  for any $\eta>0$,
 \begin{align*}
 \sup_{E\in[-\mathcal{K}, E_--N^{-\frac23+6\varepsilon}]} \Im m_H(E+\mathrm{i}\eta)= \sup_{E\in [-\mathcal{K}, E_--N^{-\frac23+6\varepsilon}]}\frac{1}{N}\sum_{i=1}^N \frac{\eta}{(\lambda_i-E)^2+\eta^2}\geq \frac{1}{N\eta}\,,
 \end{align*}
 which contradicts the fact that (\ref{17080315}) holds uniformly on $\wt{\mathcal{D}}(\varepsilon)$. Hence, we have (\ref{17072820}).

Next, from (\ref{17011304}), (\ref{mdiff1}) and (\ref{mdiff2}) and a  standard application of Helffer-Sj{\"o}strand formula (\cf Lemma 5.1 \cite{AEK15}) on  $\mathcal{D}_\tau (\eta_{\rm m}, \eta_\mathrm{M})$ yields 
\begin{align}
\sup_{x\leq E_-+c}|\mu_H((-\infty,x])-\mu_A\boxplus\mu_B((-\infty, x])|\prec \frac{1}{N} \,,\label{17073140}
\end{align}
for any sufficiently small $c=c(\tau)$. 
Then (\ref{17072820}), (\ref{17073140}), together with the rigidity (\ref{rigi2}) and the square root behavior  of the distribution $\mu_\alpha\boxplus\mu_\beta$ (\cf (\ref{17080390})) will lead to the conclusion. The same conclusion holds with $\gamma_j^*$'s replaced by $\gamma_j$'s by rigidity (\ref{rigi2}). 
\end{proof}

Finally, with the aid of Theorem \ref{thm. rigidity of eigenvalues}, we can prove  (\ref{17072847}) in Theorem \ref{thm. strong law at the edge}. 
\begin{proof}[Proof of  (\ref{17072847}) in Theorem \ref{thm. strong law at the edge}]  Let $\varepsilon>0$ be any (small) constant.   Since $\kappa=E_--E\geq N^{-\frac23+\varepsilon}$ in~\eqref{17072847}, we see that  (\ref{17072847}) follows from (\ref{17011304}) directly in the regime $\eta\geq \frac{\kappa}{4}$, say. Hence, in the sequel, we work in the regime $\eta\leq \frac{\kappa}{4}$ only.  For any $z=E+\mathrm{i}\eta \in \mathcal{D}_\tau(\eta_{\rm m}, \eta_\mathrm{M})$ with $\kappa\geq N^{-\frac23+\varepsilon}$, we introduce the contour 
\begin{align*}
\mathcal{C}\equiv \mathcal{C}(z)\deq\mathcal{C}_l\cup\mathcal{C}_r\cup \mathcal{C}_u\cup \overline{\mathcal{C}}_u\,,
\end{align*}
where 
\begin{align*}
&\mathcal{C}_l\equiv  \mathcal{C}_l(z)\deq\big\{\tilde{z}=E+\frac{\kappa}{2}+\mathrm{i}\tilde{\eta}: -\eta-\kappa\leq \tilde{\eta} \leq \eta+\kappa\big\}\,,\nonumber\\
&\mathcal{C}_r\equiv  \mathcal{C}_r(z)\deq\big\{\tilde{z}=E-\frac{\kappa}{2}+\mathrm{i}\tilde{\eta}: -\eta-\kappa\leq \tilde{\eta} \leq \eta+\kappa\big\}\,,\nonumber\\
&\mathcal{C}_u\equiv \mathcal{C}_u(z)\deq \big\{ \tilde{z}= \tilde{E}+\mathrm{i}(\eta+\kappa): E-\frac{\kappa}{2}\leq \tilde{E}\leq E+\frac{\kappa}{2}\big\}\,.
\end{align*}
We then further decompose $\mathcal{C}=\mathcal{C}_{<}\cup\mathcal{C}_{\geq}$, where 
\begin{align*}
\mathcal{C}_<\equiv \mathcal{C}_<(z)\deq \big\{\tilde{z}\in \mathcal{C}: |\Im \tilde{z}|< \eta_{\rm m}\big\}, \qquad \mathcal{C}_{\geq}\equiv\mathcal{C}_{\geq} (z)\deq\mathcal{C}\setminus \mathcal{C}_<. 
\end{align*} 
Now, we further introduce the event 
\begin{align*}
\Xi\deq\bigcap_{\tilde{z}\in \mathcal{C}>}\Big\{ \big| m_H(\tilde{z})-m_{\mu_A\boxplus\mu_B}(\tilde{z})\big|\leq \frac{N^\varepsilon}{N\Im \tilde{z}}\Big\} \bigcap \Big\{\lambda_1\geq E_--\frac{1}{4}N^{-2/3+\varepsilon} \Big\}\,.
\end{align*}
Then, on the event $\Xi$, we have 
\begin{align}
m_H(z)-m_{\mu_A\boxplus\mu_B}(z) &= \frac{1}{2\pi \ii}\oint_{\mathcal{C}} \frac{1}{\tilde{z}-z} \big(m_H(\tilde{z})-m_{\mu_A\boxplus\mu_B}(\tilde{z})\big){\rm d} \tilde{z}\nonumber\\
&= \frac{1}{2\pi \ii}\Big(\int_{\mathcal{C}_<}+ \int_{\mathcal{C}_\geq }\Big)\frac{1}{\tilde{z}-z} \big(m_H(\tilde{z})-m_{\mu_A\boxplus\mu_B}(\tilde{z})\big){\rm d} \tilde{z}. \label{17080201}
\end{align}
Note that, for $\tilde{z}\in \mathcal{C}$, we always have  $\frac{1}{|\tilde{z}-z|}\leq \frac{2}{\kappa} $. In addition, for $\tilde{z}\in\mathcal{C}_<$, we have the fact $|\mathcal{C}_<|\leq \eta_{\rm m}$, and 
\begin{align*}
|m_H(\tilde{z})|\leq \frac{C}{\kappa}\,, \qquad \qquad|m_{\mu_A\boxplus\mu_B}(\tilde{z})|\leq \frac{C}{\kappa}\,,
\end{align*}
which hold on $\Xi$.  For $\tilde{z}\in\mathcal{C}_\geq$, we have the fact $|\mathcal{C}_\geq|\leq C\kappa$  and the bound  
\begin{align*}
\big| m_H(\tilde{z})-m_{\mu_A\boxplus\mu_B}(\tilde{z})\big|\leq \frac{N^\varepsilon}{N\Im \tilde{z}}\,,
\end{align*}
which holds on $\Xi$.  Applying the above bounds to (\ref{17080201}), it is elementary to check that
\begin{align*}
|m_H(z)-m_{\mu_A\boxplus\mu_B}(z)|\leq C\big(\eta_{\rm m}+N^{-1+\varepsilon}\log N\big) \frac{1}{\kappa}
\end{align*}
on $\Xi$.  Since $\gamma$ in $\eta_{\rm m}=N^{-1+\gamma}$ and $\varepsilon$ can be arbitrary, we can conclude that
\begin{align}
|m_H(z)-m_{\mu_A\boxplus\mu_B}(z)|\prec \frac{1}{N\kappa} \label{17080210}
\end{align}
if we can show that $\Xi$ holds with high probability.  Using (\ref{17072820}),  it suffices to show that 
\begin{align*}
\big| m_H(\tilde{z})-m_{\mu_A\boxplus\mu_B}(\tilde{z})\big|\prec \frac{1}{N\Im \tilde{z}}\,,
\end{align*}
uniformly in  $\tilde{z}\in\mathcal{C}_>$. This only requires enlarging the domain $\mathcal{D}_\tau(\eta_{\rm m}, \eta_\mathrm{M})$ and also consider its complex conjugate to include $\mathcal{C}_>$ during the proof of (\ref{17011304}). Hence, we conclude the proof of (\ref{17072847}) by combining the
$\frac{1}{N\kappa}$ bound in (\ref{17080210}) with the $\frac{1}{N\eta}$ bound in (\ref{17011304}).  
\end{proof}

We conclude the main part of the paper with the proof of Corollary~\ref{c. rigidity for whole spectrum}. 
\begin{proof}[Proof of Corollary  \ref{c. rigidity for whole spectrum}] With the additional  Assumption  \ref{a. rigidity entire spectrum}, we can show analogously that the estimates (\ref{17011304})  and (\ref{17072845a}) hold as well around the upper edge. According to Assumption \ref{a. rigidity entire spectrum} $(vii)$ and the fact $\sup_{\mathbb{C}^+}|m_{\mu_\alpha\boxplus\mu_\beta}|\leq C$ (\cf (\ref{17080326})), we see that except for the two vicinities of the lower and upper edge, the remaining spectrum is within the regular bulk.  Together with the strong local law in the bulk regime, \cf Theorem 2.4 in \cite{BES16}, we have 
\begin{align}
\big| m_H(z)-m_{\mu_A\boxplus\mu_B}(z)\big|\prec \frac{1}{N\eta}. \label{17080220}
\end{align}
uniformly on the domain $
\mathcal{D}(\eta_{\rm m}, \eta_\mathrm{M})\deq\{z=E+\ii \eta\in \mathbb{C}^+:  -\mathcal{K}\leq E\leq \mathcal{K}, \quad \eta_{\rm m}\leq \eta\leq \eta_\mathrm{M}\}$.
Then, (\ref{17080220}) together with (\ref{17072845a}) and its counterpart at the upper edge implies the rigidity for all eigenvalues, \ie~\eqref{17072845} can be proved again with Helffer-Sj{\"o}strand formula. Then, from (\ref{17072845}), we conclude that (\ref{17080225}) holds. This completes the proof of Corollary~\ref{c. rigidity for whole spectrum}.
\end{proof}

\appendix 
\section{} \label{appendix A} In this appendix, we collect some basic technical results.
\subsection{Stochastic domination and large deviation properties}\label{stochastic domination section}
Recall the stochastic domination in Definition~\ref{definition of stochastic domination}. The relation $\prec$ is transitive and it satisfies the following arithmetic rules: if $X_1\prec Y_1$ and $X_2\prec Y_2$ then $X_1+X_2\prec Y_1+Y_2$ and $X_1 X_2\prec Y_1 Y_2$. Further assume that $\Phi(v)\ge N^{-C}$ is deterministic and that~$Y(v)$ is a nonnegative random variable satisfying $\E [Y(v)]^2\le N^{C'}$ for all~$v$. Then $Y(v) \prec \Phi(v)$, uniformly in $v$, implies $\E [Y(v)] \prec \Phi(v)$, uniformly in~$v$.

Gaussian vectors have well-known large deviation properties which we use in the following form:
\begin{lem} \label{lem.091720} Let $X=(x_{ij})\in M_N(\C)$ be a deterministic matrix and let $\bs{y}=(y_{i})\in\C^N$ be a deterministic complex vector. For a Gaussian random vector $\mathbf{g}=(g_1,\ldots, g_N)\in \mathcal{N}_{\mathbb{R}}(0,\sigma^2 I_N)$ or $\mathcal{N}_{\mathbb{C}}(0,\sigma^2 I_N)$, we have
 \begin{align}\label{091731}
  |\bs{y}^* \bs{g}|\prec\sigma \|\bs{y}\| \,,\qquad\qquad  |\bs{g}^* X\bs{g}-\sigma^2N \ntr X|\prec \sigma^2\| X\|_2\,.
 \end{align}
\end{lem}
\subsection{Stability for large $\eta$} For any probability measures $\mu_1$ and $\mu_2$ on the real line, we define the functions $\Phi_1,\,\Phi_2: (\mathbb{C}^+)^3\to \mathbb{C}$ by setting

\begin{align}
\Phi_1(\omega_1, \omega_2,z)\deq F_{\mu_1} (\omega_2)-\omega_1-\omega_2+z\,,\qquad
\Phi_2(\omega_1,\omega_2,z)\deq F_{\mu_2} (\omega_1)-\omega_1-\omega_2+z
 \,.\label{170803110}
\end{align}
We observe that the system of subordination equations  (\ref{le definiting equations}) is equivalent to
\begin{align*}
\Phi_1(\omega_1(z), \omega_2(z),z)=0\,, \qquad \Phi_1(\omega_1(z), \omega_2(z),z)=0\,,\qquad\qquad \forall z\in \mathbb{C}^+. 
\end{align*}

We have the following linear stability for the subordination equation in the large $\eta$ regime.  
 A somewhat weaker version of this result has already been proven  
 in Lemma~4.2 of \cite{BES15} requiring   an unnecessarily stronger condition
  (compare (4.14) of \cite{BES15} with the current (\ref{170803100}) below).
 However, in our applications only a weaker assumption can be guaranteed. In fact, already 
 in \cite{BES15}  (in equation (6.56)) we tacitly relied  on the current version of this stability result.
 Thus by proving the stronger stability result below
  we  also correct this small inconsistency  in \cite{BES15}.
  
\begin{lem} \label{lem. stability for large eta} Let $\wt\eta_0>0$ be any (large) positive number and 
let  $\wt{\omega}_1, \wt{\omega}_2, \wt{r}_1, \wt{r}_2: \mathbb{C}_{\wt\eta_0}\to \mathbb{C}$ be analytic functions
where $\C_{\wt\eta_0}: = \{ z\in \C \; : \; \im z\ge \wt\eta_0\}$. Assume that there is a constant $C>0$ such 
 that the following hold  for all $z\in \mathbb{C}_{\wt\eta_0}$:
\begin{align}
& |\Im \wt{\omega}_1(z)-\Im z|\leq C\,, \qquad  & &|\Im \wt{\omega}_2(z)-\Im z|\leq C\,,\label{170803100}\\
&  |\wt{r}_1(z)|\leq C\,, \qquad & &|\wt{r}_2(z)|\leq C\,,\label{170803101}\\
& \Phi_1 (\wt{\omega}_1(z), \wt{\omega}_2(z), z)=\wt{r}_1(z)\,,\qquad & &\Phi_2 (\wt{\omega}_1(z), \wt{\omega}_2(z), z)=\wt{r}_2(z) \,.\label{170803102}
\end{align}
 
 Then there is a constant $\eta_0$ with $\eta_0 \ge \widetilde\eta_0$,  such that
 \begin{align}\label{le conclusion of large eta lemma}
 |\widetilde\omega_1(z)-\omega_1(z)|&\le 2\|\widetilde{r}(z)\|\,,\qquad\qquad|\widetilde\omega_2(z)-\omega_2(z)|\le 2\|\widetilde{r}(z)\|\,,
\end{align}
on the domain $\C_{\eta_0}\deq\{z\in\C\,:\,\im z\ge \eta_0\}$, where $\omega_1(z)$ and $\omega_2(z)$ are the subordination functions associated with~$\mu_1$ and~$\mu_2$.
\end{lem}
\begin{proof}  
Since most of the proof is identical to that  in \cite{BES15},
 here we only give the necessary modifications involving the weaker condition  (\ref{170803100}).
Following the proof in \cite{BES15} to the letter up to (4.23), for every $z\in \C_{\eta_0}$ 
 we have constructed functions $\widehat\omega_1(z)$, $\widehat\omega_2(z)$  such that $\Phi_{\mu_1, \mu_2}(\widehat\omega_1(z), \widehat\omega_2(z), z)=0$ with
 \begin{align}\label{170803104}
   |\wt\omega_j(z) -\widehat \omega_j(z)|\le 2\| \wt r(z)\|\,,\qquad \qquad j=1,2\,, \qquad z\in \C_{\eta_0}\,.
\end{align}
From (4.20) of \cite{BES15} we know that the Jacobian of the subordination equations 
(denoted by $\Gamma_{\mu_1, \mu_2}$ in \cite{BES15}) is close to 1 for sufficiently large $\wt\eta_0$.
Thus by analytic inverse
function theorem we obtain that $\widehat \omega_j(z)$, $j=1,2$, are also 
analytic functions for large $\eta=\im z$. From (\ref{170803100}), (\ref{170803101}) and (\ref{170803104}), we see that 
\begin{align*}
\lim_{\eta\nearrow \infty} \frac{\Im \widehat{\omega}_1(\ii \eta)}{\ii \eta} =\lim_{\eta\nearrow \infty} \frac{\Im \widehat{\omega}_2(\ii \eta)}{\ii \eta}=1\,.
\end{align*}
It is known from the proof of the uniqueness of the solution to the subordination equations near $z=\mathrm{i}\infty$ that $(\widehat{\omega}_1(z), \widehat{\omega}_2(z))$ is the unique solution  in a neighborhood of $z=\ii \infty$ and it can be analytically extended to all $z\in \mathbb{C}^+$.
Hence, $(\widehat{\omega}_1(z), \widehat{\omega}_2(z))=(\omega_1(z), \omega_2(z))$.   This together with (\ref{170803104}) concludes the proof.
\end{proof}

\section{} \label{appendix B}
In this appendix, we prove some technical lemmas. First, we estimate the small terms involving $\Delta_G$.  Specifically, we provide the bounds for the $\Delta_G$ involved terms in the  the last four estimates in Lemma~\ref{lem.17021201}. Then, we prove Lemma \ref{lem.17021201}. We summarize the estimates for $\Delta_G$ involved terms in the following lemma.
\begin{lem} \label{lem. estimate for delta terms} Fix a $z\in \mathcal{D}_\tau(\eta_{\rm m},\eta_\mathrm{M})$. Let $Q\in M_{N}(\mathbb{C})$ be arbitrary, with $\|Q\|\prec 1$. Let $X_i=I$ or $\wt{B}^{\la i\ra}$, and $X=I$ or $A$. Suppose the assumptions of Proposition \ref{pro.17020310} hold. Then, we have
\begin{align}
&\frac{1}{N} \sum_k^{(i)} \mathbf{e}_k^*X_i \Delta_G(i,k) \mathbf{e}_i=O_\prec (\Pi_i^2),\qquad & &\frac{1}{N} \sum_{k}^{(i)} \mathbf{e}_i^* X\Delta_G(i,k) \mathbf{e}_i\mathbf{e}_k^* X_iG\mathbf{e}_i  =O_\prec (\Pi_i^2),\nonumber\\
& \frac{1}{N} \sum_{k}^{(i)} \mathbf{h}_i^* \Delta_G(i,k) \mathbf{e}_i\mathbf{e}_k^* X_iG\mathbf{e}_i =O_\prec (\Pi_i^2) ,\qquad  & & \frac{1}{N} \sum_{k}^{(i)} \ntr QX \Delta_G(i,k) \mathbf{e}_k^* X_iG\mathbf{e}_i =O_\prec (\Psi^2\Pi_i^2).  \label{17030105}
\end{align}
\end{lem}
\begin{proof}
 The proof is similar to that of Lemma B.1 in \cite{BES16b}. But here we need finer estimates.
Recall $\Delta_R(i,k)$ and $\Delta_G(i,k)$ from (\ref{170729100}) and  (\ref{17022801}).  We note that $\Delta_R(i,k)$ is a sum of terms of the form $
\wt{d}_i \bar{g}_{ik} \boldsymbol{\alpha}_i\boldsymbol{\beta}_i^*$ 
for some $\wt{d}_i\in \mathbb{C}$ with $|\wt{d}_i|\prec 1$, where $\boldsymbol{\alpha}_i,\boldsymbol{\beta}_i=\mathbf{e}_i$ or $\mathbf{h}_i$. Hereafter, we use $\wt{d}_i$ to represent a generic number satisfying $|\wt{d}_i|\prec 1$ uniformly on $\mathcal{D}_\tau (\eta_{{\rm m}}, 1)$. Then, we see that $\Delta_G(i,k)$ is a sum of terms of the form
\begin{align}
\wt{d}_i \bar{g}_{ik}G\boldsymbol{\alpha}_i\boldsymbol{\beta}_i^*\wt{B}^{\la i\ra} R_i G,
 \qquad\qquad \wt{d}_i \bar{g}_{ik}GR_i\wt{B}^{\la i\ra}\boldsymbol{\alpha}_i\boldsymbol{\beta}_i^*G. \label{17030110}
\end{align}
Then, the left hand side of the first estimate in (\ref{17030105}) is a sum  of terms of the form 
\begin{align}
\frac{1}{N}\wt{d}_i \big(\mathring{\mathbf{g}}_i^*X_iG\boldsymbol{\alpha}_i\big) \big(\boldsymbol{\beta}_i^*\wt{B}^{\la i\ra} R_i G \mathbf{e}_i\big),\qquad \qquad \frac{1}{N} \wt{d}_i\big(\mathring{\mathbf{g}}_i^*X_i GR_i\wt{B}^{\la i\ra}\boldsymbol{\alpha}_i\big) \big(\boldsymbol{\beta}_i^*G \mathbf{e}_i\big). \label{17030104}
\end{align}
By the Cauchy-Schwarz inequality, we have
\begin{align}
&\big|\mathring{\mathbf{g}}_i^*X_iG\boldsymbol{\alpha}_i\big|\prec \|G\boldsymbol{\alpha}_i\| =\sqrt{ \frac{\Im \boldsymbol{\alpha}_i^*G\boldsymbol{\alpha}_i}{\eta}},\nonumber\\
& \big|\boldsymbol{\beta}_i^*\wt{B}^{\la i\ra} R_i G \mathbf{e}_i\big|\prec  \|G\mathbf{e}_i\| =\sqrt{ \frac{\Im G_{ii}}{\eta}},\qquad \big|\boldsymbol{\beta}_i^*G \mathbf{e}_i\big|\prec  \|G\mathbf{e}_i\| =\sqrt{ \frac{\Im G_{ii}}{\eta}},\nonumber\\
& \big|\mathring{\mathbf{g}}_i^*X_i GR_i\wt{B}^{\la i\ra}\boldsymbol{\alpha}_i\big|\prec \|GR_i\wt{B}^{\la i\ra}\boldsymbol{\alpha}_i \| =  \sqrt{  \frac{  \Im \boldsymbol{\alpha}_i^*\wt{B}^{\la i\ra}R_iGR_i\wt{B}^{\la i\ra}\boldsymbol{\alpha}_i}{\eta}}. \label{17030103}
\end{align}
Note that for $\boldsymbol{\alpha}_i=\mathbf{e}_i$, 
\begin{align}
\boldsymbol{\alpha}_i^*G\boldsymbol{\alpha}_i=G_{ii},\qquad \qquad\boldsymbol{\alpha}_i^*\wt{B}^{\la i\ra}R_iGR_i\wt{B}^{\la i\ra}\boldsymbol{\alpha}_i=b_i^2\mathbf{h}_i^*G\mathbf{h}_i=b_i^2 \mathcal{G}_{ii},  \label{17030101}
\end{align}
and for $\boldsymbol{\alpha}_i=\mathbf{h}_i$, 
\begin{align}
\boldsymbol{\alpha}_i^*G\boldsymbol{\alpha}_i=\mathcal{G}_{ii}, \qquad \qquad\boldsymbol{\alpha}_i^*\wt{B}^{\la i\ra}R_iGR_i\wt{B}^{\la i\ra}\boldsymbol{\alpha}_i=\mathbf{e}_i^* \wt{B}G\wt{B}\mathbf{e}_i=\wt{B}_{ii}-(a_i-z)+(a_i-z)G_{ii}. \label{17030102}
\end{align}
Plugging (\ref{17030101}) and (\ref{17030102}) into the bounds in (\ref{17030103}), we see that both terms in (\ref{17030104}) are of order $O_\prec(\Pi_i^2)$.  Hence, we proved the first estimate in (\ref{17030105}). 

Next, we verify the second estimate (\ref{17030105}).  Since $\Delta_G(i,k)$ is a sum of terms of the form in (\ref{17030110}), we see that the left side of the second estimate in (\ref{17030105}) is a sum of terms of the form
\begin{align}
\frac{1}{N}\wt{d}_i  \big(\mathbf{e}_i^* XG\boldsymbol{\alpha}_i \big) \big(\boldsymbol{\beta}_i^*\wt{B}^{\la i\ra} R_i G \mathbf{e}_i \big) \big(\mathring{\mathbf{g}}_i^* X_iG\mathbf{e}_i\big)\,,\qquad\qquad  \frac{1}{N}\wt{d}_i   \big(\mathbf{e}_i^* XGR_i\wt{B}^{\la i\ra}\boldsymbol{\alpha}_i \big) \big(\boldsymbol{\beta}_i^*G \mathbf{e}_i\big) \big(\mathring{\mathbf{g}}_i^* X_iG\mathbf{e}_i  \big)\,.\label{17030125}
\end{align}
Note that
\begin{align*}
\mathbf{e}_i^*\wt{B}^{\la i\ra} R_i G \mathbf{e}_i=-b_iT_i,\qquad \qquad\mathbf{h}_i^*\wt{B}^{\la i\ra} R_i G \mathbf{e}_i= -(\wt{B}G)_{ii}. 
\end{align*}
Hence, we have 
\begin{align}
|\boldsymbol{\beta}_i^*\wt{B}^{\la i\ra} R_i G \mathbf{e}_i|\prec 1,\qquad\qquad |\boldsymbol{\beta}_i^*G \mathbf{e}_i|\prec 1. \label{17030120}
\end{align}
Further, we claim that 
\begin{align}
|\mathbf{e}_i^* XG\boldsymbol{\alpha}_i|, \; |\mathbf{e}_i^* XGR_i\wt{B}^{\la i\ra}\boldsymbol{\alpha}_i| \prec \sqrt{\frac{\Im (G_{ii}+\mathcal{G}_{ii})}{\eta}}. \label{17030121}
\end{align}
The proof of the above bounds is analogous to the proof of (\ref{17030103}). We thus omit the details.  Then, using the first estimate in (\ref{17030103}), (\ref{17030120}) and (\ref{17030121}), we see that both terms in (\ref{17030125}) are of order $O_\prec(\Pi_i^2)$. 

The proof of the third estimate in (\ref{17030105}) is nearly the same as that for the second one, we thus omit~it.

To show the last estimate, we again use the fact that $\Delta_G(i,k)$ is a sum of terms of the form in (\ref{17030110}). Then it is not difficult to see that the left side of the last estimate in (\ref{17030105}) is a sum of terms of the form 
\begin{align}
  \frac{\wt{d}_i }{N^2}   \big(\boldsymbol{\beta}_i^*\wt{B}^{\la i\ra} R_i GQXG\boldsymbol{\alpha}_i  \big) \big(\mathring{\mathbf{g}}_i^* X_iG\mathbf{e}_i\big),\qquad
   \frac{\wt{d}_i}{N^2}  \big( \boldsymbol{\beta}_i^*G QX  GR_i\wt{B}^{\la i\ra}\boldsymbol{\alpha}_i \big) \big(\mathring{\mathbf{g}}_i^* X_iG\mathbf{e}_i\big). \label{17030134}
\end{align}
Note that
\begin{align}
\big| \boldsymbol{\beta}_i^*\wt{B}^{\la i\ra} R_i GQXG\boldsymbol{\alpha}_i\big| \prec \frac{1}{\eta} \|G\boldsymbol{\alpha}_i\| \leq \frac{1}{\eta} \sqrt{\frac{\Im (G_{ii}+\mathcal{G}_{ii})}{\eta}}.  \label{17030130}
\end{align}
Analogously, we have
\begin{align}
\big|\boldsymbol{\beta}_i^*G QX  GR_i\wt{B}^{\la i\ra}\boldsymbol{\alpha}_i\big| \prec \frac{1}{\eta} \sqrt{\frac{\Im (G_{ii}+\mathcal{G}_{ii})}{\eta}}.  \label{17030131}
\end{align}
Applying  (\ref{17030130}), (\ref{17030131}), and the first estimate in (\ref{17030103}), we see that both terms in (\ref{17030134}) are of order $O_\prec(\Psi^2 \Pi_i^2)$. Hence, we obtain the last estimate in (\ref{17030105}). This concludes the proof of Lemma \ref{lem. estimate for delta terms}.
\end{proof}

\begin{proof}[Proof of Lemma \ref{lem.17021201}]  The proof is similar to that for Lemma 7.4 in \cite{BES16b}. In the latter, we used $\Psi$ instead of $\Pi_i$ in the statement. However, the proof of Lemma 7.4 in \cite{BES16b} shows readily that the stronger bounds in (\ref{17021202}) hold for the counterparts of the  block additive model  (\cf  (7.77), (7.80), (7.81) and (7.87) of \cite{BES16b}). The proof for our additive model given here analogous. 

First, by (\ref{17020505}), (\ref{170726100}),  (\ref{17020531}), (\ref{17020534}), and the fact $\mathring{T}_i=T_i-h_{ii}G_{ii}$, we have $|\mathring{S}_i|\prec 1$, $|\mathring{T}_i|\prec 1$, under the assumption ((\ref{17020501}).  Then, for the first estimate in (\ref{17021202}), we have 
\begin{align*}
\frac{1}{N} \sum_k^{(i)}  \frac{\partial \|\mathbf{g}_i\| ^{-1}}{\partial g_{ik}} \mathbf{e}_k^* X_i G\mathbf{e}_i=-\frac{1}{2N}\frac{1}{\|\mathbf{g}_i\|^3} \sum_{k}^{(i)} \bar{g}_{ik} \mathbf{e}_k^* X_i\mathbf{e}_i=-\frac{1}{2N}\frac{1}{\|\mathbf{g}_i\|^2} \mathring{\mathbf{h}}_i^* X_iG\mathbf{e}_i=O_\prec(\frac{1}{N}), 
\end{align*}
where we used the fact that $\mathring{\mathbf{h}}_i^* X_iG\mathbf{e}_i=\mathring{S}_i$ or $\mathring{T}_i$ if $X_i=\wt{B}^{\la i\ra}$ or $I$, respectively. 

Next, we show the second bound in (\ref{17021202}). It is convenient to set $
I^{\la i\ra}\deq I-\mathbf{e}_i\mathbf{e}_i^*$.  
Using (\ref{17071801}), we~get 
\begin{align}
&\frac{1}{N} \sum_k^{(i)} \mathbf{e}_i^*X \frac{\partial G}{\partial g_{ik}} \mathbf{e}_i\mathbf{e}_k^* X_i G\mathbf{e}_i=  \frac{c_i}{N} \mathbf{e}_i^*XGI^{\la i\ra} X_i G\mathbf{e}_i (\mathbf{e}_i+\mathbf{h}_i^*) \wt{B}^{\la i\ra} R_i G \mathbf{e}_i\nonumber\\
&\qquad\qquad+ \frac{c_i}{N}  \mathbf{e}_i^*XGR_i\wt{B}^{\la i\ra}I^{\la i\ra} X_i G\mathbf{e}_i (\mathbf{e}_i+\mathbf{h}_i)^*G \mathbf{e}_i+\frac{1}{N} \sum_k^{(i)} \mathbf{e}_i^*X\Delta_G(i,k) \mathbf{e}_i\mathbf{e}_k^* X_i G\mathbf{e}_i. \label{1707}
\end{align}
The desired estimate of the last term was obtained in the second line of (\ref{17030105}).  Further, using (\ref{170726105}) we~get
\begin{align*}
(\mathbf{e}_i+\mathbf{h}_i^*) \wt{B}^{\la i\ra} R_i G \mathbf{e}_i=-b_iT_i-(\wt{B}G)_{ii}=O_\prec(1), \qquad\qquad (\mathbf{e}_i+\mathbf{h}_i)^*G \mathbf{e}_i=G_{ii}+T_i=O_\prec(1), 
\end{align*}
where the estimates follows from (\ref{17020505}) and (\ref{170726100}). Hence, it suffices to show that 
\begin{align}
|\mathbf{e}_i^*XGI^{\la i\ra} X_i G\mathbf{e}_i|\prec \frac{\Im (G_{ii}+\mathcal{G}_{ii})}{\eta},\qquad\qquad|\mathbf{e}_i^*XGR_i\wt{B}^{\la i\ra}I^{\la i\ra} X_i G\mathbf{e}_i|\prec  \frac{\Im (G_{ii}+\mathcal{G}_{ii})}{\eta}.  \label{170726110}
\end{align}
Note that, by the assumption  $X=I$ or $A$,  both terms in (\ref{170726110}) can be  bounded by 
\begin{align*}
C\|GX\mathbf{e}_i\|\|G\mathbf{e}_i\|=\frac{C}{\eta} \sqrt{\Im (XGX)_{ii}} \sqrt{\Im G_{ii}}\leq C'\frac{\Im G_{ii}}{\eta}. 
\end{align*}
This completes the proof of the second inequality in (\ref{17021202}). Next, we show the third estimate in (\ref{17021202}). In light of the definition of $T_i$, it suffices to show 
\begin{align}
\frac{1}{N}  \sum_k^{(i)} \frac{\partial \mathbf{h}_i^*}{ \partial g_{ik}}G\mathbf{e}_i \mathbf{e}_k^* X_i G\mathbf{e}_i= O_\prec(\frac{1}{N}), \qquad \qquad \frac{1}{N}  \sum_k^{(i)} \mathbf{h}_i^* \frac{\partial G}{ \partial g_{ik}}\mathbf{e}_i \mathbf{e}_k^* X_i G\mathbf{e}_i= O_\prec(\Pi_i^2).  \label{170726120}
\end{align} 
The first estimate in (\ref{170726120}) is proved as follows
\begin{align*}
\frac{1}{N}  \sum_k^{(i)} \frac{\partial \mathbf{h}_i^*}{ \partial g_{ik}}G\mathbf{e}_i \mathbf{e}_k^* X_i G\mathbf{e}_i &=-\frac{1}{2\|\mathbf{g}_i\|^2} \frac{1}{N}\sum_{k}^{(i)}\bar{h}_{ik} \mathbf{e}_k^* X_i G\mathbf{e}_i \mathbf{h}_i^* G\mathbf{e}_i\nonumber\\
&=-\frac{1}{2\|\mathbf{g}_i\|^2}\frac{1}{N}\mathring{\mathbf{h}}_i X_iG\mathbf{e}_i\mathbf{h}_i^*G\mathbf{e}_i=O_\prec(\frac{1}{N}),
\end{align*}
where in the last step we again use the fact $\mathring{\mathbf{h}}_i^* \wt{B}^{\la i\ra}G\mathbf{e}_i=\mathring{S}_i=O_\prec(1)$ and $\mathbf{h}_i^*G\mathbf{e}_i=T_i=O_\prec(1)$. 
The proof of the second estimate in (\ref{170726120}) is similar to that of the second inequality in (\ref{17021202}). It suffices to replace $\mathbf{e}_i^*X$ by $\mathbf{h}_i^*$ in (\ref{1707}) and estimate the resulting terms. The counterpart to the last term in (\ref{1707}) is estimated in (\ref{17030105}). The counterparts to the first two terms on the right side of (\ref{1707}) are bounded by 
\begin{align*}
C\|G\mathbf{h}_i\|\|G\mathbf{e}_i\|=\frac{C}{\eta} \sqrt{\Im \mathbf{h}_i^* G\mathbf{h}_i} \sqrt{\Im G_{ii}}= \frac{C}{\eta} \sqrt{\Im \mathcal{G}_{ii}} \sqrt{\Im G_{ii}}\leq C'\frac{\Im(G_{ii}+\mathcal{G}_{ii})}{\eta},
\end{align*}
where we have used (\ref{170726140}). 

Next, we show the fourth estimate in (\ref{17021202}).  Using (\ref{17071801}) again, we can get  
\begin{align}
&\frac{1}{N}  \sum_k^{(i)} \ntr \Big(Q X\frac{\partial G}{\partial g_{ik}}\Big) \mathbf{e}_k^* X_i G\mathbf{e}_i= \frac{c_i}{N^2} (\mathbf{e}_i+\mathbf{h}_i)^* \wt{B}^{\la i\ra} R_i GQXGI^{\la i\ra} X_iG\mathbf{e}_i\nonumber\\
&\qquad + \frac{c_i}{N^2} (\mathbf{e}_i+\mathbf{h}_i)^* GQ XG R_i\wt{B}^{\la i\ra}I^{\la i\ra} X_iG\mathbf{e}_i+\frac{1}{N}\sum_{k}^{(i)}\ntr QX\Delta_G(i,k)\mathbf{e}_k^* X_iG\mathbf{e}_i.  \label{170726153}
\end{align}
The last term above is estimated in (\ref{17030105}).  Using (\ref{170726105}) and  $\|G\|\leq \eta$, we have 
\begin{align}
&\Big|\frac{1}{N^2} (\mathbf{e}_i+\mathbf{h}_i)^* \wt{B}^{\la i\ra} R_i GQXGI^{\la i\ra} X_iG\mathbf{e}_i\Big|=\Big|\frac{1}{N^2} (b_i\mathbf{h}_i^*+\mathbf{e}_i^*\wt{B}) G QXGI^{\la i\ra} X_iG\mathbf{e}_i\Big| \nonumber\\
&\qquad \leq C\frac{1}{N^2\eta} \big(\|G\mathbf{h}_i\|+\|G\wt{B}\mathbf{e}_i\|\big)\|G\mathbf{e}_i\|\leq C\frac{1}{N^2\eta} \big(\|G\mathbf{h}_i\|^2+\|G\wt{B}\mathbf{e}_i\|^2+\|G\mathbf{e}_i\|^2\big) \nonumber\\
&\qquad = \frac{C}{N^2\eta^2} \big(\Im (\mathbf{h}_i^*G\mathbf{h}_i+(\wt{B}G\wt{B})_{ii}+G_{ii})\big)\prec \frac{\Im (G_{ii}+\mathcal{G}_{ii})}{N^2\eta^2}.  \label{170726157}
\end{align}
Here in the last step we again used (\ref{170726140}) and also fact
\begin{align}
\Im (\wt{B}G\wt{B})_{ii}=\eta+ \Im ((a_i-z)^2 G_{ii})= O_\prec (\eta+\Im G_{ii})=O_\prec(\Im G_{ii}). \label{170726150}
\end{align}
In (\ref{170726150}), we used  (\ref{17020508}), the first bound in (\ref{17020505}), and $\Im G_{ii}\gtrsim \eta$ which is easily checked by spectral decomposition. Similar to (\ref{170726157}), we get the desired estimate for the second term on the right of~(\ref{170726153}). 

Finally, the last equation in (\ref{17021202}) can be proved analogously to the fourth one. 
The only difference is, instead of the factor $\mathbf{e}_k^*X_i G\mathbf{e}_i$ in (\ref{17021303}), here we have $\mathbf{e}_k^* X_i \mathring{\mathbf{g}}_i$ which does not contain any $G$ factor, which actually makes the estimates even simpler. This completes the proof of Lemma \ref{lem.17021201}. 
\end{proof}

\section{Estimates of the cutoff errors} \label{appendix C}
In this appendix, we state more details on the estimate (\ref{19030501}). The proof can be done in the same way as the non-cutoff version (\ref{17071833}), but with the a priori inputs given by $\varphi(\Gamma_i)$'s and $\varphi(\Gamma)$. Since the proof can be done via  going through the proof of (\ref{17071833}) again, we only list the necessary modifications here.  

The first modification we need to do is the bound of the analogue of the term $O_\prec(\Psi\hat{\Upsilon})$ in  (\ref{17072903}). This error term was obtained when we bounded  the term 
$\frac{1}{N}\sum_{i=1}^N T_i\tau_{i1}\Upsilon$ in (\ref{170728100}). Here, during the proof of (\ref{19030501}), the counterpart will be $\frac{1}{N}\sum_{i=1}^N T_i\wt{\tau}_{i1} \Upsilon$, where $\wt{\tau}_{i1}$ is defined via replacing all $d_j$'s by $d_j\varphi(\Gamma_j)\varphi(\Gamma)$ in the definition of $\tau_{i1}$ in (\ref{17021305}).  According to the definition of $\Gamma$ in  (\ref{19030920}), it is easy to see that 
\begin{align*}
\Big|\frac{1}{N}\sum_{i=1}^N T_i\wt{\tau}_{i1} \Upsilon\Big|\leq  C\frac{N^{10\varepsilon}}{(N\eta)^{\frac56}}\ll \frac{1}{(N\eta)^{\frac23}}
\end{align*}
when $\varepsilon$ is sufficiently small. Hence, the term $\frac{1}{N}\sum_{i=1}^N T_i\wt{\tau}_{i1} \Upsilon$ can be absorbed into the bound for $\mathfrak{c}_1$ in (\ref{19030950}). 

The second modification we need to do is the estimate for the analogue of (\ref{17021311}). We take the case $j=1$ for example. In the step of (\ref{17021550}), we used the  estimate $\Lambda_{{\rm d}i}^c\prec \Psi$ from (\ref{17020303}) to replace $G_{ii}$ in the definition of $\varepsilon_{i1}$ in (\ref{17071805}) by $\frac{1}{a_i-\omega_B^c}$ in (\ref{17021550}), and also the bound of $T_i$ in  (\ref{17020303}) was used in (\ref{17021550}). But now, lacking the conditions in 
(\ref{17020501}), these bounds are not available. Instead, we shall need to extract similar information from the presence of the cutoff functions $\varphi(\Gamma_i)$'s and $\varphi(\Gamma)$. 
 The analogue of $\varepsilon_1$ in the proof of (\ref{19030501}) can be written as  
\begin{align*}
\wt{\varepsilon}_1=\frac{1}{N}\sum_{i=1}^N \varepsilon_{i1} \ntr G\wt{\tau}_{i1}= \frac{1}{N}\sum_{i} \mathring{\mathbf{h}}_i^*\wt{B}^{\la i\ra} \mathring{\mathbf{h}}_i G_{ii}\wt{\tau}_{i1}+\wt{\delta}_1, 
\end{align*}
with $\wt{\delta}_1$ satisfying 
\begin{align}
\mathbb{E}|\wt{\delta}_1|^k=o\Big(\frac{1}{(N\eta)^{\frac{2k}{3}}}\Big) \label{19031301}
\end{align}
for any given $k>0$. In the estimate (\ref{19031301}), we 
again used the fact $\frac{1}{N}\sum_{i=1}^N|T_i|\varphi(\Gamma)\leq C\frac{N^{2\varepsilon}}{\sqrt{N\eta}}$ to estimate the average of the second term in  $\varepsilon_{i1}$ (c.f. (\ref{17071805})). Therefore, our task is to prove the weaker but unconditional  estimate
\begin{align}
\mathbb{E} \Big[\frac{1}{N}\sum_{i} \mathring{\mathbf{h}}_i^*\wt{B}^{\la i\ra} \mathring{\mathbf{h}}_i G_{ii}\wt{\tau}_{i1} \wt{\mathfrak{m}}^{(p-1,p)}\Big]=\mathbb{E}\big[\mathfrak{c}_{\varepsilon 1}\wt{\mathfrak{m}}^{(p-1,p)}\big]+\mathbb{E}\big[\mathfrak{c}_{\varepsilon 2} \wt{\mathfrak{m}}^{(p-2,p)}\big]+\mathbb{E}\big[\mathfrak{c}_{\varepsilon 3} \wt{\mathfrak{m}}^{(p-1,p-1)}\big],   \label{19030960}
\end{align}
where 
\begin{align*}
|\mathfrak{c}_{\varepsilon 1}|\leq C\hat{\Pi}, \qquad |\mathfrak{c}_{\varepsilon 2}|\leq C\hat{\Pi}^2, \qquad |\mathfrak{c}_{\varepsilon 3}|\leq C\hat{\Pi}^2, \qquad \text{on}\quad \widehat{\Omega}_2(z). 
\end{align*}
Moreover, the $\mathfrak{c}_{\varepsilon i}$'s  also admit the moment bound  $\mathbb{E}|\mathfrak{c}_{\varepsilon i}|^k=O(1)$ for any given $k>0$. The proof of (\ref{19030960}) can be done basically in the same way as the estimate for (\ref{17021531}), we thus omit the details. 

The last and also the major modification is: the smooth cutoffs $\varphi(\Gamma_i)$ and $\varphi(\Gamma)$ bring in new terms during the integration by parts. More specifically, we will need to consider the derivative of the cutoffs. The derivatives of the cutoffs $\varphi(\Gamma_i)$'s can be treated similarly to the case in the proof of (\ref{19030501111}). In the sequel, we investigate the derivative of the term $\varphi(\Gamma)$.  For instance, in the analogue of the step (\ref{17021250}), the counterpart of the third term on the right side of (\ref{17021250}) will be 
\begin{align}
\frac{1}{N^2} \sum_{i=1}^N  \sum_k^{(i)} \mathbb{E}\Big[  \frac{1}{\|\mathbf{g}_i\|}\mathbf{e}_k^*\wt{B}^{\la i\ra} G\mathbf{e}_i \frac{\partial (\ntr G \wt{\tau}_{i1})}{\partial g_{ik}} \wt{\mathfrak{m}}^{(p-1,p)}\Big]. \label{19030971}
\end{align}
One new term in $\frac{\partial (\ntr G \wt{\tau}_{i1})}{\partial g_{ik}}$ is 
\begin{align}
d_i\ntr G\;\varphi(\Gamma_i) \varphi'(\Gamma) \frac{\partial \Gamma}{\partial g_{ik}}.  \label{19030970}
\end{align} 
In the sequel, we show the contribution of the term (\ref{19030970}) to (\ref{19030971}).  The other terms involving the derivatives of the cutoffs can be treated similarly.  To show the contribution of (\ref{19030970}), it suffices to prove the following three estimates
\begin{align}
&(c\Im m_{\mu_A\boxplus\mu_B}+\hat{\Lambda})^{-2} \frac{1}{N^2} \sum_{i=1}^N  \sum_k^{(i)} \hat{d}_i\mathbf{e}_k^*\wt{B}^{\la i\ra} G\mathbf{e}_i  \frac{\partial (|\Lambda_A|^2+|\Lambda_B|^2)}{\partial g_{ik}} \leq C\hat{\Pi},\label{190309100}\\
&\Big(\frac{N^{5\varepsilon}}{(N\eta)^{\frac13}}\Big)^{-2}  \frac{1}{N^2} \sum_{i=1}^N  \sum_k^{(i)} \hat{d}_i\mathbf{e}_k^*\wt{B}^{\la i\ra} G\mathbf{e}_i  \frac{\partial |\Upsilon|^2}{\partial g_{ik}}\leq C\hat{\Pi},\label{190309101}\\
&\Big(\frac{N^{5\varepsilon}}{\sqrt{N\eta}}\Big)^{-1} \frac{1}{N^2} \sum_{i=1}^N  \sum_k^{(i)} \hat{d}_i\mathbf{e}_k^*\wt{B}^{\la i\ra} G\mathbf{e}_i  \frac{\partial \frac{1}{N}\sum_{j=1}^N (| T_{j}|^2+N^{-1})^{\frac12}}{\partial g_{ik}}\leq C\hat{\Pi}, \label{190309102}
\end{align}
where we introduced the shorthand notation 
\begin{align*}
\hat{d}_i:= d_i\ntr G\;\varphi(\Gamma_i) \varphi'(\Gamma) \frac{1}{\|\mathbf{g}_i\|}. 
\end{align*}
To show (\ref{190309100}), we first note that $|\Lambda_\iota|^2=\Lambda_\iota\overline{\Lambda_\iota} $,  and
\begin{align}
 |\Lambda_\iota|\leq C (c\Im m_{\mu_A\boxplus\mu_B}+\hat{\Lambda}), \qquad \iota=A,B,  \label{190309110}
\end{align}
if $\hat{d}_i\neq 0$ for at least one $i$ by the definition of $\Gamma$ in (\ref{19030920})
 and $\varphi'(\Gamma) \ne 0 $ implying $\Gamma\le C$. Further, combining (\ref{190309110}) with (\ref{17073101}), we also have $|m_H-m_{\mu_A\boxplus\mu_B}|\leq C (c\Im m_{\mu_A\boxplus\mu_B}+\hat{\Lambda})$. Choosing $c$ to be sufficiently small and applying  the fact $|m_{\mu_A\boxplus\mu_B}|\gtrsim 1$, we get $|m_H|\gtrsim 1$ if $\hat{d}_i\neq 0$ for at least one $i$. 
In addition, we have 
\begin{align}
\Big|\frac{1}{N^2} \sum_{i=1}^N  \sum_k^{(i)} \hat{d}_i\mathbf{e}_k^*\wt{B}^{\la i\ra} G\mathbf{e}_i  \frac{\partial \Lambda_{\iota}}{\partial g_{ik}}\Big|\leq C\Psi^2\hat{\Pi}^2, \label{190309111}
\end{align}
which follows from the quantitative version of (\ref{17021303}).
 The same estimate holds if we replace $\Lambda_\iota$ by $\overline{\Lambda_\iota}$. 
Then by the simple fact $\partial |\Lambda_{\iota}|^2/\partial g_{ik}= \overline{\Lambda_{\iota}}\partial \Lambda_{\iota}/\partial g_{ik}+\Lambda_{\iota}\partial \overline{\Lambda_{\iota}}/\partial g_{ik}$, and the estimates (\ref{190309110}) and (\ref{190309111}), we see that the left side of (\ref{190309100}) is actually bounded by $C\Psi^4$, which is much smaller than $\hat{\Pi}$, under our choice of $\hat{\Lambda}$ in (\ref{190311200}).  The proof of (\ref{190309101}) is similar to that of (\ref{190309100}), we thus omit the details. 

At the end, we prove (\ref{190309102}). Recall from (\ref{17072590}) the fact $\mathbf{h}_j=\e{-\mathrm{i}\theta_j} \mathbf{u}_j=\e{-\mathrm{i}\theta_j}U \mathbf{e}_j$. Hence, we have 
\begin{align*}
|T_j|^2=|\mathbf{h}_j^*G\mathbf{e}_j|^2= (U^*G)_{jj}(G^*U)_{jj}=((U^{\la i\ra})^*R_iG)_{jj}(G^*R_iU^{\la i\ra})_{jj}
\end{align*}
for any $i,j$. 
Then we have 
\begin{align*}
 \frac{\partial (| T_{j}|^2+N^{-1})^{\frac12}}{\partial g_{ik}}= (| T_{j}|^2+N^{-1})^{-\frac12}\Big(\frac{\partial ((U^{\la i\ra})^*R_iG)_{jj}}{\partial g_{ik}}(G^*R_iU^{\la i\ra})_{jj}+\frac{\partial (G^*R_iU^{\la i\ra})_{jj}}{\partial g_{ik}}((U^{\la i\ra})^*R_iG)_{jj}\Big)
\end{align*}
Note that $|(G^*R_iU^{\la i\ra})_{jj}|=|T_j|$. In the sequel, we focus on the first term in the parenthesis above. The second term can be discussed similarly.  From the definition, it is elementary to derive 
\begin{align}
\frac{\partial R_i}{\partial g_{ik}}= -c_i \mathbf{e}_k(\mathbf{e}_i+\mathbf{h}_i)^*+ \Delta_R(i,k), \label{190309200}
\end{align}
where $c_i$ and  $\Delta_R(i,k)$ are  defined in (\ref{170725102}) and (\ref{170729100}), respectively. Applying  (\ref{17071801}) and (\ref{190309200}), we have 
\begin{align}
\frac{\partial ((U^{\la i\ra})^*R_iG)_{jj}}{\partial g_{ik}}=
& c_i\mathbf{e}_j^*(U^{\la i\ra})^*R_iG\mathbf{e}_k(\mathbf{e}_i+\mathbf{h}_i)^*\wt{B}^{\la i\ra}R_i G\mathbf{e}_j+c_i\mathbf{e}_j^*(U^{\la i\ra})^*R_iG R_i\wt{B}^{\la i\ra}\mathbf{e}_k(\mathbf{e}_i+\mathbf{h}_i)^* G\mathbf{e}_j\nonumber\\
&-c_i(U^{\la i\ra})^*_{jk}(\mathbf{e}_i+\mathbf{h}_i)^*G\mathbf{e}_j+\mathbf{e}_j^*(U^{\la i\ra})^* \Delta_R(i,k)G\mathbf{e}_j+ \mathbf{e}_j^*(U^{\la i\ra})^*R_i\Delta_G(i,k)\mathbf{e}_j. \label{190309205}
\end{align}
We take the first term on the right side of (\ref{190309205}) for example.  The contribution of this term to the left side of (\ref{190309102}) reads
\begin{align}
\Big(\frac{N^{4\varepsilon}}{\sqrt{N\eta}}\Big)^{-1} \frac{1}{N^3} \sum_{i,j=1}^N  \sum_k^{(i)} \hat{c}_{ij} \hat{d}_i\mathbf{e}_k^*\wt{B}^{\la i\ra} G\mathbf{e}_i  \mathbf{e}_j^*(U^{\la i\ra})^*R_iG\mathbf{e}_k(\mathbf{e}_i+\mathbf{h}_i)^*\wt{B}^{\la i\ra}R_i G\mathbf{e}_j, \label{190309210}
\end{align}
where we introduced the shorthand notation 
\begin{align*}
\hat{c}_{ij}:= c_i(| T_{j}|^2+N^{-1})^{-\frac12}(G^*R_iU^{\la i\ra})_{jj}. 
\end{align*}
Let $I^{(i)}$ be the identity matrix with $(i,i)$-th entry replaced by $0$ and let $C_i:=\text{diag}(\hat{c}_{i1}, \ldots, \hat{c}_{iN})$. We have 
\begin{align*}
\text{(\ref{190309210})}=\Big(\frac{N^{4\varepsilon}}{\sqrt{N\eta}}\Big)^{-1} \frac{1}{N^3} \sum_{i}^N  \hat{d}_i (\mathbf{e}_i+\mathbf{h}_i)^*\wt{B}^{\la i\ra}R_i GC_i(U^{\la i\ra})^*R_iGI^{(i)}\wt{B}^{\la i\ra} G\mathbf{e}_i.
\end{align*}
Similarly to (\ref{170726157}), we have 
\begin{align*}
\big|(\mathbf{e}_i+\mathbf{h}_i)^*\wt{B}^{\la i\ra}R_i GC_i(U^{\la i\ra})^*R_iGI^{(i)}\wt{B}^{\la i\ra} G\mathbf{e}_i\big|\leq C\frac{\Im G_{ii}+\Im \mathcal{G}_{ii}}{\eta^2}
\end{align*}
Therefore, we have 
\begin{align*}
\big|\text{(\ref{190309210})}\big|\leq \Big(\frac{N^{4\varepsilon}}{\sqrt{N\eta}}\Big)^{-1} \frac{\Im m_{\mu_A\boxplus\mu_B}+\hat{\Lambda}}{N^2\eta^2} \ll \hat{\Pi}. 
\end{align*}
The contributions from the other terms in (\ref{190309205}) can be estimated similarly. We thus omit the details. 

Except for the modifications listed above, the rest of the proof of (\ref{19030501}) is the same as that for (\ref{17071833}).

\end{document}